\documentclass[11pt]{article}
\usepackage{geometry}
\usepackage{amsmath}
\usepackage{amssymb}
\usepackage{amsthm}
\usepackage{amsfonts}
\usepackage[dvipsnames]{xcolor}
\usepackage{tikz}
\usepackage{mathtools}
\usepackage{pifont}
\usepackage{mathabx}
\usepackage{tabularx}
\usepackage{comment}
\usepackage{graphicx}
\usepackage{mathrsfs}

\DeclareMathOperator{\Stab}{Stab}
\DeclareMathOperator{\id}{id}
\DeclareMathOperator{\Isom}{Isom}

\DeclareMathOperator{\red}{red}
\DeclareMathOperator{\Sub}{Sub}
\DeclareMathOperator{\Aut}{Aut}

\DeclareMathOperator{\Cay}{Cay}

\DeclareMathOperator{\Sym}{Sym}

\newcommand{\la}{\langle}
\newcommand{\ra}{\rangle}
\newcommand{\con}{\square}
\newcommand{\twist}{\star}
\newcommand{\f}{\mathfrak{f}}
\newcommand{\g}{\mathfrak{g}}
\newcommand{\h}{\mathfrak{h}}

\newcommand{\aaa}{\mathfrak{a}}
\newcommand{\bbb}{\mathfrak{b}}
\newcommand{\ccc}{\mathfrak{c}}

\newcommand{\acting}{\curvearrowright}

\newtheorem{theorem}{Theorem}[section]
\newtheorem{lemma}[theorem]{Lemma}
\newtheorem{proposition}[theorem]{Proposition}
\newtheorem{corollary}[theorem]{Corollary}
\newtheorem{claim}{Claim}

\newtheorem{theoremintro}{Theorem}

\newtheorem{corollaryintro}[theoremintro]{Corollary}

\theoremstyle{definition}
\newtheorem{definition}[theorem]{Definition}
\newtheorem{example}[theorem]{Example}

\theoremstyle{remark}
\newtheorem{remark}[theorem]{Remark}
\newtheorem{question}[theorem]{Question}

\newtheorem{observation}{Observation}

\title{A guide to constructing free \\ transitive actions on median spaces}
\author{Pénélope Azuelos}
\date{\today}

\usepackage[bookmarks,bookmarksnumbered,pdfstartview=FitBH,backref=page]{hyperref}
\hypersetup{
	colorlinks=true, 
	linkcolor=black,
	citecolor= black,
        urlcolor=black,
	linktoc= all }

 \let\OLDthebibliography\thebibliography
\renewcommand\thebibliography[1]{
  \OLDthebibliography{#1}
  \setlength{\parskip}{0pt}
  \setlength{\itemsep}{2pt plus 0.3ex}
}

\newcommand*{\Lightning}[1][]{%
\text{
  \tikz[
    x=.55 * height("H"),  
    y=height("H"),  
    baseline=(current bounding box.south),
    line width=.02 * height("H"),
    line join=bevel,
  ]
  \filldraw[{#1}]
    (-.5, -.5) -- (.4, -.03) -- (-.1, .06) --
    (.5, .5) -- (-.4, .03) -- (.1, -.06) --
    cycle
    (-.5, 0)
    (.5, 0)
  ;%
}}
\newcommand{\reduced}{{\Lightning}}

\newtheorem*{the-thm}{Theorem \thethmid} 

\begin{document}

\maketitle

\begin{abstract}
    We construct large families of groups admitting free transitive actions on median spaces. In particular, we construct groups which act freely and transitively on the complete universal real tree with continuum valence such that any subgroup of the additive reals is realised as the stabiliser of an axis. We prove a more precise version of this, which implies that there are $2^{2^{\aleph_0}}$ pairwise non-isomorphic groups which admit a free transitive action on this real tree.
    We also construct free transitive actions on products of complete real trees such that any subgroup of $\mathbb{R}^n$ is realised as the stabiliser of a maximal flat, and an irreducible action on the product of two complete real trees. 
    
    To construct each of these groups, we introduce the notion of an \textit{ore}: a set equipped with the structure of a meet semilattice and a cancellative monoid with involution, which verifies some additional axioms. We show that one can \textit{extract} a group from an ore and equip this group with a left-invariant median structure.
\end{abstract}

\tableofcontents

\section{Introduction}

Median spaces are metric spaces which simultaneously generalise both $\mathbb{R}$-trees and CAT(0) cube complexes. 
We will give a number of constructions of free transitive actions on connected median spaces, revealing a great deal of variety among the groups which admit these actions. The study of groups acting freely and transitively on connected metric spaces is in part inspired by the highly fruitful idea that finitely generated groups can be studied via their actions on metric spaces, the most revealing of actions being free (or proper) and transitive (or cocompact). One of the most natural ways to obtain a transitive action on a median space is to take a finitely generated group $G$ which is coarsely median (see \cite{Bowditch-coarse-median}) with respect to some (equivalently any) proper word metric and to consider the action of an ultraproduct $G^*$ of $G$ on an asymptotic cone of $G$ (using the same ultrafilter). This space is bi-Lipschitz equivalent to a median space on which $G^*$ acts by isometries \cite{Bowditch-2018,Zeidler}. It is straightforward to see that such an action is always transitive, however it is never free. On the other hand, Casals-Ruiz, Hagen and Kazachkov have shown that many of these asymptotic cones can be equipped with free transitive actions \cite{CRHK}; their construction involves a precise understanding of the combinatorial structure of the space in question, independently of the fact it arises as an asymptotic cone. In the present paper, we will mostly restrict ourselves to actions on real trees and their products, but, even in this more restricted setting, the actions we construct exhibit entirely new behaviours.

\paragraph{$\mathbb{R}$-trees.}
An $\mathbb{R}$-tree is a geodesic metric space where any pair of points is connected by exactly one simple path, or equivalently, it is a connected median space of rank 1. 
Finitely generated groups which act on $\mathbb{R}$-trees have been extensively studied, and this study has been extremely consequential, particularly for our understanding of hyperbolic groups (see e.g. \cite{Bestvina-Feighn,Morgan-Shalen1984,Rips-Sela,Sela1995,Sela2009}).
The class of finitely generated groups which admit free actions on $\mathbb{R}$-trees includes free groups (which act on their Cayley graphs), free abelian groups (which embed as subgroups of $\mathbb{R}$) and, as was shown by Morgan--Shalen in \cite{Morgan-Shalen1991}, the fundamental group of any closed surface $\Sigma$, unless $\Sigma$ is non-orientable with Euler characteristic $\geq -1$. One can show that any free product of groups which admit free actions on $\mathbb{R}$-trees admits a free action on an $\mathbb{R}$-tree. Conversely, Rips' theorem \cite{GLP} states that any finitely generated group which admits a free action on an $\mathbb{R}$-tree splits as a free product of surface and free abelian groups. Dunwoody \cite{Dunwoody97} and Zastrow \cite{Zastrow98} produced examples which show that this characterisation fails when one drops the finite generation assumption. 
 
 More recently, Berestovki\u{\i}--Plaut \cite{Berestovskii-Plaut} produced a large family of groups which act freely on real trees: they show that every length space $X$ is the quotient of a real tree $T_X$ by the free action of some group $G_X$, where the group $G_X$ can be interpreted as a ``refined fundamental group" of $X$. Every finitely generated subgroup of $G_X$ is free, but they produce examples of spaces $X$ where $G_X$ does not split as a free product of free and surface groups. Their construction can also be used to construct a free transitive action on the tree $T_X$, provided the space $X$ is itself equipped with a free transitive action (see the proof of Theorem~38.24 on page 211 of \cite{CRHK}).

Recall that, for any cardinal $\kappa$ which is not both finite and odd, there is exactly one group which acts freely, transitively and without edge inversions on the regular simplicial tree with valence $\kappa$: namely the free group of rank $\kappa/2$. Given an $\mathbb{R}$-tree $T$ and a point $x \in T$, the \textit{set of directions} of $T$ at $x$ is the set of connected components of $T - \{x\}$.
The \textit{valence} of $T$ at $x$ is the cardinality of the set of directions at $x$. Given a cardinal $\kappa \geq 2$, there is a unique complete $\mathbb{R}$-tree up to isometry such that each point has valence $\kappa$ \cite{Mayer-Nikiel-Oversteegen,Nikiel}, called the \textit{universal real tree with valence} $\kappa$. Unlike in the discrete case, this tree does not always admit a unique group structure. This was proven by Casals-Ruiz, Hagen and Kazachkov in \cite[Section~38]{CRHK}, where they construct, for each $ 2 \leq \kappa \leq 2^{\aleph_0}$, a group $G$ which acts freely and transitively on the universal real tree $T$ with valence $2^{\aleph_0}$ such that there are exactly $\kappa$ conjugacy classes of maximal abelian subgroups $H \leq G$ such that $H \cong \mathbb{R}$. More precisely, there are exactly $\kappa$ $G$-orbits of lines $L \subseteq T$ such that the stabiliser $\Stab_G(L)$ acts transitively on $L$. They also show that any line in $T$ either has transitive,  cyclic or trivial stabiliser (see Proposition~38.28 in loc. sit.). This leads to the natural question of whether any other subgroups of $\mathbb{R}$ can be realised as the stabiliser of a line in a free transitive action on a real tree. Our first result provides a positive answer to this question.

Let $\Sub_{NC}(\mathbb{R})$ denote the set of non-cyclic subgroups of $\mathbb{R}$ and let $\mathcal{K}$ denote the set of cardinals $\kappa$ such that $\kappa \leq 2^{\aleph_0}$.

\begin{theoremintro} [Theorem~\ref{thm: centraliser spectrum}] \label{thmintro A}
    Let $\iota: \Sub_{NC}(\mathbb{R}) \rightarrow \mathcal{K}$ be any map which is supported on $\leq 2^{\aleph_0}$ elements of $\Sub_{NC}(\mathbb{R})$. Then there exists a group $G$ and a free transitive action of $G$ on the universal real tree $T$ with valence $2^{\aleph_0}$ such that the following holds.
    For each $H \leq \mathbb{R}$, let $A_H$ be the set of orbits $G \cdot L$ such that $L \subseteq T$ is a line and the induced action of $\Stab_G(L)$ on $L$ is isomorphic to the action of $H$ on $\mathbb{R}$ by translations. If $H \leq \mathbb{R}$ is non-cyclic then $|A_H| = \iota(H)$.
\end{theoremintro}

An easy consequence of this is that the set of groups which admit free transitive actions on $T$ is not only infinite, it is as large as possible. Recall that $T$ is the asymptotic cone of any non-elementary hyperbolic group \cite{Dyubina-Polterovich} and, since these are countable, it follows that the cardinality of $T$ is $2^{\aleph_0}$. Thus there are at most $2^{2^{\aleph_0}}$ possible group operations on $T$.

\begin{corollaryintro}
    Let $T$ be the universal real tree with valence $2^{\aleph_0}$. Then there are $2^{2^{\aleph_0}}$ pairwise non-isomorphic groups which admit a free transitive action on $T$.
\end{corollaryintro}
\begin{proof}
    Let $\mathcal{A}$ be a set of pairwise non-isomorphic elements of $\Sub_{NC}(\mathbb{R})$ with cardinality\footnote{One way of constructing such a family of groups is given by Yves Cornulier here: \url{https://mathoverflow.net/questions/264438/number-of-torsion-free-abelian-groups}.} $|\mathcal{A}| = 2^{2^{\aleph_0}}$, and for each $H \in \mathcal{A}$ let $\chi_H: \Sub_{NC}(\mathbb{R}) \rightarrow \mathcal{K}$ be the characteristic map of $H$. Let $G_H$ be a group satisfying the conclusion of Theorem A with $\iota = \chi_H$. Recall that every maximal abelian subgroup of $G_H$ is the stabiliser of a line in $T$. Therefore, for each $H,K \in \mathcal{A}$ with $H \ncong K$, we have $G_H \ncong G_K$ since $G_H$ has a maximal abelian subgroup isomorphic to $H$ and $G_K$ does not. Conversely, if $G$ is a group acting freely and transitively on $T$ then $|G| = |T| = 2^{\aleph_0}$ and there are at most $2^{2^{\aleph_0}}$ groups with that cardinality.
\end{proof}

The construction of the group $G$ in Theorem~\ref{thmintro A} is inspired by the constructions of universal real trees as function spaces given by Dyubina--Polterovich in \cite{Dyubina-Polterovich}.
In the case of infinite cardinals, their construction can be described as follows.
Fix an infinite cardinal $\mu$ and a set $C_\mu$ of cardinality $\mu$. The universal real tree $A_\mu$ is given as the set of maps $f: [0, \ell_f) \rightarrow C_\mu$ where $\ell_f \geq 0$ and $f$ is \textit{piecewise constant from the right}, meaning that for all $t \in [0,\ell_f)$ there exists $\varepsilon > 0$ such that $f|_{[t, t+ \varepsilon]}$ is constant. If $\ell_f = 0$ then $f $ is the empty map. Given $f,g \in A_\mu$, the distance from $f$ to $g$ is defined by $d(f,g) \coloneqq (\ell_f - s) + (\ell_g - s)$, where $s \coloneqq \sup\{t \in [0, \min\{\ell_f, \ell_g\}) : f|_{[0,t]} = g|_{[0,t]}\}$. The fact that the elements of $A_\mu$ are piecewise constant from the right ensures that the set of directions at any point of $A_\mu$ is in bijection with $C_\mu$. 

The underlying idea for the construction of the group $G$ is to construct a function space $T$ as above but where the set of ``directions" is equipped with an involution and an action of $\mathbb{R}$ which are used to define a group operation on the space, in such a way that the choice of $\mathbb{R}$-action determines the possible axis stabilisers of $T$. The elements of $A_\mu$ in the Dyubina--Polterovich construction are not well suited to this, as the piecewise constant from the right condition, and the fact that the domains are half open intervals, make it difficult to define appropriate inverses of elements. 

Instead, the group $G$ will be a subgroup of a group $T_X$ whose construction we now briefly outline. We start with a set $X$ equipped with an action of $\Isom(\mathbb{R}) = \mathbb{R} \rtimes \la * \ra$, where $\la * \ra \cong \mathbb{Z} / 2\mathbb{Z}$. Two maps $f,g: [0, \ell] \rightarrow X$ are said to be \textit{equivalent} if they agree on all but countably many elements in $[0,\ell]$.
The elements of $T_X$ are equivalence classes of maps $f: [0,\ell_f] \rightarrow X$ satisfying an admissibility criterion. We will define a group operation $\star$ on $T_X$, where the product $\f \star \g$ of two elements $\f,\g \in T_X$ depends on $\ell_f$. The metric on $T_X$ can be defined similarly to the metric on $A_\mu$ to make $T_X$ into an $\mathbb{R}$-tree. Given $x \in X$ such that $x^* \notin \mathbb{R} \cdot x$, there is a line in $T_X$ given by $\mathbf{L}_x \coloneqq \{\f \in T_X: f(t) = x \; \forall t$ or $f(t) = x^* \: \forall t\}$ which has stabiliser isomorphic to $\Stab_\mathbb{R}(x)$. Because we allow so many maps in $T_X$, the set of directions at each point of $T_X$ is much bigger than $X$ itself; it turns out that the valence of $T_X$ is $\geq 2^{2^{\aleph_0}}$ as long as $|X| \geq 2$. Moreover, it may not be possible to control the exact set of orbits of axes in $T_X$ with non-cyclic stabilisers (see Remark~\ref{rem: weird axes}). The right object to consider is instead the smallest subgroup $G \leq T_X$ which is closed with respect to the topology induced by the metric on $T_X$, and which contains all the axes $\mathbf{L}_x$. The construction of $T_X$ is given in Section~\ref{sec: TX construction} using tools developed in Section~\ref{sec: ores} (we will address these later in this introduction).
We will prove in Section~\ref{sec: templates} that $G$ is indeed the universal complete real tree with valence $2^{\aleph_0}$ as long as $2 \leq |X| \leq 2^{\aleph_0}$ and, in Section~\ref{sec: centraliser spectrum}, we will prove Theorem~\ref{thmintro A} using this construction.

We can also use the above ideas to construct groups which act freely and transitively on real trees with valence $3 \leq \kappa < 2^{\aleph_0}$ but these only exist for incomplete real trees:

\begin{theoremintro} [Theorem~\ref{thm: small valence action}]
    Let $3 \leq \kappa < 2^{\aleph_0}$ be a cardinal. There are no free transitive actions on the complete universal $\mathbb{R}$-tree $T_\kappa$ with valence $\kappa$. 

    Let $\kappa \geq 3$ be any cardinal.
    There exists a free transitive action $G \acting S_\kappa$, where $G$ is a group and $S_\kappa$ is an incomplete $\mathbb{R}$-tree with valence $\kappa$, if an only if $\kappa$ is either infinite or even. 
    If $\kappa$ is finite and even, then this action is unique.
\end{theoremintro}

In \cite{Chiswell-Muller}, Chiswell--M\"{u}ller show that the free product of $\kappa$ copies of $\mathbb{R}$ acts freely and transitively on an $\mathbb{R}$-tree. Although it is not explicitly mentioned, it is not hard to see from their proof that this real tree is incomplete and has valence $2\kappa$. In fact, the action we construct is isometric to theirs (see Remark~\ref{rem: isomorphic actions}). If one applies their construction to the free product of $2^{\aleph_0}$ copies of $\mathbb{R}$, one obtains the free transitive action on Uryson's $\mathbb{R}$-tree constructed by Berestovski\u{\i}  \cite{Berestovskii1989,Berestovskii2019}.

Another interesting source of free transitive actions on incomplete $\mathbb{R}$-trees is the following result of Chiswell--M\"{u}ller \cite[Theorem~5.4]{Chiswell-Muller2010}: \textit{Any free action on an $\mathbb{R}$-tree is contained in a free transitive action on an $\mathbb{R}$-tree}. More precisely, given a free action $G \acting T$, where $T$ is an $\mathbb{R}$-tree, there exists a group $\widehat{G}$ and an $\mathbb{R}$-tree $\widehat{T}$, such that $G \leq \widehat{G}$ and there is a $G$-equivariant isometric embedding $T \hookrightarrow \widehat{T}$. If one starts with an action $G \acting T$ which is not already transitive then the tree $\widehat{T}$ will not be complete (see Proposition~6.2 in loc. sit.).

\smallskip

In Section~\ref{sec: Lambda tree groups}, we construct free transitive actions on $\Lambda$-trees for any totally ordered abelian group $\Lambda$ (see Definition~\ref{def: Lambda tree}). The study of group actions on $\Lambda$-trees was initiated by Morgan--Shalen in \cite{Morgan-Shalen1984}. Due to the relationship between free actions on $\Lambda$-trees and algorithmic properties of groups, there has been significant interest in producing such actions. Free actions of Lyndon's free $\mathbb{Z}[t]$-group $F^{\mathbb{Z}[t]}$ (and other groups of infinite words over discretely ordered groups) were constructed in \cite{MRS}. In \cite{KMS2}, the authors give a natural folding construction associating free actions on trees to groups of infinite words, and establish a universal embedding property for the resulting trees.

\paragraph{Actions on products of trees.}

Let $T_1, T_2$ be regular, locally finite simplicial trees with even degrees $d_1,d_2 \geq 4$. Unlike the case with a single factor, there are several groups which act freely, transitively and without edge inversions on the vertex set of the product $T_1 \times T_2$. The groups which admit these actions -- called BMW groups, after Burger, Moses and Wise -- can have extremely varied structures, ranging from direct products of free groups to virtually simple groups. Indeed, a great deal of the interest in these groups stems from the discoveries by the aforementioned authors of classes of examples answering longstanding questions in geometric group theory. Wise used them to construct the first example of a group acting properly discontinuously and cocompactly on a CAT(0) space which is not residually finite \cite{Wise-thesis} (see also \cite{Wise-CSC}) and Burger--Moses used them to construct finitely presented simple groups of the form $F_n \ast_{E} F_m$ where $F_n,F_m,E$ are free groups with finite rank and the embedding of $E$ in both $F_n$ and $F_m$ has finite index \cite{Burger-Moses}. A survey on this topic can be found in \cite{Caprace-finite&infinite}. 

Before one can fathom the existence of a simple group acting freely and cocompactly on a product of trees, one must first come to terms with the existence of a group which acts freely and cocompactly on such a space without virtually splitting as a direct product.
A BMW group $G$ is called \textit{reducible} if it contains a finite index subgroup which splits non-trivially as a direct product and \textit{irreducible} otherwise. Examples of irreducible BMW groups include those constructed by Burger--Moses and Wise, as well as some constructed in \cite{Radu,Rattaggi,Rungtanapirom,Stix-Vdovina}.

Using the construction from Section~\ref{sec: higher rank}, we can show that this phenomenon persists in the continuous setting:

\begin{corollaryintro}[Corollary~\ref{cor: irreducible product}] \label{cor: irreducible product intro}
    There exists a group $G$ which admits a free transitive action by isometries on a product $T_1 \times T_2$ of two complete $\mathbb{R}$-trees with valence $2^{\aleph_0}$ such that, for any subgroup $H \leq G$ which splits non-trivially as a direct product, the induced action $H \acting T_1 \times T_2$ is not cobounded.
\end{corollaryintro}

Although the group $G$ constructed in the proof or the above corollary does not contain any subgroups which split as direct products and act coboundedly on $T_1 \times T_2$, it does have some rather large proper normal subgroups (see the proof of Corollary~\ref{cor: irreducible product}). Therefore the following question remains mysterious:

\begin{question}
    Does there exist a simple group which acts freely and coboundedly on a product of (complete) $\mathbb{R}$-trees?
\end{question}

By Proposition~\ref{prop: no BMW subgroups}, the group constructed to prove Corollary~\ref{cor: irreducible product intro} does not contain any isometrically embedded irreducible BMW groups. By taking direct unions of certain BMW groups, we can construct groups which act on products of trees with dense orbits and contain isometrically embedded irreducible BMW groups. See Section~\ref{sec: BMW prelims} for the definition of a positive BMW presentation. The group from \cite[Example~4.1]{Wise-CSC} is an example of an irreducible BMW group with such a presentation.

\begin{theoremintro}[Theorem~\ref{thm: BMW metric groups}] \label{thmintro: BMW}
    Let $H$ be a BMW group with a positive BMW presentation $\la A \cup X \; | \; R \ra$ and let $\Cay(H, A \cup X)$ be the corresponding Cayley graph.
    There exists a group $G$ such that the following hold:
    \begin{itemize} 
        \item[i.] there is an injective homomorphism $H \hookrightarrow G$;
        \item[ii.] $G$ acts freely with dense orbits on the $\ell^1$ product of two $\mathbb{R}$-trees $T_1 \times T_2$;
        \item[iii.] there is an isometric embedding $\psi: \Cay(H,A \cup X) \hookrightarrow T_1 \times T_2$ which is equivariant relative to $H \hookrightarrow G$.
    \end{itemize}
    If $H$ is irreducible, then for any subgroup $L \leq G$ which splits non-trivially as a direct product, the induced action of $L$ on $T_1 \times T_2$ does not have dense orbits.
\end{theoremintro}

This leads to the following question:

\begin{question} \label{Q: BMW embedding}
    For which BMW groups $H$, with BMW presentation $\la A \cup X | R \ra$, does there exist a group $G$ which acts freely and transitively on a product of complete $\mathbb{R}$-trees $T_1 \times T_2$ such that $H \leq G$ and there is an equivariant isometric embedding $\Cay(H,A \cup X) \hookrightarrow T_1 \times T_2$?
\end{question}

The construction used to produce the group in Corollary~\ref{cor: irreducible product intro} can in fact be used to construct free transitive actions on products of arbitrarily many real trees. The flexibility of this construction is illustrated in the theorem below.

Let $N \in \{\mathbb{N}\} \cup \{\{1, \dots, n\}: n \in \mathbb{N}\}$ and let $\mathbf{R} \coloneqq \ell^1(N)$ be equipped with its natural additive group structure and the $\ell^1$ norm. Let $\Sub_{D}(\mathbf{R})$ denote the set of dense subgroups of $\mathbf{R}$ and let $\overline{\Sub}_D(\mathbf{R})$ be the quotient of $\Sub_D(\mathbf{R})$ under linear isometries of $\mathbf{R}$. Let $\mathbf{T}$ be the $\ell^1$ product of $|N|$ copies of the complete universal real tree $T$ with valence $2^{\aleph_0}$.

\begin{theoremintro}[Theorem~\ref{thm: arbitrary flats}] \label{thm: arbitrary flats intro}
    Let $\iota: \overline{\Sub}_{D}(\mathbf{R}) \rightarrow \mathcal{K}$ be any map which is supported on $\leq 2^{\aleph_0}$ elements of $\overline{\Sub}_{D}(\mathbf{R})$.
    Then there exists a group $G$, which acts freely and transitively on $\mathbf{T}$, such that: for each $[H] \in \overline{\Sub}_{D}(\mathbf{R})$, the cardinality of the set of orbits of maximal flats $F \subseteq \mathbf{T}$ such that $\Stab_G(F) \acting F$ is isomorphic to $H \acting \mathbf{R}$ is $\iota([H])$.
\end{theoremintro}

\paragraph{A strategy for constructing free transitive actions on median spaces.}

Each of the existence results mentioned so far involves a different construction of a group acting freely and transitively (or with dense orbits) on a median space. But, in every case involving a transitive action, the proof that what we obtain from the construction really is a group with the required properties follows a similar strategy. Namely, we first define a set $Y$ which we equip with a binary operation $\con$, an involution $-1$ and a relation $\preceq$, and we distinguish an element $\id \in Y$. We then show that $(Y,\con)$ is a cancellative monoid with identity $\id$ and involution $-1$ and $(Y, \preceq)$ is a partially ordered set with minimal element $\id$ which satisfies some extra conditions making it a \textit{median semi-lattice} (see Definition~\ref{def: median semilattice}). These statements, together with a few more technical properties, imply that $(Y, \preceq, \con, \id, -1)$ is an algebraic object called an \textit{ore}\footnote{This is unrelated to the work of \O ystein Ore. The name is in reference to naturally occurring \href{https://en.wikipedia.org/wiki/Ore}{ores} from which one extracts metals.} (Definition~\ref{defn: ore}). This object admits a canonical subset $G \subseteq Y$ of \textit{admissible} elements on which we define a new binary operation $\twist$. We then apply the following result from Section~\ref{sec: ores}:

\begin{theoremintro}[Theorem~\ref{thm: extracted group}, Proposition~\ref{prop: median metric}]
    Let $(Y, \preceq, \con, \id, -1)$ be an ore and $G \subseteq Y$ be the set of admissible elements of $Y$. Then $(G, \twist)$ is a group.

    If $\Lambda$ is a totally ordered abelian group and $\ell:Y \rightarrow \Lambda$ is a length function, define $d: G \times G \rightarrow \Lambda$ by $d(f,g) = \ell(f^{-1} \twist g)$. Then $d$ is a $\Lambda$-metric which is invariant under left multiplication by $G$ and the resulting $\Lambda$-metric space $(G,d)$ is median.
\end{theoremintro}

In the above theorem a \textit{length function} is defined as a map $\ell:Y \rightarrow \Lambda$ such that, for all $g \in Y$, we have $\ell(g) = \ell(g^{-1}) \geq 0$ with equality if and only if $g = \id$, and $\ell(g \con h) = \ell(g) + \ell(h)$ for all $f,g \in Y$ (Definition~\ref{def: length function}).

\begin{remark}
    All of the actions we construct in this paper are on spaces which are not locally compact. It turns out this is to be expected when looking for interesting free transitive actions on median spaces. Indeed, Messaci proved in \cite{Messaci} that any finite rank locally compact connected median space which admits a transitive action is isometric to $\mathbb{R}^n$ equipped with its $\ell^1$ metric, for some $n$.
\end{remark}

\paragraph{Acknowledgements.}

 I am indebted to Mark Hagen for telling me about Berestovski\u{\i}--Plaut's construction and his, Casals-Ruiz and Kazachkov's work on $\mathbb{R}$-cubings and groups of (proper) headings, which was the starting point of this project. Thanks also for uncountably many discussions since then and for carefully reading many versions of this text. I am grateful to Montse Casals-Ruiz and Ilya Kazachkov for a number of enlightening conversations, Indira Chatterji for many interesting discussions and to the members of the Laboratoire J.A. Dieudonné for their hospitality during a research stay during which some of this work was carried out. This work was funded by a University of Bristol PhD scholarship and a grant from the Académie Syst\`emes Complexes of the Université C\^{o}te d'Azur.
 
\section{Preliminaries}

\subsection{$\Lambda$-metrics} \label{sec: Lambda spaces}

We will at times need to work with generalised metric spaces, where metrics take values in some totally ordered abelian group. We record some definitions and facts about these spaces here; for more details the reader can refer to \cite{Chiswell-intro_to_lambda_trees,GKMS}.

Let $\Lambda$ be an abelian group, which we denote additively. We say that $\Lambda$ is \textit{totally ordered} if there is a total order $\leq$ on $\Lambda$ which is $\Lambda$-invariant (i.e. $x \leq y \Leftrightarrow \lambda + x \leq \lambda + y$ for all $\lambda, x,y \in \Lambda$). It follows immediately from the definition that such a group is torsion-free. Fix a totally ordered abelian group $\Lambda$ for the rest of this section. Note that we can define intervals $[\lambda_1, \lambda_2], (\lambda_1, \lambda_2), (\lambda_1, \lambda_2], [\lambda_1, \lambda_2)$ in exactly the same way as they are defined in $\mathbb{R}$.

\begin{definition}
    Let $X$ be a set. A $\Lambda$-metric on $X$ is a map $d: X \times X \rightarrow \Lambda$ such that, for all $x,y,z \in X$,
    \begin{itemize}
        \item $d(x,y) \geq 0$, with equality if and only if $x=y$;
        \item $d(x,y) = d(y,x)$;
        \item $d(x,y) \leq d(x,z) + d(z,y)$.
    \end{itemize}
    The pair $(X,d)$ is called a $\Lambda$-metric space.
\end{definition}

\begin{example}
\begin{itemize}
    \item The group $\Lambda$ is itself a $\Lambda$-metric space, with $\Lambda$-metric given by $d(x,y) = |x-y| = x-y$ if $x \geq y$ and $y-x$ otherwise, for all $x,y \in \Lambda$.
    \item A metric space in the usual sense is an $\mathbb{R}$-metric space.
    \item The 0-skeleton of a connected graph equipped with the path metric is a $\mathbb{Z}$-metric space. 
    \item If $X$ is a metric space, $\omega$ is an ultrafilter on $\mathbb{N}$ and $X^\omega, \mathbb{R}^\omega$ are the ultrapowers of $X, \mathbb{R}$ with respect to $\omega$, then $\mathbb{R}^\omega$ has a natural group structure and order making it a totally ordered abelian group and $X^\omega$ is an $\mathbb{R}^\omega$-metric space.
\end{itemize}
\end{example}

If $\Lambda = \mathbb{R}$, we will usually write metric rather than $\mathbb{R}$-metric. 

\begin{definition}
    Given two $\Lambda$-metric spaces $(X,d_X), (Y,d_Y)$, an \textit{isometric embedding} $\varphi: X \hookrightarrow Y$ is a map such that $d_Y(\varphi(x), \varphi(y)) = d_X(x,y)$ for all $x,y \in X$. If $\varphi$ is surjective then it is an \textit{isometry}. Let $\Isom(X)$ denote the group of all isometries of $X$.

    A \textit{geodesic} in $X$ is an isometric embedding $[0,\lambda] \hookrightarrow X$ for some $\lambda \in \Lambda$ with $\lambda \geq 0$. The image of a geodesic is called a \textit{segment}. A $\Lambda$-metric space is called \textit{geodesic} if every pair of points is connected by a geodesic.
\end{definition}

\subsection{Median spaces}
\begin{definition}
    A $\Lambda$-metric space $(X,d)$ is \textit{median} if there exists a unique map $m: X^3 \rightarrow X$ such that for all $x_1, x_2, x_3 \in X$, if $i \neq j$ then
    \[ d(x_i,x_j) = d(x_i, m(x_1,x_2,x_3)) + d(m(x_1,x_2,x_3), x_j).\]
    The element $m(x_1, x_2, x_3) \in X$ is called the \textit{median} of $x_1, x_2, x_3$.
\end{definition}

Median $\Lambda$-metric spaces are examples of median algebras:

\begin{definition}
    A \textit{median algebra} is a set $X$ equipped with a symmetric ternary operation $m: X^3 \rightarrow X$ such that, for all $a,b,c,d \in X$, we have
    \begin{itemize}
        \item $m(a,a,b) = a$ and
        \item $m(m(a,b,d),c,d) = m(m(a,c,d),b,d)$.
    \end{itemize}
    A subset $Y \subseteq X$ is a \textit{median subalgebra} if $m(a,b,c) \in Y$ for all $a,b,c \in Y$.
\end{definition}

\begin{definition}
    Given $k \in \mathbb{N}$, a $k$\textit{-cube} is a median algebra of the form $\sigma_k = \{0,1\}^k$ where $m(x,y,z) = m \in \sigma_k$ such that $m$ agrees with at least two elements of $\{x,y,z\}$ on each coordinate. 

    The \textit{rank} of a median $\Lambda$-metric space $(X,d)$ is the supremum over all $k \in \mathbb{N}$ such that there is a $k$-cube $\sigma_k$ which embeds in $X$ as a median subalgebra.
\end{definition}

The following lemma is proven in \cite[Lemma~13.1.1]{Bowditch} in the case where $\Lambda = \mathbb{R}$, but the same proof applies here.

\begin{lemma} \label{lem: algebra to space}
    Let $(X,m)$ be a median algebra and let $d: X \times X \rightarrow \Lambda$ be a $\Lambda$-metric such that: for all $a,b,c \in X$ such that $m(a,b,c) = c$, we have $d(a,b) = d(a,c) + d(c,b)$. Then $(X,d)$ is a median $\Lambda$-metric space.
\end{lemma}

The notion of a $\Lambda$-tree was first introduced by Morgan--Shalen in \cite{Morgan-Shalen1984}. In the case where $\Lambda = \mathbb{R}$, it is equivalent to the older notion of an $\mathbb{R}$-tree first defined by Tits \cite{Tits}.

\begin{definition} \label{def: Lambda tree}
    A $\Lambda$-tree is a geodesic $\Lambda$-metric space such that
    \begin{itemize}
        \item if two segments intersect at a single point, which is an endpoint of both, then their union is a segment;
        \item the intersection of two segments with a common endpoint is itself a segment.
    \end{itemize}
\end{definition}

An immediate consequence of this definition is that, if $T$ is a $\Lambda$-tree, then there is a unique segment connecting any pair of points $x_1, x_2 \in T$. We denote this segment by $[x_1, x_2]$.

As in the real case, this notion can be characterised in terms of Gromov hyperbolicity. The notion of Gromov hyperbolicity was extended to $\Lambda$-metric spaces by Chiswell in \cite{Chiswell-intro_to_lambda_trees} as follows.

Let $\Lambda_\mathbb{Q} \coloneqq \mathbb{Q} \otimes_\mathbb{Z} \Lambda$, where $\Lambda$ and $\mathbb{Q}$ are viewed as $\mathbb{Z}$-modules. The elements of $\Lambda_\mathbb{Q}$ can be viewed as equivalence classes of the equivalence relation $\sim$ on $\Lambda \times (\mathbb{Z} - \{0\})$ given by $(\lambda, m) \sim (\mu, n)$ if and only if $m\lambda = n\mu$. The equivalence class of $(\lambda,m)$ is denoted by $\frac{\lambda}{m}$. We equip $\Lambda_\mathbb{Q}$ with the structure of a totally ordered abelian group with operation given by $\frac{\lambda}{m} + \frac{\mu}{n} = \frac{n \lambda + m \mu}{nm}$ and order given by $\frac{\lambda}{m} > 0$ if and only if $m \lambda > 0$. The map $\Lambda \rightarrow \Lambda_\mathbb{Q}$ given by $\lambda \mapsto \frac{\lambda}{1}$ is an injective homomorphism. We identify $\Lambda$ with its image in $\Lambda_\mathbb{Q}$.

Let $(X,d)$ be a $\Lambda$-metric space and $p,x,y \in X$. The \textit{Gromov product} $(x,y)_p \in \Lambda_\mathbb{Q}$ is given by:
\[
    (x,y)_p \coloneqq \frac{1}{2}(d(x,p) + d(y,p) - d(x,y)).
\]

\begin{definition}
    Let $\delta \geq 0$ be a constant. A $\Lambda$-metric space $(X,d)$ is $\delta$\textit{-hyperbolic with respect to} $p \in X$ if for all $x,y,z \in X$ we have 
    \[
        (x,y)_p \geq \min\{(x,z)_p, (y,z)_p\} - \delta.
    \]
    The space $(X,d)$ is $\delta$\textit{-hyperbolic} if it is $\delta$-hyperbolic with respect to every $p \in X$.
\end{definition}

If $(X,d)$ is $\delta$-hyperbolic with respect to $p_1 \in X$ then it is $2\delta$-hyperbolic with respect to any $p_2 \in X$ \cite[Lemma~1.2.5]{Chiswell-intro_to_lambda_trees}. In particular, $X$ is 0-hyperbolic if and only if $X$ is 0-hyperbolic with respect to some $p \in X$. 

\begin{lemma}[Chiswell, {\cite[Lemmas~2.1.6 and 2.4.3]{Chiswell-intro_to_lambda_trees}}] \label{lem: 0-hyp characterisation}
    Let $(X,d)$ be a geodesic $\Lambda$-metric space. Then $X$ is a $\Lambda$-tree if and only if, for some (equivalently any) $p \in X$, the following hold:
    \begin{itemize}
        \item[(i)] $X$ is 0-hyperbolic with respect to $p$;
        \item[(ii)] for all $x,y \in X$ we have $(x,y)_p \in \Lambda$. 
    \end{itemize}
\end{lemma}

A slightly stronger version of the following characterisation was given by Bowditch \cite[Lemma~15.1.2]{Bowditch} in the case where $\Lambda = \mathbb{R}$ (in that case, it is enough to assume that $X$ is path-connected rather than geodesic).

\begin{lemma} \label{lem: tree characterisation}
    Let $(X,d)$ be a geodesic $\Lambda$-metric space. Then $X$ is a $\Lambda$-tree if and only if $X$ is a rank 1 median $\Lambda$-metric space.
\end{lemma}
\begin{proof}
    Suppose that $X$ is a $\Lambda$-tree. Let $x_1, x_2, x_3 \in X$ and let $m \in X$ be the unique point such that $[x_1,x_2] \cap [x_1, x_3] = [x_1,m]$. Then $[x_2,m] \cup [m,x_3]$ is a segment so it follows that $m$ is a median of $x_1,x_2,x_3$. Uniqueness of $m$ follows from the uniqueness of the segments $[x_i,x_j]$. If the rank of $X$ is $\geq 2$ then there exists an isometric embedding $\varphi: \sigma_2 \rightarrow X$, where $\sigma_2$ is a 2-cube. For each $(\varepsilon_1, \varepsilon_2) \in \sigma_2$, let $x_{\varepsilon_1,\varepsilon_2} \coloneqq \varphi((\varepsilon_1, \varepsilon_2))$. Then $[x_{0,0}, x_{0,1}] \cup [x_{0,1}, x_{1,1}]$ is a segment and $[x_{0,0}, x_{1,0}] \cup [x_{1,0}, x_{1,1}]$ is a segment but their intersection is $\{x_{0,0}, x_{1,1}\}$ which is not a segment so this is a contradiction. Thus $X$ has rank~1.

    Conversely, suppose that $X$ is a geodesic rank 1 median $\Lambda$-metric space, and let $p,x,y,z \in X$. If $(x,y)_p < \min\{(x,z)_p, (y,z)_p\}$, then let $x_{0,0} \coloneqq m(p,x,y), \; x_{1,0} \coloneqq m(p,x,z)$, $x_{1,1} \coloneqq m(x,y,z), \; x_{0,1} \coloneqq m(p,y,z)$. Upon observing that, for any $a,b,c \in X$ we have $(a,b)_c = d(c, m(a,b,c))$, it is not hard to show that $\{x_{0,0}, x_{0,1}, x_{1,0}, x_{1,1}\}$ are pairwise distinct. It follows that $\{x_{0,0}, x_{0,1}, x_{1,0}, x_{1,1}\}$ is a 2-cube embedded in $X$ as a median subalgebra. This contradicts the assumption that $X$ has rank 1 so $X$ must be 0-hyperbolic. Lastly, if $x,y,p \in X$ then $(x,y)_p = d(p, m(x,y,p)) \in \Lambda$ so, by Lemma~\ref{lem: 0-hyp characterisation}, $X$ is a $\Lambda$-tree.
\end{proof}

\begin{definition}
    Let $(X,d)$ be a $\Lambda$-tree and $x \in X$. The \textit{valence of $x$} is the cardinality of the set of geodesic-connected components of $X - \{x\}$. Given a cardinal $\kappa$, we say that $X$ \textit{has valence $\kappa$} if every every point in $X$ has valence $\kappa$.
\end{definition}

For any cardinal $\kappa$, there is a unique complete $\mathbb{R}$-tree with valence $\kappa$ \cite{Mayer-Nikiel-Oversteegen,Nikiel}.

\subsection{Products}

The usual ways of constructing product metric spaces apply to $\Lambda$-metric spaces. Of particular interest to us is the $\ell^1$ product: Given $N \in \{\{1, \dots, n\} : n \in \mathbb{N}\} \cup \{\mathbb{N}\}$ and a family of $\Lambda$-metric spaces $\mathcal{X} \coloneqq ((X_n, d_n))_{n \in N}$, the \textit{$\ell_1$ product of $\mathcal{X}$ based at $z = (z_n)_{n \in N} \in \prod_{n \in N} X_n$} is the set 
\[
\ell^1(\mathcal{X}, z) \coloneqq \Big\{ x \in \prod_{n \in N} X_n : \sum_{n \in N} d_n(x_n, z_n) < \infty \Big\}
\]
equipped with the $\Lambda$-metric $d(x,y) = \sum_{n \in N} d_n(x_n,y_n)$.
 If each $(X_n,d_n)$ is median then $\prod_n X_n$ is median and the median map is given by $m(x,y,z) = (m(x_n,y_n,z_n))_{n \in N}$.

\subsection{BMW groups} \label{sec: BMW prelims}

Let $\Gamma_1 = (V_1, E_1), \Gamma_2 = (V_2, E_2)$ be simple graphs. The \textit{Cartesian product} $\Gamma$ of $\Gamma_1$ and $\Gamma_2$ is the graph with vertex set $V \coloneqq V_1 \times V_2$ and edge set $E \coloneqq \{\{(v_1,v_2), (w_1, w_2)\} : \{v_1, w_1\} \in E_1$ and $v_2 = w_2$ or $\{w_1, w_2\} \in E_2$ and $v_1 = w_1\}$. Equivalently, the $\ell^1$ product $\Gamma_1 \times \Gamma_2$ admits a natural cell structure where all the closed cells are either points, closed intervals or squares, and $\Gamma$ is the 1-skeleton of $\Gamma_1 \times \Gamma_2$.

A \textit{BMW group} is a group $G$ which admits a free transitive action on the vertex set of the Cartesian product of two locally finite simplicial trees $T_1, T_2$ such that the action preserves the factors, meaning that $G \leq \Isom(T_1) \times \Isom(T_2)$. All BMW groups admit a specific type of presentation:

\begin{definition}
    A \textit{BMW presentation} is a group presentation of the form $\la A \cup X | R \ra$ where $A,X$ are disjoint finite sets and the set of relations $R$ satisfies the following:
    \begin{itemize}
        \item $R = R_2 \sqcup R_4$ where each $r \in R_2$ is of the form $r = t^2$ for some $t \in A \cup X$ and each $r \in R_4$ is of the form $r = axa'x'$ for some $a,a' \in A \cup A^{-1}, x,x' \in X \cup X^{-1}$.
        \item For all $a \in A \cup A^{-1}, x \in X \cup X^{-1}$ there exists a unique $a' \in A \cup A^{-1}$ and a unique $x' \in X \cup X^{-1}$ such that $axa'x'$ or $a'x'ax$ or $a^{-1}x'a'x^{-1}$ or $a'x^{-1}a^{-1}x'$ belongs to $R_4$.
    \end{itemize}
\end{definition}

The following proposition is from \cite[Proposition~4.2]{Caprace-finite&infinite}.

\begin{proposition} \label{prop: BMW basics}
    Every BMW group admits a BMW presentation. Conversely, let $G = \la A \cup X | R \ra$ be a BMW presentation with $R = R_2 \cup R_4$ as above. Let $A' \coloneqq \{a \in A : a^2 \in R_2\}, X' \coloneqq \{x \in X : x^2 \in R_2\}$ and $m \coloneqq |A - A'|, m' \coloneqq |A'|, n \coloneqq |X - X'|, n' \coloneqq |B'|$. Then the Cayley graph of $G$ with respect to $A \cup X$ is the Cartesian product of two simplicial trees $T_A, T_X$ with degree $2m+m', 2n +n'$ respectively. The action of $G$ on its Cayley graph preserves the factors and is free and transitive on the vertex set so in particular $G$ is a BMW group.
\end{proposition}

\begin{definition}
    A group presentation $\la S | R \ra$ is \textit{positive} if every element $r \in R$ is of the form $r = xy$ where $x$ is a non-trivial word in $S$ and $y$ is a non-trivial word in $S^{-1}$. 
\end{definition}

We will only consider BMW groups which admit positive BMW presentations. In particular, if $G = \la A \cup X | R \ra$ is a positive BMW presentation then $R_2 = \emptyset$ and for each $r = axa'x' \in R_4$ we have $|\{a,a'\} \cap A| = |\{a,a'\} \cap A^{-1}| = |\{x,x'\} \cap X| = |\{x,x'\} \cap X^{-1}| = 1$ so it follows that $|R| =  mn$ and $G$ is torsion-free \cite[Proposition~4.2(ii)]{Caprace-finite&infinite}.

\begin{definition}
    A BMW group $G$ is \textit{reducible} if there is a finite index subgroup of $G$ which splits non-trivially as a direct product. Otherwise $G$ is \textit{irreducible}.
\end{definition}

\subsection{Cantor--Bendixson rank}
Let $Y$ be a Polish space (i.e. separable and completely metrisable). The Cantor--Bendixson theorem (see e.g. \cite{Kechris}) states that there is a unique decomposition of $Y$ as a disjoint union $\mathcal{K}(Y) \sqcup C$, where $\mathcal{K}(Y)$ is perfect (i.e. closed with no isolated points) and $C$ is countable. Any non-empty perfect Polish space contains a Cantor set, so the perfect subspace $\mathcal{K}(Y)$ is empty if and only if $Y$ is countable. 

The Cantor--Bendixson derivatives of $Y$ are defined by transfinite induction. The first derivative $Y^{(1)}$ is the set consisting of all the points of $Y$ which are not isolated. If $\alpha$ is an ordinal for which $Y^{(\alpha)}$ is defined then $Y^{(\alpha+1)}$ is the set of non-isolated points of $Y^{(\alpha)}$. If $\beta$ is a limit ordinal such that $Y^{(\alpha)}$ is defined for all $\alpha < \beta$ then $Y^{(\beta)} \coloneqq \cap_{\alpha < \beta} Y^{(\alpha)}$. The Cantor--Bendixson theorem implies that there is a countable ordinal $\alpha$ such that $Y^{(\alpha + 1)} = Y^{(\alpha)}$ (i.e. $Y^{(\alpha)} = \mathcal{K}(Y)$). The minimal such $\alpha$ is called the  \textit{Cantor--Bendixson rank} (or \textit{CB-rank}) of $Y$.

\begin{remark} \label{rem: CB-rank succesor}
    If $Y$ is a countable compact metrisable space then its CB-rank must be a successor ordinal. Indeed, if $\alpha$ is a limit ordinal and the CB-rank of $Y$ is greater than or equal to $\alpha$, then there exists $y_\beta \in Y^{(\beta)}$ for all $\beta < \alpha$ and, by compactness, this implies that there is a sequence in $Y$ which converges to a point in $\cap_{\beta < \alpha} Y^{(\beta)} = Y^{(\alpha)}$. Since $Y$ is countable, $\mathcal{K}(Y) = \emptyset$ so $\alpha$ is the not the CB-rank of $Y$.
\end{remark}

\section{Ores and their extracted groups} \label{sec: ores}

The goal of this section is to introduce an abstract construction which we will use to produce groups which act freely and transitively on various median metric spaces. A good example to keep in mind while reading this section is the construction of a free group using words in an alphabet $A \cup A^{-1}$.

We start with some notions and results from order theory. Recall that a \textit{partially ordered set} is a set $Y$ equipped with a reflexive, antisymmetric and transitive relation $\preceq$.

\begin{definition} [Semilattice, meet, join, bottom] \label{def: semilattice}
Let $(Y, \preceq)$ be a partially ordered set. Fix $x,y \in Y$.
\begin{itemize}
    \item Suppose there exists an element $z \in Y$ such that $z \preceq x$, $z \preceq y$ and, for all $s \in Y$ such that $s \preceq x$ and $s \preceq y$, we have that $s \preceq z$. Then $z$ is called the \textit{meet} of $x$ and $y$ and we write $x \wedge y \coloneqq z$.
    \item Suppose there exists an element $z \in Y$ such that $x \preceq z$, $y \preceq z$ and for all $s \in Y$ such that $x \preceq s$ and $y \preceq s$ we have that $z \preceq s$. Then $z$ is called the \textit{join} of $x$ and $y$ and we write $x \vee y \coloneqq z$.
    \item Suppose there exists $z \in Y$ such that $z \preceq s$ for all $s \in Y$. Then $z$ is called the \textit{bottom element} of $Y$ (by the antisymmetry of $\preceq$, such an element must be unique).
\end{itemize}
The partially ordered set $Y$ is a \textit{meet semilattice} if, for all $x,y \in Y$, the meet $x \wedge y \in Y$ exists.
\end{definition}

Where they are defined, the operations $\wedge$ and $\vee$ are clearly associative, so we will write $x \wedge y \wedge z = x \wedge (y \wedge z) = (x \wedge y) \wedge z$ and $x \vee y \vee z = x \vee (y \vee z) = (x \vee y) \vee z$.

\begin{definition}[Orthogonality]
    Let $Y$ be a meet semilattice with a bottom element $\id$. We say that two elements $x,y \in Y$ are \textit{orthogonal}, denoted $x \perp y$, if $x \wedge y = \id$ and $x \vee y$ exists.
\end{definition}

\begin{remark}
The relation $\perp$ is symmetric and irreflexive on $Y - \{\id\}$.
\end{remark}

\begin{definition}[Median semilattice] \label{def: median semilattice}
A meet semilattice $(Y, \preceq)$ is \textit{median} if the following holds. For any $a,b,c \in Y$, the join $(a \wedge b) \vee (b \wedge c) \vee (a \wedge c)$ exists and, for any $x \in Y$, 
\[ (x \wedge a \wedge b) \vee (x \wedge b \wedge c) \vee (x \wedge a \wedge c) = x \wedge ((a \wedge b) \vee (b \wedge c) \vee (a \wedge c)).\]
\end{definition}
    
The term median is justified by the following result of Sholander.

\begin{theorem}[Sholander, {\cite{Sholander}}] \label{thm: median semilattice}
    Let $(Y, \preceq)$ be a median semilattice and define $m:Y^3 \rightarrow Y$ by $m(a,b,c) \coloneqq (a \wedge b) \vee (b \wedge c) \vee (a \wedge c)$. Then $(Y,m)$ is a median algebra.
\end{theorem}

\begin{definition} \label{def: monoids}
    Let $(M ,\; \cdot \;, 1)$ be a monoid.
    \begin{itemize}
        \item $M$ is \textit{cancellative} if, for all $x,y,z \in M$, we have that $x \cdot y = x \cdot z$ implies that $y = z$ and $x \cdot y = z \cdot y$ implies that $x = z$. 
        \item Let $*: M \rightarrow M$ be an involution. Then $M$ is a \textit{monoid with involution} $*$ if $(x \cdot y)^* = y^* \cdot x^*$ for all $x,y \in M$, where we denote by $m^*$ the $*$-image of $m \in M$.
    \end{itemize}
\end{definition}

\begin{remark}
    If $M$ is a cancellative monoid with involution $*$ then $1^* = 1$. Indeed, for any $m \in M$, we have $1 \cdot m^* = m^* = (m \cdot 1)^* = 1^* \cdot m^*$ so $1=1^*$ by right cancellation .
\end{remark}
\begin{definition}[Ore] \label{defn: ore}
An \textit{ore} is a tuple $(Y, \preceq, \id, \con, -1)$ such that the axioms (O1) -- (O6) below hold, where $Y$ is a set, $\preceq$ is a partial order on $Y$, $\id \in Y$, $\con$ is a binary operation on $Y$ and $-1:Y \rightarrow Y$ is a map.

\begin{enumerate} 
\item[(O1)] $(Y,\preceq)$ is a meet semilattice with bottom element $\id$.
\item[(O2)] $(Y,\con)$ is a cancellative monoid with involution $-1$ and identity $\id$.
\item[(O3)] For all $x,y \in Y$ we have that $x \preceq y$ if and only if there exists $z \in Y$ such that $y = x \con z$.
\end{enumerate}
Let $y^{-1}$ denote the $-1$-image of each $y \in Y$.
An element $f \in Y$ is called \textit{inadmissible} if there exist $x,y,z \in Y$ such that $y \neq \id$ and $f = x \con y \con y^{-1} \con z$. Otherwise $f$ is \textit{admissible}. 

Let $G \subseteq Y$ be the set of admissible elements.  It follows from (O3) that if $y \in G$ and $x \preceq y$ then $x \in G$. Therefore $(G,\preceq)$ is a meet semilattice.
\begin{enumerate}
\item[(O4)] $(G,\preceq)$ is a median semilattice.
\item[(O5)] Suppose that $x,y \in Y$ are orthogonal, $x$ is admissible and there exists $y' \in G$ is such that $x \vee y = x \con y'$. Then $y$ is admissible.
\item[(O6)] Let $x,y \in G$ be orthogonal and, using (O3), let $x', y' \in Y$ be such that $x \vee y = x \con y' = y \con x'$. Then $x',y'$ are admissible, $x^{-1} \perp y'$ and $y^{-1} \perp x'$, and
\[
    x^{-1} \vee y' = x^{-1} \con y = y' \con (x')^{-1} = (y^{-1} \vee x')^{-1}.
\]
\end{enumerate}
\end{definition}

\begin{example} \label{ex: free group from an ore}
    Let $A$ be a non-empty set and $A^{-1}$ be the set of formal inverses of elements in $A$. Let $\mathcal{W} = \mathcal{W}(A \cup A^{-1})$ be the set of words in $A \cup A^{-1}$. Given $w,w' \in \mathcal{W}$, let $w \con w'$ be the concatenation of $w$ and $w'$. We say that $w \preceq w'$ if and only if $w' = w \con v$ for some $v \in \mathcal{W}$. If $w = a_1^{\varepsilon_1} \con \dots \con a_n^{\varepsilon_n}$, where $a_i \in A$ and $\varepsilon_i \in \{1,-1\}$ for each $i$, then $w^{-1} \coloneqq a_n^{-\varepsilon_n} \con \dots \con a_1^{-\varepsilon_1}$. Let $\id \in \mathcal{W}$ be the empty word. Then $(\mathcal{W}, \preceq, \id, \con, -1)$ is an ore.
\end{example}

\begin{lemma} \label{lem: existence of joins}
    Let $(Y, \preceq, \id, \con, -1)$ satisfy (O1), (O2) and (O3). Let $x,y \in Y$ and suppose  there exists $z \in Y$ such that $x \preceq z$ and $y \preceq z$. Then there exists $x \vee y \in Y$.
\end{lemma}
\begin{proof}
    By (O3) there exists $a,b \in Y$ such that $z = x \con a = y \con b$. By (O1) there exists $c \coloneqq a^{-1} \wedge b^{-1}$ and by (O3) again there exist $a', b' \in Y$ such that $a^{-1} = c \con {a'}^{-1}$ and $b^{-1} = c \con {b'}^{-1}$. By (O2) this implies that $a = a' \con c^{-1}$ and $b = b' \con c^{-1}$. Note that $x \con a' \con c^{-1} = z = y \con b' \con c^{-1}$ so by (O2) $x \con a' = y \con b'$. Let $w \coloneqq x \con a'$. Then $x \preceq w$ and $y \preceq w$ by (O3). Let $z' \in Y$ be such that $x \preceq z'$ and $y \preceq z'$ and let $w' \coloneqq z' \wedge w$. Then there exists $c' \in Y$ such that $w = w' \con c'$ and, by the above argument there exist $a'',b'' \in Y$ such that $w' = x \con a'' = y \con b''$. We then have
    \begin{align*} 
        z &= x \con a = x \con a' \con c^{-1} = x \con a'' \con c' \con c^{-1} \\
        &= y \con b = y \con b' \con c^{-1} = y \con b'' \con c' \con c^{-1}.
    \end{align*}
    Since $c = a^{-1} \wedge b^{-1}$, this implies that $c' = \id$ and therefore $w' = w$. Therefore $w \preceq z'$ and $w = x \vee y$. 
\end{proof}

\begin{remark} \label{rem: orthogonal things}
    Suppose $(Y, \preceq, \id, \con, -1)$ is as in Lemma~\ref{lem: existence of joins}. Let $x,y,z \in Y$ be such that $x \preceq y$ and $y \perp z$. Then $x \wedge z = \id$ and $x, z \preceq y \vee z$, so $x \perp z$ by Lemma~\ref{lem: existence of joins}.
\end{remark}
For the rest of the section, we assume that $(Y, \preceq, \id, \con, -1)$ is an ore and $G \subseteq Y$ is the set of admissible elements in $Y$.

\begin{definition}[Median]
    Define a map $m: G^3 \rightarrow G$ by $m(f,g,h) \coloneqq (f \wedge g) \vee (g \wedge h) \vee (f \wedge h)$ for all $f,g,h \in G$. The map $m$ is called the \textit{median map} and for all $f,g,h \in G$, the point $m(f,g,h)$ is the \textit{median of} $f,g,h$.
\end{definition}

\begin{remark} \label{rem: G is median}
    The map $m$ is well-defined by (O4). By Theorem~\ref{thm: median semilattice}, $(G,m)$ is a median algebra.
\end{remark}

\begin{definition}[Orthogonal complement, parallel transport] \label{def: parallel transport}
Fix $y \in Y$. The \textit{orthogonal complement} of $y$ is the set $y^\perp \coloneqq \{x \in Y : x \perp y\}$. Define a map $\Phi_y: y^\perp \rightarrow Y$ by $\Phi_y(x) = x'$, where $x' \in Y$ is the unique element such that $y \vee x = y \con x'$ (existence is given by (O3) and uniqueness by (O2)). The element $\Phi_y(x)$ is called the \textit{parallel transport of $x$ along $y$}.
\end{definition}

\begin{remark} \label{rem: admissible factors}
    It follows from (O6) that $\Phi_g(G \cap g^\perp) = G \cap (g^{-1})^\perp$ and $\Phi_{g^{-1}}$ is the inverse map of $\Phi_g$ for all $g \in G$.
\end{remark}

\begin{definition}[Faces]
Given $x,y \in Y$, we say that $x$ is a \textit{face} of $y$ if $x \preceq y$. Given $y = x \con z$ with $x,z \in Y$, let $y \divdot x \coloneqq z$ and let $y \dotdiv z \coloneqq x$.
\end{definition}

\begin{remark} \label{rem: transports preserve the face relation}
    The parallel transport maps preserve the face relation.
    Indeed, suppose that $a,x,y \in Y$ with $x \perp y$ and $a \preceq x$. Then $a \wedge y = \id$ and $a \preceq x \vee y$, so $a \perp y$. Moreover $y \con \Phi_y(a) = y \vee a \preceq y \vee x = y \con \Phi_y(x)$ by Definitions~\ref{def: semilattice} and \ref{def: parallel transport}. By (O2) and (O3), we then have that $\Phi_y(a) \preceq \Phi_y(x)$.
    
    In particular, it follows from Remark~\ref{rem: admissible factors} that $x = \id$ if and only if $\Phi_y(x) = \id$.
\end{remark}

\begin{lemma}[Spanning a cube] \label{spanning a cube}
    Suppose that $a,b,c \in G$ are pairwise orthogonal. Then $\Phi_a(b) \perp \Phi_a(c)$, $\Phi_b(a) \perp \Phi_b(c)$, $\Phi_c(a) \perp \Phi_c(b)$ and we have the following equalities:
    \begin{align*}
        \Phi_{\Phi_c(b)}(\Phi_c(a)) &= \Phi_{\Phi_b(c)}(\Phi_b(a)) \\
        \Phi_{\Phi_a(b)}(\Phi_a(c)) &= \Phi_{\Phi_b(a)}(\Phi_b(c)) \\
        \Phi_{\Phi_c(a)}(\Phi_c(b)) &= \Phi_{\Phi_a(c)}(\Phi_a(b)).
    \end{align*}
\end{lemma}

\begin{proof}
    By (O6) we have that $a^{-1} \perp \Phi_a(b)$, $a^{-1} \perp \Phi_a(c)$ and $b = \Phi_{a^{-1}}(\Phi_a(b))$, $c = \Phi_{a^{-1}}(\Phi_a(c))$. Therefore, using Remark~\ref{rem: transports preserve the face relation}, if $\Phi_a(b) \wedge \Phi_a(c) \neq \id$ then $b \wedge c \neq \id$, which is a contradiction. Thus $\Phi_a(b) \wedge \Phi_a(c) = \id$ and, by a similar argument, $\Phi_b(a) \wedge \Phi_b(c) = \Phi_c(a) \wedge \Phi_c(b) = \id$.

    Let $m \coloneqq m(a \vee b, a \vee c, b \vee c) = a \vee b \vee c \in G$ and let $x,y,z \in G$ be such that $m = (a \vee b) \con x = (a \vee c) \con y = (b \vee c) \con z$. Then $\Phi_a(b) \con x = \Phi_a(c) \con y$, $\Phi_b(a) \con x = \Phi_b(c) \con z$ and $\Phi_c(a) \con y = \Phi_c(b) \con z$. Thus $\Phi_a(b) \perp \Phi_a(c)$, $\Phi_b(a) \perp \Phi_b(c)$ and $\Phi_c(a) \perp \Phi_c(b)$ by Lemma~\ref{lem: existence of joins}. Moreover $a,b,c \preceq a \con (\Phi_a(b) \vee \Phi_a(c))$ so 
    \[ a \con \Phi_a(b) \con x = a \con \Phi_a(c) \con y = m \preceq a \con (\Phi_a(b) \vee \Phi_a(c)). \]
    Therefore $\Phi_a(b) \con x = \Phi_a(c) \con y = \Phi_a(b) \vee \Phi_a(c)$, which implies that $\Phi_{\Phi_a(b)}(\Phi_a(c)) = x$ and $\Phi_{\Phi_a(c)}(\Phi_a(b)) = y$. By a similar argument, $\Phi_{\Phi_b(c)}(\Phi_b(a)) = z$, $\Phi_{\Phi_b(a)}(\Phi_b(c)) = x$, $\Phi_{\Phi_c(a)}(\Phi_c(b)) = y$ and $\Phi_{\Phi_c(b)}(\Phi_c(a)) = z$.
\end{proof}

\begin{definition}
    Let $f, g \in G$. We say that $f$ and $g$ \textit{concatenate geodesically} if $f^{-1} \wedge g = \id$.
\end{definition}

\begin{lemma} \label{geodesic concatenation lemma}
    Let $f,g \in G$ and suppose that $f$ and $g$ concatenate geodesically. Then $f \con g \in G$.
\end{lemma}
\begin{proof}
    Let $a,b,c \in Y$ be such that $f \con g = a \con b \con b^{-1} \con c$. Using the fact that admissibility is preserved by passing to faces, we can assume without loss of generality that $a \wedge f = c^{-1} \wedge g^{-1} = \id$. Note that $a, f \preceq f \con g$ and $c^{-1}, g^{-1} \preceq g^{-1} \con f^{-1}$ so $a \perp f$ and $c^{-1} \perp g^{-1}$ by Lemma~\ref{lem: existence of joins}. Moreover, $f \vee a \preceq f \con g$ so $\Phi_{f}(a) \preceq g$. Therefore $\Phi_f(a)$ is admissible, which implies by (O5) that $a$ is admissible. 
    Let $f_1 \coloneqq f \wedge (a \con b), f_2 \coloneqq f \divdot f_1$, $g_1 \coloneqq \Phi_f(a), g_2 \coloneqq g \divdot g_1$ and $x \coloneqq \Phi_{f_1}(a)$ (this exists by Remark~\ref{rem: orthogonal things}). By (O6) we have $x \in G$. Note that $f_1 \con x \preceq a \con b$ so by definition $f_2 \wedge x = \id$ and $f_2, x \preceq f_2 \con g$, so $f_2 \perp x$. By (O6) again, $\Phi_x(f_2) \in G$ and $x^{-1} \perp \Phi_x(f_2)$.

    Since $(Y, \con)$ is cancellative, $\Phi_a(f_1) \preceq b$. Let $b_1 \coloneqq \Phi_a(f_1), b_2 \coloneqq b \divdot b_1$. Then $b_2 \perp \Phi_x(f_2)$ and $\Phi_{\Phi_x(f_2)}(b_2) \preceq g_2$. Let $y \coloneqq x \con b_2$ and note that $y \perp f_2$. Let $z \coloneqq \Phi_{\Phi_x(f_2)}(b_2)$ and $g_2' \coloneqq g_2 \divdot z$.
    Since $y \con b^{-1} \con c = f_2 \con g$, it follows from (O5) that $y$ is admissible and from (O6) that $y^{-1} \perp \Phi_y(f_2)$ so $b_2^{-1} \perp \Phi_y(f_2)$. Now $b^{-1} \vee \Phi_y(f_2) \preceq b^{-1} \con c = \Phi_y(f_2) \con g_2'$ so $z^{-1} \preceq \Phi_{\Phi_y(f_2)}(b^{-1}) \preceq g_2'$ which implies that $g = g_1 \con z \con z^{-1} \con (g_2' \divdot z^{-1})$.
    Since $g \in G$, $z = \id$ so $b_2 = \id$.
    
    Therefore $\Phi_a(f_1) = b$. It follows that $b, b^{-1} \in G$ and $b^{-1} \perp x^{-1}$. The admissibility of $f$ implies that $b^{-1} \wedge \Phi_x(f_2) = \id$ so, since $b^{-1} \con c = \Phi_x(f_2) \con (g \divdot \Phi_{f_2}(x))$, we have $b^{-1} \perp \Phi_x(f_2)$. Applying Lemma ~\ref{spanning a cube} to the triple $b^{-1}, \Phi_x(f_2), x^{-1}$, we find that $f_1^{-1} \perp f_2, g_1^{-1} \perp g_2$ and $\Phi_{f_2}(f_1^{-1}) = \Phi_{g_1^{-1}}(g_2)$. By (O6), $f^{-1} = f_2^{-1} \con f_1^{-1} = \Phi_{f_2}(f_1^{-1}) \con \Phi_{f_1^{-1}}(f_2)^{-1}$ and $g = g_1 \con g_2 = \Phi_{g_1^{-1}}(g_2) \con \Phi_{g_2}(g_1^{-1})^{-1}$ so, since $f$ and $g$ concatenate geodesically, $\Phi_{f_2}(f_1^{-1}) = \id$. This implies that $f_1 = \id$ and therefore $b = \id$. 
\end{proof}

Given $f,g \in G$, observe that $f \dotdiv (f^{-1} \wedge g)^{-1}$ and $g \divdot (f^{-1} \wedge g)$ concatenate geodesically.
We can therefore define an operation $\twist: G \times G \rightarrow G$ by 
\[
    f \twist g \coloneqq  (f \dotdiv (f^{-1} \wedge g)^{-1}) \con (g \divdot (f^{-1} \wedge g)).
\]

We will prove that $\twist$ defines a group operation on $G$. To this end, let us prove some supporting lemmas.

 \begin{lemma} \label{pre-associative}
Let $a,b,c \in G$ and suppose $a$ and $b$ concatenate geodesically. Let $x, c_2, c_3 \in G$ be such that $b^{-1} = (b^{-1} \wedge c) \con x$, $c \wedge (b^{-1} \con a^{-1}) = (c \wedge b^{-1}) \con c_2$ and $c = (c \wedge b^{-1}) \con c_2 \con c_3$.
Then 
\begin{enumerate}
    \item[(i)] $x \perp c_2$;
    \item[(ii)] $\Phi_x(c_2) \preceq a^{-1}$;
    \item[(iii)] $(a \con b) \twist c = (a \dotdiv \Phi_x(c_2)^{-1}) \con \Phi_{c_2}(x)^{-1} \con c_3$.
\end{enumerate}
 \end{lemma}
\begin{proof}
    Let $F \coloneqq x \wedge c_2$. Then $b^{-1} = (b^{-1} \wedge c) \con F \con (x \divdot F)$ and $c = (c \wedge b^{-1}) \con F \con (c_2 \divdot F) \con c_3$. Therefore $(c \wedge b^{-1}) \con F \preceq c \wedge b^{-1}$ which implies that $(c \wedge b^{-1}) \con F = c \wedge b^{-1}$ and then $F = \id$ by left cancellation.
    Let $U \coloneqq c \wedge (b^{-1} \con a^{-1})$. Then
    \begin{align*}
        b^{-1} \con a^{-1} &= (c \wedge b^{-1}) \con x \con a^{-1} \\
        &= (c \wedge b^{-1}) \con c_2 \con ((b^{-1} \con a^{-1}) \divdot U).
    \end{align*}
   Since $(Y,\con)$ is cancellative, this implies that $x \con a^{-1} = c_2 \con ((b^{-1} \con a^{-1}) \divdot U)$, so $x \vee c_2$ exists by Lemma~\ref{lem: existence of joins}. This completes the proof of (i).

    We have that $x \con a^{-1} = (b^{-1} \con a^{-1}) \divdot (c \wedge b^{-1}) = c_2 \con ((b^{-1} \con a^{-1}) \divdot U)$ so there exists $A \in G$ such that 
    \[ 
    (b^{-1} \con a^{-1}) \divdot (c \wedge b^{-1}) = (x \vee c_2) \con A = x \con \Phi_x(c_2) \con A.\]
    Thus $a^{-1} = \Phi_x(c_2) \con A$ by left cancellation, which proves (ii).

    We have that 
    \begin{align*}
        (b^{-1} \con a^{-1}) \divdot U &= (x \con a^{-1}) \divdot c_2 \\
        &= (x \con \Phi_x(c_2) \con A) \divdot c_2 \\
        &= (c_2 \con \Phi_{c_2}(x) \con A) \divdot c_2 \\
        &= \Phi_{c_2}(x) \con A
    \end{align*}
    and $A = a^{-1} \divdot \Phi_x(c_2)$ so $A^{-1} = a \dotdiv \Phi_x(c_2)^{-1}$. Therefore
    \begin{align*}
        (a \con b) \twist c &= ((a \con b) \dotdiv ((a \con b)^{-1} \wedge c)^{-1}) \con (c \divdot ((a \con b)^{-1} \wedge c)) \\
        &= ((b^{-1} \con a^{-1}) \divdot U)^{-1} \con (c \divdot U) \\
        &= (a \dotdiv \Phi_x(c_2)^{-1}) \con \Phi_{c_2}(x)^{-1} \con c_3
    \end{align*}
    which completes the proof of (iii).
\end{proof}
 
\begin{lemma} \label{Ore associativity}
    The operation $\twist$ is associative.
\end{lemma}

\begin{figure}
    \centering
    \scalebox{1.3}[1.2]{
    \begin{tikzpicture}
        \draw[Thistle!60] (1,0.2722) -- (2.97,0.9176);
        \draw[fill=Thistle!23,Thistle!23] (1,0.2222) -- (3,0) -- (6,1) -- (4.00029,1.2051) -- (1,0.2222);
        \draw[fill=YellowGreen!30,YellowGreen!30] (1,1.9) -- (3,1.7) -- (6,2.66667) -- (4.00029,2.85041);
        \draw[thick,YellowGreen!130,-stealth] (3.08,2.55905) -- (4.00029,2.85041) -- (5.3,3.8);
        \draw[YellowGreen!80,-stealth] (3.08,2.60905) -- (3.99,2.8975) -- (5.27,3.83);
        \draw[YellowGreen!80] (1.03,1.95) -- (2.945,2.56613);
        \draw[fill=MidnightBlue!23,MidnightBlue!23] (3,0) -- (3,1.7) -- (6,2.66667) -- (6,1);
        \draw[fill=Plum!25,Plum!25] (3,0) -- (6,1) -- (4.00029,1.2051) -- (3,0.8774);
        \draw[thick,Plum!80] (3.08,0.9036) -- (4.00029,1.2051) -- (6,1);
        \draw[thick,Thistle] (1,0.2222) -- (2.97,0.8676);
        \draw[thick,YellowGreen!130] (2.975,4) -- (2.975,1.7);
        \draw[Gray!140,-stealth] (1,1.9) -- (1,0.2222);
        \node[anchor=west,color=MidnightBlue!80] at (3.02,0.66) {$W$};
        \node[anchor=east,color=Gray!140] at (1,1.1) {$\Phi_A(W)$};
        \draw[YellowGreen!80,-stealth] (2.97,1.64995) -- (1.03,1.85);
        \node[anchor=north,color=YellowGreen!80] at (1.9,1.75) {$A$};
        \node[anchor=west,color=MidnightBlue!60] at (3,3.1) {$D$};
        \draw[thick,YellowGreen!130] (4.00029,2.85041) -- (5.975,2.660) -- (3,1.7) -- (1.03,1.9) -- (2.945,2.51613);
        \draw[thick,MidnightBlue,-stealth] (7.5,1.025) -- (6,1.025) -- (3,0.025) -- (3,4);
        \draw[MidnightBlue!60,-stealth] (3.05,4) -- (3.05,1.755);
        \draw[MidnightBlue!60,-stealth] (3.05,1.625) -- (3.05,0.075);
        \node[color=Gray!140,anchor=east] at (4.05,2.3) {$Y$};
        \draw[thick,MidnightBlue] (6,1) -- (6,2.64167) -- (3,1.675);
        \draw[MidnightBlue!60,-stealth] (3.045,1.625) -- (5.97,2.5675);
        \node[anchor=south,color=MidnightBlue] at (2.99,4) {$h_2$};
        \draw[thick,Thistle,-stealth] (-1,0.2222) -- (1,0.2222) -- (3,0) -- (6,1) -- (7.515,1);
        \node[anchor=west,color=Thistle] at (7.5,1) {$h_1$};
        \node[anchor=south,color=YellowGreen!130] at (5.45,3.75) {$h_3$};
        \draw[color=Thistle!60,-stealth] (7.5,0.95) -- (6,0.95) -- (3.03,-0.049);
        \node[anchor=north,color=Thistle!60] at (4.9,0.625) {$B$};
        \node[color=MidnightBlue!80,anchor=north] at (4.7,2.2) {$E$};
        \draw[color=Thistle!60,-stealth] (-1,0.1722) -- (1,0.1722);
        \node[anchor=north,color=Thistle!60] at (-0.1,0.18) {$h_1'$};
        \draw[Plum!60,-stealth] (3.08,0.9536) -- (4,1.255);
        \draw[Gray!140,-stealth] (4.00029,2.05) -- (4.00029,2.84);
        \draw[Gray!140] (4.00029,1.2051) -- (4.00029,1.915);
        \node[color=YellowGreen!80] at (2.4,2.75) {$h_3'$};
        \node[color=Thistle!60] at (2.5,1.05) {$X$};
    \end{tikzpicture}}
    \caption{Notation for the proof of Lemma~\ref{Ore associativity}. Elements of $G$ are represented by paths, two paths represent the same element if they have the same endpoints and the operation $\con$ is represented by path concatenation.}
    \label{fig: proof of associativity}
\end{figure}

The strategy of the proof of this lemma is inspired by the proof of Lemma~40.141 in \cite{CRHK}.

\begin{proof}
    Fix $h_1,h_2,h_3 \in G$ and let $B \coloneqq h_2 \wedge h_1^{-1}$. Then $h_1 \twist h_2 = \widetilde{h}_1 \con \widetilde{h}_2$, where $h_1 = \widetilde{h}_1 \con B^{-1}$ and $h_2 = B \con \widetilde{h}_2$. Let $D \coloneqq h_3 \wedge \widetilde{h}_2^{-1}$ and $A \coloneqq (h_3 \wedge  (\widetilde{h}_1 \con \widetilde{h}_2)^{-1}) \divdot D$. Then there exist $W, h_3' \in G$ such that $\widetilde{h}_2^{-1} = D \con W$ and $h_3 = D \con A \con h_3'$. By Lemma~\ref{pre-associative}
    \begin{itemize}
        \item[(i)] $W \perp A$,
        \item[(ii)] there exists $h_1' \in G$ such that $\widetilde{h}_1 = h_1' \con \Phi_W(A)^{-1}$ and
        \item[(iii)] $(h_1 \twist h_2) \twist h_3 = h_1' \con h_2' \con h_3'$, where $h_2' \coloneqq \Phi_A(W)^{-1}$.
    \end{itemize}
    Since $D \preceq \tilde{h}_2^{-1} \preceq h_2^{-1}$ and $D \preceq h_3$, we have $D \preceq h_2^{-1} \wedge h_3$. So there exists $E \in G$ such that $h_2^{-1} \wedge h_3 = D \con E$. Then $D \con E \preceq h_3 = D \con A \con h_3'$ so $E \preceq A \con h_3'$. Also $D \con E \preceq h_2^{-1} = \widetilde{h}_2^{-1} \con B^{-1} = D \con W \con B^{-1}$ so $E \preceq W \con B^{-1}$.

    \begin{claim} \label{Claim 1}
        $E \perp W$ and $\Phi_W(E) \preceq B^{-1}$.
    \end{claim}
    \begin{proof}
    \renewcommand{\qedsymbol}{$\blacksquare$}
    Note that $D \con (E \wedge W) \preceq D \con W = \widetilde{h}_2^{-1}$ so, since $D \con E = h_2^{-1} \wedge h_3 \preceq h_3$, we have that $D \con (E \wedge W) \preceq \widetilde{h}_2^{-1} \wedge h_3 = D$. Therefore $E \wedge W = \id$. Moreover $h_2^{-1} = D \con W \con B^{-1}$ and $h_2^{-1} \wedge h_3 = D \con E$, so $E, W \preceq W \con B^{-1}$, which implies that $E \perp W$. Finally, $E \vee W = W \con \Phi_W(E) \preceq W \con B^{-1}$ so $\Phi_W(E) \preceq B^{-1}$ as required.
    \end{proof}
    
    \begin{claim} \label{Claim 2}
        $E \perp A$.
    \end{claim}
    \begin{proof}
    \renewcommand{\qedsymbol}{$\blacksquare$}
    Let $F \coloneq E \wedge A$. Then, by Remark~\ref{rem: transports preserve the face relation}, $\Phi_W(F) \preceq \Phi_W(A)$ and, by Remark~\ref{rem: transports preserve the face relation} and Claim~\ref{Claim 1}, $\Phi_W(F) \preceq \Phi_W(E) \preceq B^{-1}$. Therefore:
    \[
        B^{-1} = \Phi_W(F) \con (B^{-1} \divdot \Phi_W(F)) \quad \text{and} \quad \Phi_W(A)^{-1} = (\Phi_W(A)^{-1} \dotdiv \Phi_W(F)^{-1}) \con \Phi_W(F)^{-1}.
    \]
    But then
    \begin{align*}
        h_1 &= \widetilde{h}_1 \con B^{-1} \\
        &= h_1' \con (\Phi_W(A)^{-1} \dotdiv \Phi_W(F)^{-1}) \con \Phi_W(F)^{-1} \con \Phi_W(F) \con (B^{-1} \divdot \Phi_W(F)).
    \end{align*}
    Since $h_1$ is admissible this implies that $\Phi_W(F) = \Phi_W(F)^{-1} = \id$ and therefore $F = \id$.
    Since $E \preceq A \con h_3'$, we then have that $E \perp A$.
    \end{proof}
    
In summary:
\[
    \begin{array}{lllll}
        h_1 &= &\widetilde{h}_1 \con B^{-1} &= &h_1' \con \Phi_W(A)^{-1} \con \Phi_W(E) \con (B^{-1} \divdot \Phi_W(E)), 
        \vspace{3pt} \\
        h_2 &= &B \con \widetilde{h}_2 &= &(B \dotdiv \Phi_W(E)^{-1}) \con \Phi_W(E)^{-1} \con W^{-1} \con D^{-1} \\
        && &= &(B \dotdiv \Phi_W(E)^{-1}) \con \Phi_E(W)^{-1} \con E^{-1} \con D^{-1},
        \vspace{3pt} \\
        h_3 &= &D \con A \con h_3' &= &D \con A \con \Phi_A(E) \con (h_3' \divdot \Phi_A(E)) \\
        && &= &D \con E \con \Phi_E(A) \con (h_3' \divdot \Phi_A(E)),
\end{array}
\]
$A \perp W$, $E \perp A$ and $E \perp W$.

Recall that $h_3 \wedge h_2^{-1} = D \con E$. We therefore have:
\begin{align*}
    h_2 \twist h_3 = \;&\big[h_2 \dotdiv (h_3 \wedge h_2^{-1})^{-1}\big] \con \big[h_3 \divdot (h_3 \wedge h_2^{-1})\big] \\
    = \; &\big[(B \dotdiv \Phi_W(E)^{-1}) \con \Phi_E(W)^{-1}\big] \con \big[\Phi_E(A) \con (h_3' \divdot \Phi_A(E))\big]
\end{align*}
Let $X \coloneqq \Phi_{\Phi_W(A)}(\Phi_W(E))$ and $Y \coloneqq \Phi_{\Phi_E(A)}(\Phi_E(W))^{-1}$. 
We then have:
\begin{align*}
    h_1 \twist (h_2 \twist h_3) 
    = \; &\big[ h_1' \con \Phi_W(A)^{-1} \con \Phi_W(E) \con (B^{-1} \divdot \Phi_W(E)) \big] \\
    &\twist \big[(B \dotdiv \Phi_W(E)^{-1}) \con \Phi_E(W)^{-1} \con \Phi_E(A) \con (h_3' \divdot \Phi_A(E))\big] \\
    = \; &\big[ h_1' \con \Phi_W(A)^{-1} \con \Phi_W(E) \big] \twist \big[ \Phi_E(W)^{-1} \con \Phi_E(A) \con (h_3' \divdot \Phi_A(E))\big] \\
    = \; &\big[ h_1' \con \Phi_{\Phi_W(A)}(\Phi_W(E)) \con \Phi_{\Phi_E(W)}(\Phi_E(A))^{-1} \big] \\
    &\twist \big[ \Phi_{\Phi_E(W)}(\Phi_E(A)) \con \Phi_{\Phi_E(A)}(\Phi_E(W))^{-1} \con (h_3' \divdot \Phi_A(E)) \big] \\
    = \; &\big[ h_1' \con X \big] \twist \big[ Y \con (h_3' \divdot \Phi_A(E) \big],
     \end{align*}
    where we use (O6) for the third equality. Let $Z \coloneqq X^{-1} \wedge (Y \con (h_3' \divdot \Phi_A(E)))$. Then $Z \preceq X^{-1}$ and $X^{-1} \perp Y$ by Lemma~\ref{spanning a cube}, so $Z \perp Y$ by Remark~\ref{rem: orthogonal things}. It follows that $\Phi_Y(Z) \preceq h_3' \divdot \Phi_A(E)$. But $\Phi_Y(Z) \preceq \Phi_Y(X^{-1}) = \Phi_A(E)^{-1}$ by Lemma~\ref{spanning a cube}. Therefore, since 
    \[
        h_3' = (\Phi_A(E) \dotdiv \Phi_Y(Z)^{-1}) \con \Phi_Y(Z)^{-1} \con \Phi_Y(Z) \con ((h_3' \divdot \Phi_A(E)) \divdot \Phi_Y(Z)),
    \]
    and $h_3'$ is admissible, we have $\Phi_Y(Z) = \Phi_Y(Z)^{-1} = \id$, which implies that $Z = \id$. Moreover, Lemma~\ref{spanning a cube} implies that $X \con Y = \Phi_A(W)^{-1} \con \Phi_A(E)$ so

    \begin{align*}
    h_1 \twist (h_2 \twist h_3) = &h_1' \twist \big[ X \con Y \con (h_3' \divdot \Phi_A(E)) \big] \\
    = &h_1' \twist \big[\Phi_A(W)^{-1} \con \Phi_A(E) \con (h_3' \divdot \Phi_A(E))\big] \\
    = &h_1' \twist (h_2' \con h_3') \\
    = &h_1' \con h_2' \con h_3' \\
    = &(h_1 \twist h_2) \twist h_3.
    \qedhere
\end{align*}
\end{proof}

\begin{theorem} \label{thm: extracted group}
    $(G,\twist)$ is a group.
\end{theorem}
\begin{proof}
    By Lemma~\ref{Ore associativity}, $\twist$ is an associative operation. It is clear from the definition that $\id$ is a two-sided identity for $\twist$. For any $g \in G$, $g \twist g^{-1} = (g^{-1} \divdot ((g^{-1})^{-1} \wedge g))^{-1} \con (g \divdot ((g^{-1})^{-1} \wedge g) = \id \con \id = \id$ and similarly $g^{-1} \twist g = \id$ so $g^{-1}$ is a two-sided inverse for $g$.
\end{proof}

\begin{definition} 
We will call $(G,\twist)$ the \textit{group extracted from} $Y$. 
\end{definition}

\begin{example}
    Let $(\mathcal{W}, \preceq, \id, \con, -1)$ be the ore defined in Example~\ref{ex: free group from an ore}. Then the group extracted from $\mathcal{W}$ is the free group $F(A)$.
\end{example}

\begin{remark}
    In general, there is no reason for a retraction $Y \rightarrow G$ which commutes with $\con$ and $-1$ to exist. In certain nice settings however, it is possible to construct such a retraction (see Sections~\ref{sec: reduction}).
\end{remark}

\begin{lemma} \label{lem: G is median preserving}
    For each $a,b \in G$, define:
    \[
        I(a,b) \coloneqq \{c \in G : (c^{-1} \twist a) \wedge (c^{-1} \twist b) = \id \}.
    \]
    Then $I(a,b) = \{c \in G : m(a,b,c) = c\}$.
    It follows that the action of $G$ on itself by left multiplication is median preserving.
\end{lemma}
\begin{proof}
    If the first part of the lemma holds, then $m(a,b,c)$ is the unique element in the intersection $I(a,b) \cap I(a,c) \cap I(b,c)$ for any $a,b,c \in G$ (see \cite[Lemma~3.2.1]{Bowditch}). It is immediate from the definition of $I(a,b)$ that $g \twist I(a,b) = I(g \twist a, g \twist b)$ for all $a,b,g \in G$, so this implies that $g \twist m(a,b,c) = m(g \twist a, g \twist b, g \twist c)$ for all $a,b,c,g \in G$.

    We now prove the first part of the lemma. Let $a,b,c \in G$ be such that $c \in I(a,b)$. Suppose that $a \wedge b = \id$. Then $a \wedge b, b \wedge c, a \wedge c \preceq c$, so $m \coloneqq m(a,b,c) \preceq c$. Let $D \in G$ be such that $c = m \con D$. Observe that $(a \wedge c) \perp (b \wedge c)$, so $m = (a \wedge c) \vee (b \wedge c) = (a \wedge c) \con \Phi_{a \wedge c}(b \wedge c) = (b \wedge c) \con \Phi_{b \wedge c}(a \wedge c)$. We then have:
    \begin{align*}
        c^{-1} \twist a 
        &= D^{-1} \con \Phi_{a \wedge c}(b \wedge c)^{-1} \con (a \divdot (a \wedge c)) \\
        c^{-1} \twist b
        &= D^{-1} \con \Phi_{b \wedge c}(a \wedge c)^{-1} \con (b \divdot (b \wedge c)).
    \end{align*}
    Therefore $D^{-1} \preceq (c^{-1} \twist a) \wedge (c^{-1} \twist b)$ which implies that $D = \id$ and $m = c$.

    In the general case, let $D,E,F \in G$ be such that 
    \[
        a \wedge b = (a \wedge b \wedge c) \con D, \quad b \wedge c = (a \wedge b \wedge c) \con E, \quad a \wedge c = (a \wedge b \wedge c) \con F.
    \]
    Observe that $D,E$ and $F$ are pairwise orthogonal, so by Lemma~\ref{spanning a cube}, $\Phi_E(D) \perp \Phi_E(F)$ and $\Phi_F(D) \perp \Phi_F(E)$. Let $a',a'',b',b'', c',c'' \in G$ be such that 
    \[
        a = (a \wedge c) \con a' = (a \wedge b) \con a'', \quad b = (b \wedge c) \con b' = (a \wedge b) \con b'', \quad c = (a \wedge c) \con c' = (b \wedge c) \con c''.
    \]
    Then $E \con c'' = F \con c' = (E \vee F) \con c'''$ for some $c''' \in G$ such that ${c'''}^{-1} \preceq {c'}^{-1}$ and ${c'''}^{-1} \preceq {c''}^{-1}$. But then
    \[
        {c'''}^{-1} \preceq {c'}^{-1} \con a' = c^{-1} \twist a \quad \text{and} \quad {c'''}^{-1} \preceq {c''}^{-1} \con b' = c^{-1} \twist b,
    \]
    so the fact that $c \in I(a,b)$ implies that $c''' = \id$. Therefore $c' = \Phi_F(E),$ and $c'' = \Phi_E(F)$. Moreover, 
    \[
        b = (a \wedge b \wedge c) \con E \con b' = (a \wedge b \wedge c) \con D \con b''
    \]
    so $E \con b' = D \con b'' = (E \vee D) \con b'''$ for some $b''' \in G$.
    Therefore $b' = \Phi_E(D) \con b'''$ and $c^{-1} \star b = \Phi_E(F)^{-1} \con \Phi_E(D) \con b'''$. By (O6), $\Phi_E(F)^{-1} \perp \Phi_{\Phi_E(F)}(\Phi_E(D))$ and 
    \[
        \Phi_E(F)^{-1} \vee \Phi_{\Phi_E(F)}(\Phi_E(D)) = \Phi_E(F)^{-1} \con \Phi_E(D) = \Phi_{\Phi_E(F)}(\Phi_E(D)) \con (\Phi_{\Phi_E(D)}(\Phi_E(F)))^{-1}.
    \]
    Thus $\Phi_{\Phi_E(F)}(\Phi_E(D)) \preceq \Phi_E(F)^{-1} \con \Phi_E(D) \con b''' = c^{-1} \star b$.
    By a symmetric argument, $\Phi_{\Phi_F(E)}(\Phi_F(D)) \preceq c^{-1} \twist a$. By Lemma~\ref{spanning a cube}, $\Phi_{\Phi_E(F)}(\Phi_E(D)) = \Phi_{\Phi_F(E)}(\Phi_F(D))$, so this implies that $D = \id$ since $c \in I(a,b)$. Thus $a \wedge b \preceq c$.

    Now let $A,B,C \in G$ be such that $a = (a \wedge b) \con A, b = (a \wedge b) \con B, c = (a \wedge b) \con C$. Then $c^{-1} \twist a = C^{-1} \twist A$ and $c^{-1} \twist b = C^{-1} \twist B$ so, by the above argument, we have:
    \[
        m(a,b,c) = (a \wedge b) \con m(A,B,C) = (a \wedge b) \con C = c.
    \]

    Conversely, let $a,b,c \in G$ be such that $m(a,b,c) = c$. Suppose once more that $a \wedge b = \id$, so $(a \wedge c) \perp (b \wedge c)$ and $c = (a \wedge c) \con \Phi_{a \wedge c}(b \wedge c) = (b \wedge c) \con \Phi_{b \wedge c}(a \wedge c)$. Then 
    \[
        c^{-1} \twist a = \Phi_{a \wedge c}(b \wedge c)^{-1} \con (a \divdot (a \wedge c)), \quad c^{-1} \twist b = \Phi_{b \wedge c}(a \wedge c)^{-1} \con (b \divdot (b \wedge c)).
    \]
    Let $D \in G$ be such that $D \preceq (c^{-1} \twist a) \wedge (c^{-1} \twist b)$ and let $D_1 \coloneqq D \wedge \Phi_{a \wedge c}(b \wedge c)^{-1}, D_2 \coloneqq D \wedge \Phi_{b \wedge c}(a \wedge c)^{-1}$. Then $D_1 \perp \Phi_{b \wedge c}(a \wedge c)^{-1}$ and $\Phi_{\Phi_{b \wedge c}(a \wedge c)^{-1}}(D_1) \preceq (b \wedge c)^{-1} \wedge (b \divdot (b \wedge c))$ which, by the admissibility of $b$, implies that $D_1 = \id$. By a symmetric argument $D_2 = \id$, so $D \perp \Phi_{a \wedge c}(b \wedge c)^{-1}$ and $D \perp \Phi_{b \wedge c}(a \wedge c)^{-1}$. Let $E,F \in G$ be such that 
    \[
        D \vee \Phi_{a \wedge c}(b \wedge c)^{-1} = \Phi_{a \wedge c}(b \wedge c)^{-1} \con E, \quad D \vee \Phi_{b \wedge c}(a \wedge c)^{-1} = \Phi_{b \wedge c}(a \wedge c)^{-1} \con F.
    \]
    Then $(b \wedge c) \con F \preceq b$ and $(a \wedge c) \con E \preceq a$. By Lemma~\ref{spanning a cube} it follows that $\Phi_{(b \wedge c)^{-1}}(F) = \Phi_{(a \wedge c)^{-1}}(E) \preceq a \wedge b$ so $E = F = D = \id$.

    If we don't assume that $a \wedge b = \id$, we have $a \wedge b \preceq c$ by assumption. So let $A,B,C \in G$ be such that $a = (a \wedge b) \con A$, $b = (a \wedge b) \con B$, $c = (a \wedge b) \con C$. Then $(c^{-1} \twist a) \wedge (c^{-1} \twist b) = (C^{-1} \twist A) \wedge (C^{-1} \twist B) = \id$ and $c \in I(a,b)$.
\end{proof}

\begin{definition} \label{def: length function}
    Let $\Lambda$ be a totally ordered abelian group. A $\Lambda$\textit{-length function on $Y$} is a map $\ell:Y \rightarrow \Lambda$ such that 
    \begin{itemize}
        \item $\ell$ is positive definite:  $\ell(\id) = 0$ and $\ell(f) > 0$ for all $f \neq \id$;
        \item $\ell$ is symmetric: $\ell(f^{-1}) = \ell(f)$ for all $f \in Y$;
        \item $\ell(f \con g) = \ell(f) + \ell(g)$ for all $f,g \in Y$.
    \end{itemize}
    If $\Lambda = \mathbb{R}$, we say that $\ell$ is a \textit{length function}.
\end{definition}

Suppose that $Y$ admits a $\Lambda$-length function $\ell$ and define $d:G \times G \rightarrow \mathbb{R}$ by $d(f,g) \coloneqq \ell(f^{-1} \twist g) = \ell(f) + \ell(g) - 2\ell(f \wedge g)$ for all $f,g \in G$.

\begin{proposition} \label{prop: median metric}
    The map $d$ is a $\Lambda$-metric and the $\Lambda$-metric space $(G,d)$ is median.
\end{proposition}
\begin{proof}
Symmetry and positive definitiveness of $d$ follows immediately from the definition of a length function. The triangle inequality follows from median property, which we now prove. Let $a,b,c \in G$ be such that $m(a,b,c) = c$. By Lemma~\ref{lem: G is median preserving} $(c^{-1} \twist a) \wedge (c^{-1} \twist b) = \id$ so $a^{-1} \twist b = (a^{-1} \twist c) \con (c^{-1} \twist b)$. Therefore $d(a,b) = d(a,c) + d(c,b)$ and, by Lemma~\ref{lem: algebra to space}, $(G,d)$ is a median $\Lambda$-metric space with median map $m$.
\end{proof}

\begin{remark}
    It is immediate from the definition of the $\Lambda$-metric that the action of $G$ on itself by left multiplication is by isometries.
\end{remark}

\begin{lemma} \label{lem: rank}
    The rank of $(G,d)$ is 
    \[\mathrm{rk}(G) = \sup\{k \in \mathbb{N} : \exists \; g_1, \dots, g_k \in G - \{\id\} \text{ such that } g_i \perp g_j \; \forall i \neq j\}.\]
\end{lemma}
\begin{proof}
    Let $k \in \mathbb{N}$, let $\sigma_k = \{0,1\}^k$ be a $k$-cube and let $\varphi:\sigma_k \rightarrow G$ be a median-preserving embedding. Up to composing with an element of $G$, we can assume that $\varphi(0) = \id$. For each $i$, let $\chi_i:\{1, \dots, k\} \rightarrow \{0,1\}$ be the characteristic map of $i$. Then $\{\varphi(\chi_i): 1 \leq i \leq k\} \subseteq G - \{\id\}$ are pairwise orthogonal. Conversely, if $\{x_1, \dots, x_k\} \subseteq G - \{\id\}$ are pairwise orthogonal then the set $\{\id, \bigvee_{i \in I} x_i : I \subseteq \{1, \dots, k\}\}$ is an embedded $k$-cube.
\end{proof}

\section{Actions on $\mathbb{R}$-trees with prescribed axis stabilisers} \label{sec: TX construction sum}

We will use the framework established in the previous section to construct groups which act freely and transitively on $\mathbb{R}$-trees. The main results of this section are stated below and will be proven in Section~\ref{sec: centraliser spectrum}. In Section~\ref{sec: reduction} we will show that some of the ores we construct admit retractions to their extracted groups which commute with the operations.

Let $\Sub_{NC}(\mathbb{R})$ denote the set on non-cyclic subgroups of $\mathbb{R}$ and let $\mathcal{K}$ denote the set of cardinals $\kappa$ such that $\kappa \leq 2^{\aleph_0}$. 

\begin{theorem} \label{thm: centraliser spectrum}
    Let $\iota: \Sub_{NC}(\mathbb{R}) \rightarrow \mathcal{K}$ be any map which is supported on $\leq 2^{\aleph_0}$ elements of $\Sub_{NC}(\mathbb{R})$. Then there exists a group $G$ and a free transitive action of $G$ on the complete universal real tree $T$ with valence $2^{\aleph_0}$ such that the following holds.
    For each $H \leq \mathbb{R}$, let $A_H$ be the set of orbits $G \cdot L$ such that $L \subseteq T$ is a line and the action $\Stab_G(L) \acting L$ is isomorphic to $H \acting \mathbb{R}$.
    If $H \in \Sub_{NC}(\mathbb{R})$ then $|A_H| = \iota(H)$.
\end{theorem}

\begin{theorem} \label{thm: small valence action}
    Let $3 \leq \kappa < 2^{\aleph_0}$ be a cardinal. There are no free transitive actions on the complete universal $\mathbb{R}$-tree $T_\kappa$ with valence $\kappa$. 
    
    Let $\kappa \geq 3$ be any cardinal. There exists an incomplete $\mathbb{R}$-tree $S_\kappa$ with valence $\kappa$ and a free transitive action $G \acting S_\kappa$, for some group $G$, if and only if $\kappa$ is either infinite or even. 
    If $\kappa$ is finite and even, then this action is unique: if $S$ is an $\mathbb{R}$-tree with valence $\kappa$, and $H \acting S$ is a free transitive action of a group $H$, then there is a group isomorphism $G \rightarrow H$ and an isometry $S_\kappa \rightarrow S$ which is equivariant relative to $G \rightarrow H$.
\end{theorem}

\subsection{The initial construction} \label{sec: TX construction}

We first construct a family of groups acting freely and transitively on complete real trees with large valence. The groups we construct to prove Theorem~\ref{thm: centraliser spectrum} will be subgroups of these.

Let $\alpha:\mathbb{R} \rightarrow \mathbb{R}$ be the order reversing automorphism of $\mathbb{R}$ defined by $\alpha(\lambda) = - \lambda$ for all $\lambda \in \mathbb{R}$.
Let $X$ be a set equipped with an action of $\mathbb{R} \rtimes_\alpha \la * \ra$, where $*$ is an element of order two. We will abuse notation and identify $\mathbb{R}$ with the normal subgroup $(\mathbb{R}, \id) \unlhd \mathbb{R} \rtimes_\alpha \la \ast \ra$. 

Given $\ell \in \mathbb{R}$ such that $\ell \geq 0$ and a map $f:[0,\ell] \rightarrow X$, let $f^* \coloneqq * \circ f$. The \textit{length} of $f$ is $\ell$ and is denoted by $\ell(f) \coloneqq \ell$.

\begin{definition}[Equivalence, length, identity] \label{defn: equivalent paths} 
Two maps $f:[0,\ell(f)] \rightarrow X$, $g:[0,\ell(g)] \rightarrow X$ are \textit{equivalent}, denoted $f \simeq g$, if $\ell(f) = \ell(g)$ and $|\{s:f(s) \neq g(s)\}| \leq \aleph_0$. The equivalence class of a function $f$ is denoted by $\mathfrak{f}$. The \textit{length} of $\mathfrak{f}$ is $\ell(\mathfrak{f}) \coloneqq \ell(f)$.

The unique equivalence class with length 0 is denoted by $\id$.
\end{definition}

\begin{definition}
    Let $\mathcal{Y}_X$ denote the set of equivalence classes of maps $[0,\ell] \rightarrow X$. Define a binary relation $\preceq$ on $\mathcal{Y}_X$ by: $\mathfrak{\f} \preceq \mathfrak{\g}$ if $\ell(\f) \leq \ell(\g)$ and, for all but countably many $t \in [0,\ell(\f)]$, we have $f(t) = g(t)$. 
\end{definition}

\begin{lemma} \label{lambda trees: median semilattice}
    $(\mathcal{Y}_X,\preceq)$ is a median semilattice with bottom element $\id$.
\end{lemma}
\begin{proof}
    It is immediate from the definition that $\preceq$ is reflexive, antisymmetric and transitive so $(\mathcal{Y}_X, \preceq)$ is a partially ordered set. It is also clear from the definition that $\id$ is a bottom element. Let $S \subseteq \mathcal{Y}_X$ be a non-empty subset and let $\ell \coloneqq \inf\{\ell(\mathfrak{s}) : \mathfrak{s} \in S\}$ and 
    \[m \coloneqq \sup\{\lambda \in  [0,\ell] : \mathfrak{s}|_{[0,\lambda]} \simeq \mathfrak{s}'|_{[0,\lambda]} \; \forall \; \mathfrak{s} \in S\}.\]
    Let $g \coloneqq s|_{[0,m]}$ for some $\mathfrak{s} \in S$. Then the equivalence class $\mathfrak{g}$ of $g$ is independent of the choice of $\mathfrak{s}$ and $\mathfrak{g} \preceq \mathfrak{s}'$ for all $\mathfrak{s}' \in S$. Moreover if $\mathfrak{f} \in \mathcal{Y}_X$ with $\mathfrak{f} \preceq \mathfrak{s}$ for all $\mathfrak{s} \in S$ then the definition of $m$ implies that $\mathfrak{f} \preceq \mathfrak{g}$. Therefore $\mathfrak{g} = \bigwedge S$ is the meet of $S$ and $(\mathcal{Y}_X, \preceq)$ is a meet semilattice.

    Observe that, for any $\mathfrak{f} \in \mathcal{Y}_X$, the set $\{\mathfrak{g} \in \mathcal{Y}_X : \mathfrak{g} \preceq \mathfrak{f}\}$ is totally ordered. It follows that, for any $\mathfrak{f}_1,\mathfrak{f}_2,\mathfrak{f}_3 \in \mathcal{Y}_X$, the set $\{\mathfrak{f}_i \wedge \mathfrak{f}_j :  i \neq j\}$ is totally ordered so, $\mathfrak{f}_k \wedge \mathfrak{f}_m = \bigvee \{\mathfrak{f}_i \wedge \mathfrak{f}_j :  i \neq j\}$ for some $k \neq m$. Let $\g \in \mathcal{Y}_X$ be an arbitrary element. Then, for each $i \neq j$, we have that $g \wedge f_i \wedge f_j \preceq g$ and $g \wedge f_i \wedge f_j \preceq f_i \wedge f_j \preceq f_k \wedge f_m$ so
    \[ \bigvee \{g \wedge f_i \wedge f_j : i \neq j\} \preceq g \wedge f_k \wedge f_m \preceq \bigvee \{g \wedge f_i \wedge f_j : i \neq j\}.\]
    Thus $\bigvee \{g \wedge f_i \wedge f_j\} = g \wedge f_k \wedge f_m$ and $\mathcal{Y}_X$ is median.
\end{proof}

\begin{definition} 
\begin{itemize}
    \item Define $-1:\mathcal{Y}_X \rightarrow \mathcal{Y}_X$ as follows. The $-1$-image of $\mathfrak{f}$ is the equivalence class of the map $f^{-1}:[0,\ell(f)] \rightarrow X$ defined by $f^{-1}(t) = -\ell(f) \cdot f^*(\ell(f)-t)$. 
    \item Define $\con: \mathcal{Y}_X \times \mathcal{Y}_X \rightarrow \mathcal{Y}_X$ as follows. Given $\mathfrak{f}, \mathfrak{g} \in \mathcal{Y}_X$, $\mathfrak{f} \con \mathfrak{g} \in \mathcal{Y}_X$ is the equivalence class of the map $f \con g: [0, \ell(f) + \ell(g)] \rightarrow X$ given by
    \[f \con g(t) = 
    \begin{cases}
        f(t) \quad &\text{if } t \in [0,\ell(f)]; \\
        -\ell(f) \cdot g(t - \ell(f)) &\text{otherwise.}
    \end{cases}
    \]
    \item Let $T_X \subseteq \mathcal{Y}_X$ be the set of admissible elements of $\mathcal{Y}_X$.
\end{itemize}
\end{definition}

\begin{lemma} \label{lem: big tree ore}
    $(\mathcal{Y}_X,\preceq, \id, -1,\con)$ is an ore.
\end{lemma}
\begin{proof}
(O1) and (O4) were proven in Lemma~\ref{lambda trees: median semilattice} and, for all $\f \in \mathcal{Y}_X - \{\id\}$, we have $\f^\perp = \{\id\}$ so (O5) and (O6) are trivially satisfied. It remains to check (O2) and (O3).

\begin{itemize}
    \item[(O2)] Let $\f,\g,\h \in \mathcal{Y}_X$. Then $\ell((\f \con \g) \con \h) = \ell(\f) + \ell(\g) + \ell(\h) = \ell(\f \con (\g \con \h))$. Moreover for all $t \in [0,\ell(\f) + \ell(\g) + \ell(\h)]$ we have
    \begin{align*}
        (f \con g) \con h(t) &= 
        \begin{cases}
            f(t) \quad &\text{if } t \in [0,\ell(\f)] \\
            -\ell(f) \cdot g(t - \ell(f)) &\text{if } t \in (\ell(f), \ell(f) + \ell(g)] \\
            -(\ell(f) + \ell(g)) \cdot h(t - (\ell(f) + \ell(g))) &\text{otherwise}
        \end{cases}
        \\
        &= f \con (g \con h)(t).
    \end{align*}
    Thus $\con$ is associative. Clearly $\id$ is a two-sided identity for $\con$, so $(\mathcal{Y}_X, \con, \id)$ is a monoid.
    
    Let $\f, \g, \h \in \mathcal{Y}_X$ and suppose that $\f \con \g = \f \con \h$. Then clearly $\ell(\g) = \ell(\h)$ and, for all but countably many $t \in (\ell(f), \ell(g)]$, we have $-\ell(\f) \cdot g(t-\ell(\f)) = -\ell(\f) \cdot h(t-\ell(\f))$, so $g(t) = h(t)$ for all but countably many $t \in [0,\ell(\g)]$, which means that $\g = \h$. Similarly, if $\f \con \g = \h \con \g$ then $\ell(\f) = \ell(\h)$ and for all but countably many $t \in [0,\ell(\f)]$ we have $f(t) = h(t)$, so $\f = \h$. Thus $\mathcal{Y}_X$ is cancellative.

    Let $\f \in \mathcal{Y}_X$. Then for all $t \in [0,\ell(f)]$ we have 
    \[ (f^{-1})^{-1}(t) = -\ell(f) \cdot (-\ell(f) \cdot f^*(\ell(f) - (\ell(f) - t)))^* = -\ell(f) \cdot (\ell(f) \cdot f(t)) = f(t).\]
    Therefore $(\f^{-1})^{-1} = \f$ and $-1: \mathcal{Y}_X \rightarrow \mathcal{Y}_X$ is an involution.
    Let $\f,\g \in \mathcal{Y}_X$. Then $\ell((\f \con \g)^{-1}) = \ell(\f \con \g) = \ell(\f) + \ell(\g) = \ell(\g \con \f)$ and for all $t \in [0,\ell(\f) + \ell(\g)]$ we have 
    \begin{align*} 
    (f \con g)^{-1}(t) &= -(\ell(f) + \ell(g)) \cdot (f \con g)^*(\ell(f) + \ell(g) - t) \\
    &=
    \begin{cases}
        -(\ell(f) + \ell(g)) \cdot (-\ell(f) \cdot g(\ell(g) - t))^* \quad &\text{if } t \in [0,\ell(g)) ; \\
        -(\ell(f) + \ell(g)) \cdot f^*(\ell(f) + \ell(g) - t) &\text{otherwise;}
    \end{cases}
     \\ &=
     \begin{cases}
         -\ell(g) \cdot g^*(\ell(g) - t) \hspace{3.4cm} &\text{if } t \in [0,\ell(g)) ; \\
         -\ell(g) \cdot (-\ell(f) \cdot f^*(\ell(f) - (t - \ell(g)))) &\text{otherwise}
     \end{cases}
     \\
     &= g^{-1} \con f^{-1} (t), \text{ unless }t = \ell(\g).
     \end{align*}
     Therefore $(\f \con \g)^{-1} = \g^{-1} \con \f^{-1}$. This completes the proof of (O2).
    \item[(O3)]  Let $\mathfrak{f},\mathfrak{g} \in \mathcal{Y}_X$. It is clear from the definition of $\con$ that if there exists $\mathfrak{h} \in \mathcal{Y}_X$ such that $\f = \g \con \h$ then $\g \preceq \f$. Conversely suppose that $\g \preceq \f$. Define $h:[0,\ell(f) - \ell(g)] \rightarrow X$ by $h(t) = \ell(g) \cdot f(t + \ell(g))$. Then $\f = \g \con \h$. \qedhere
\end{itemize}
\end{proof}

Let $(T_X, \twist)$ be the group extracted from $\mathcal{Y}_X$ and define $d:T_X \times T_X \rightarrow \mathbb{R}$ by $d(\mathfrak{f},\mathfrak{g}) = \ell(\mathfrak{f}^{-1} \twist \mathfrak{g})$.

\begin{lemma} \label{lem: complete R tree}
    $(T_X,d)$ is a complete $\mathbb{R}$-tree.
\end{lemma}
\begin{proof}
    By Lemma~\ref{prop: median metric}, $T_X$ is a median metric space and it is clear from the construction that $T_X$ is connected. For all $\f \in \mathcal{Y}_X - \{\id\}$, we have $\f^{\perp} = \{\id\}$, so  Lemma~\ref{lem: rank} implies that $T_X$ has rank 1. Thus $T_X$ is an $\mathbb{R}$-tree by \cite[Lemma~15.1.2]{Bowditch}; let us show that it is complete.
    Let $(\f_n)_{n \in \mathbb{N}} \subseteq T_X$ be a Cauchy sequence. Then the sequence $(\ell(\f_n))_{n \in \mathbb{N}} \subseteq \mathbb{R}$ is Cauchy and has a limit $\ell \in \mathbb{R}$. Define a map $f:[0,\ell] \rightarrow \mathbb{R}$ as follows. If $t < \ell$ then there exists $N_t \in \mathbb{N}$ such that, for all $n \geq N_t$, we have $\ell(\f_n \wedge \f_{N_t}) > t$. Let $f(t) = f_{N_t}(t)$ for all such $t$ and let $f(\ell)$ be arbitrary. Then $\f_n \rightarrow \f$ in $T_X$. 
\end{proof}

\begin{remark}
    If $2 \leq |X| \leq 2^{\aleph_0}$ then the valence (and the cardinality) of $T_X$ is $2^{2^{\aleph_0}}$.  
\end{remark}

The following characterisation of admissibility for elements of $\mathcal{Y}_X$ will be useful.

\begin{lemma} \label{lem: characterise admissiblity}
    Let $\ell \geq 0$ and $f:[0,\ell] \rightarrow X$ be a map.
    The equivalence class of $f$ is admissible if and only if either $\ell = 0$ or, for each $t \in (0,\ell)$ and every non-degenerate interval $[t-s, t+s] \subseteq [0,\ell(f)]$, there exist uncountably many $0 < \varepsilon < s$ such that $f^*(t-\varepsilon) \neq 2t \cdot f(t+\varepsilon)$.
\end{lemma}

\subsection{Axes}

\begin{definition} \label{axis defn}
    Let $G$ be a group acting by isometries on a metric space $(Y,d_Y)$. An \textit{axis} in $Y$ is an isometric embedding $L: \mathbb{R} \rightarrow Y$ such that the stabiliser $\Stab_G(L(\mathbb{R}))$ acts coboundedly on $L(\mathbb{R})$. We will also call the image of such a map an \textit{axis}.
\end{definition}

Let $X, \mathbb{R}, \mathcal{Y}_X, T_X$ be as in Section~\ref{sec: TX construction}.

\begin{definition}
An element $\f \in \mathcal{Y}_X$ is \textit{constant} if there is a representative of $\f$ which is a constant map. Given a constant element $\f \in \mathcal{Y}_X$, we will always assume that the representative $f$ is a constant map. The \textit{image} of $\f$ is the image of $f$.
\end{definition}

\begin{lemma} \label{lem: admissible constants}
    Let $x \in X$ be such that $x$ and $x^*$ are in different $\mathbb{R}$-orbits. Then every constant element with image $x$ is admissible.
\end{lemma}
\begin{proof}
    Let $f:[0,\ell] \rightarrow X$ be a constant map to $x \in X$ and suppose that $f$ is inadmissible. Then, by Lemma~\ref{lem: characterise admissiblity}, there exists $\lambda \in [0,\ell]$ and $s > 0$ such that $[\lambda - s,\lambda + s] \subseteq [0,\ell]$ and for uncountably many $0 < \varepsilon < s$ we have $f^*(\lambda-\varepsilon) = 2\lambda \cdot f(\lambda + \varepsilon)$. But then $x^* = 2\lambda \cdot x$. 
\end{proof}

\begin{definition} \label{def: standard axis}
    Let $x \in X$ be such that $x$ and $x^*$ are in different $\mathbb{R}$-orbits. Define a subspace 
    \[L_x \coloneqq \{\f \in T_X : f(t) = x \; \forall \; t \in [0,\ell(f)]\}.\]
    The $x$\textit{-axis} of $T_X$ is the subspace $\mathbf{L}_x \coloneqq L_x \cup L_{x^*}$. A subspace $L \subseteq T_X$ is a \textit{standard axis} if it is an $x$-axis for some $x \in X$ such that $\Stab_\mathbb{R}(x)$ is non-trivial and $x^* \notin \mathbb{R} \cdot x$.
\end{definition}

The lemma below follows immediately from the relevant definitions.

\begin{lemma} \label{lem: stab of standard axis}
     Let $x \in X$ be such that $x$ and $x^*$ are in different $\mathbb{R}$-orbits. Then
     \[ 
     	\Stab_{T_X}(\mathbf{L}_x) = \{ \f \in \mathbf{L}_x : \ell(\f) \in \Stab_{\mathbb{R}}(x)\} \cong \Stab_{\mathbb{R}}(x).
     \]
     Moreover, if $\varphi: \mathbf{L}_x \rightarrow \mathbb{R}$ is the map defined by $\varphi(\f) = \ell(\f)$ if $\f \in L_x$ and $\varphi(\f) = -\ell(\f)$ otherwise, then $\varphi$ is a $\Stab_{T_X}(\mathbf{L}_X)$-equivariant isometry.
\end{lemma}

\begin{corollary}
    Standard axes in $T_X$ are axes in the sense of Definition~\ref{axis defn}.
\end{corollary}

\begin{remark} \label{rem: weird axes}
Clearly `most' lines in $T_X$ will have trivial stabilisers. Moreover it is not difficult to construct non-standard axes with cyclic stabilisers and even such that the generator of the stabiliser acts on its axis with arbitrary translation length. 

A more surprising observation is that it is also possible for a non-standard axis of $T_X$ to have a dense stabiliser. Suppose that, for some uncountable proper subgroup $H \lneq \mathbb{R}$ and $x,y \in X$, the orbits $\mathbb{R} \cdot x, \mathbb{R} \cdot x^*, \mathbb{R} \cdot y, \mathbb{R} \cdot y^*$ are pairwise disjoint and $\Stab_\mathbb{R}(x) = \Stab_\mathbb{R}(y) = H$. Let $L \subseteq T_X$ be the set of elements of the form $f:[0,\ell] \rightarrow X$ such that $f(t) = x$ if $t \in H$ and $f(t) = y$ otherwise and let $\mathbf{L} \coloneqq L \cup L^*$. The fact that $H$ is uncountable implies that $\mathbf{L}$ is not standard, yet $\Stab_{T_X}(\mathbf{L}) = \{\f \in L : \ell(\f) \in H\} \cong H$. This behaviour will disappear when we pass to more sensible subgroups of $T_X$.
\end{remark}

\subsection{Templates} \label{sec: templates}

We can now introduce the subgroup of $T_X$ which will be used to prove Theorem~\ref{thm: centraliser spectrum}.

Let $S \subseteq T_X$ be a set. The closed subgroup $\overline{\la S \ra}$ generated by $S$ is the smallest closed subgroup of $T_X$ containing $S$, with respect to the topology induced by the metric $d$.

\begin{lemma} \label{lem: connected subgroups}
Let $S \subseteq T_X$ be a symmetric set which is closed under restriction (i.e. $\mathfrak{s}^{-1} \in S$ for all $\mathfrak{s} \in S$ and, if the equivalence class of $s: [0,\ell] \rightarrow X$ is in $S$, then the equivalence class of the restriction $s|_{[0,t]}$ is in $S$ for all $t \in [0,\ell]$). Then $\la S \ra$ and $\overline{\la S \ra}$ are connected.
\end{lemma}
\begin{proof}
    Let $H \leq T_X$ be a subgroup and $\f,\g \in H$ be elements which are connected to $\id$ via paths $\gamma_\f,\gamma_\g: [0,1] \rightarrow H$. Then the composition $\f \star \gamma_\g$ is a path from $\f$ to $\f \star \g$ and the concatenation of $\gamma_\f$ with $\f \star \gamma_\g$ is a path from $\id$ to $\f \star \g$. It follows that $\la S \ra$ is connected. Now suppose that $(\f_n)_{n \in N} \subseteq H$ is a sequence which converges to $\f \in H$ such that each $\f_n$ is connected to $\id$ via a path in $H$. Then we can assume without loss of generality that $\f_n \preceq \f_{n+1}$ for all $n \in \mathbb{N}$ and fix geodesics $\gamma_n: [0,\ell(\f_n)] \rightarrow H$ connecting $\id$ to $\f_n$. Let $\gamma(\ell(\f)) = \f$ and, if $0 \leq t < \ell(f)$, let $\gamma(t) = \gamma_n(t)$ for some (equivalently any) $n \in \mathbb{N}$ such that $\ell(\f_n) > t$. Then $\gamma: [0, \ell(\f)] \rightarrow H$ is a path connecting $\id$ to $\f$. It follows that $\overline{\la S \ra}$ is connected.
\end{proof}

\begin{definition}    
    Given a subset $Y \subseteq X$, let $T_X(Y) \coloneqq \overline{\la S \ra}$, where $S = \cup_{y \in Y} L_y$.
\end{definition}

We will show that $T_X(Y)$ is the universal real tree with valence $2^{\aleph_0}$ whenever $2 \leq |\mathbb{R} \cdot Y| \leq 2^{\aleph_0}$. To do this, we will first characterise the elements of $T_X(Y)$ using \textit{templates}. We start by showing that $T_X(Y) = T_X(\mathbb{R} \cdot Y) = T_X(Y^*) = T_X(\mathbb{R} \cdot Y^*)$.

\begin{lemma} \label{lem: orbit closure}
    Let $Y \subseteq X$ be non-empty and suppose that $\Stab_\mathbb{R}(y) \neq \{0\}$ and $y^* \notin \mathbb{R} \cdot y$ for all $y \in Y$. If $H \leq T_X$ is a closed subgroup containing $L_y$ for each $y \in Y$, then $H$ contains $L_{\lambda \cdot y}$ and $L_{\lambda \cdot y^*}$ for each $\lambda \in \mathbb{R}$.
\end{lemma}
\begin{proof}
     Fix $y \in Y$ and, for each $\ell > 0$, let $\f_\ell \in L_y$ denote the constant element with length $\ell$ and image $y$. Suppose $s \in \Stab_\mathbb{R}(y)$ and $s > 0$. Then $\f_s^{-1} \in H$ has length $s$ and $f_s^{-1}(t) = (-s) \cdot y^* = (s \cdot y)^* = y^*$ for any $t \in [0,s]$. Given any $\ell > 0$ let $s,t \geq 0$ be such that $s \in \Stab_\mathbb{R}(y)$ and $\ell = s-t$. Then $\f_s^{-1} \star \f_t = \f_\ell^* \in L_{y^*}$ so $L_{y*} \subseteq H$. 
    Now let $\lambda \in \mathbb{R}$ and suppose that $\lambda \geq 0$. For each $\ell \geq 0$, let $\g_\ell \coloneqq \f_\lambda^{-1} \star \f_{\lambda + \ell}$. 
    Since $\f_{\lambda + \ell} = \f_\lambda \con \lambda \cdot \f_\ell$, the element $\g_\ell \in H$ has length $\ell$ and $g_\ell(t) = \lambda \cdot y$ for all $t \in [0,\lambda]$. Thus $L_{\lambda \cdot y} \subseteq H$. Suppose that $\lambda < 0$. For any $\ell > 0$, if $g_\ell' \coloneqq (\f^*_{-\lambda})^{-1} \twist \f^*_{\ell - \lambda}$ then $g_\ell'$ has length $\ell$ and $g_\ell'(t) = (-\lambda) \cdot y^* = (\lambda \cdot y)^*$. Thus $L_{(\lambda \cdot y)^*} \subseteq H$ and by the argument above this implies that $L_{\lambda \cdot y} \subseteq H$.
\end{proof}

\begin{definition}[Templates] \label{templates}
Let $\ell \in \mathbb{R}$ with $\ell \geq 0$. A countable set $P \subseteq [0,\ell]$ is a \textit{template} of $[0,\ell]$ if it satisfies the following.
\begin{enumerate}
    \item[(T1)] If $p \in P$ and $p > 0$ then there exist $p_1 \in P$ -- called the \textit{predecessor} of $p$ -- such that $p_1 < p$ and $P \cap (p_1,p) = \emptyset$. If $p \in P$ and $p < \ell$ then there exists $p_2 \in P$ -- called the \textit{successor} of $p$ -- such that $p < p_2$ and $P \cap (p,p_2) = \emptyset$.
    \item[(T2)] The union $\cup \{[p,p'] : p,p' \in P $ and $ P \cap (p,p') = \emptyset\}$ has countable complement in $[0,\ell]$.
\end{enumerate}
Given a template $P \subseteq [0,\ell]$, the \textit{inverse} of $P$ is the template $P^{-1} \coloneqq \{\ell - p : p \in P\}$.
\end{definition}
Note that (T2) implies that every point in the complement is an accumulation point of $P$ and that the closure $\overline{P}$ is countable.

Let $\ell \geq 0$ and $P \subseteq [0,\ell]$ be a template. There are two related but slightly different ways of ``filling in" a template to produce an element of $\mathcal{Y}_X$ or $T_X$. In the first version, a sequence of elements of $X$, indexed by $P$, directly defines a map. In the second, the input is a suitable sequence of elements of $\mathcal{Y}_X$, also indexed by $P$, and the map they define should be viewed as an infinite concatenation (so the action of $\mathbb{R}$ must be taken into account). 

\begin{definition}[Realising sequences]
\begin{itemize} 
\item[(i)] Let $(x_p)_{p \in P}$ be a sequence in $X$. The \textit{realisation} of $(x_p)_{p \in P}$ is the equivalence class of any function $f:[0,\ell] \rightarrow X$ such that $f(t) = x_p$ if $t \in (p,p')$ for some $p,p' \in P$ such that $P \cap (p,p') = \emptyset$. We say that $(x_p)_{p \in P}$ is \textit{admissible} if $\mathfrak{f} \in T_X$.
\item[(ii)] A sequence $(\f_p)_{p \in P} \subseteq \mathcal{Y}_X$ is \textit{consistent} if the following holds. For each $p \in P - \{\ell\}$, if $p' \in P$ is the successor of $p$, then $\ell(\f_p) = p'-p$ and if $\ell \in P$ then $\f_\ell = \id$. The \textit{concatenation} of $(\f_p)_{p \in P}$ is the equivalence class of any function $f:[0,\ell] \rightarrow X$ such that $f(t) = -p \cdot f_p(t-p)$ if $t \in (p,p')$ for some $p,p' \in P$ such that $P \cap (p,p') = \emptyset$. We say that $(\f_p)_{p \in P}$ is \textit{admissible} if $\mathfrak{f} \in T_X$.
\end{itemize}
\end{definition}

\begin{definition}[Template ore]
Let $\mathcal{Z}_X$ be the set of realisations of sequences $(x_p)_{p \in P} \subseteq X$, where $P \subseteq [0,\ell]$ is a template and $\ell \geq 0$. The set $\mathcal{Z}_X \subseteq \mathcal{Y}_X$ is called the \textit{template ore} over X. 
\end{definition}

One checks easily that $\mathcal{Z}_X$ is closed under the operations $\con$ and $-1$, and under the relation $\preceq$. Moreover the element $\id$ is the realisation of any sequence $(x_0)$ so $\id \in \mathcal{Z}$. Therefore the following lemma follows immediately from Lemma~\ref{lem: big tree ore}.

\begin{lemma} \label{lem: template ore}
    $(\mathcal{Z}_X, \preceq, \id, \con, -1)$ is an ore.
\end{lemma}

\begin{definition}[Complexity] \label{def: complexity}
Given $\f \in \mathcal{Z}_X$, the \textit{complexity} of $\f$ is the smallest countable ordinal $\alpha$ such that $\f$ is the realisation of a sequence in $X$ indexed over a template $P$ whose closure has CB-rank $\alpha$. For each countable ordinal $\alpha$, we denote by $\mathcal{Z}_X^{[\alpha]}$ the set of elements of $\mathcal{Z}_X$ with complexity $\leq \alpha$.
\end{definition}

\begin{remark}
The closure of a template is a countable compact Polish space, so its CB-rank is a successor ordinal (see Remark~\ref{rem: CB-rank succesor}). Therefore the complexity of an element of $\mathcal{Z}_X$ is always a successor ordinal. We will use this liberally from now on.
\end{remark}

\begin{lemma}[Inversion] \label{lem: template inverse}
Let $\f \in \mathcal{Z}_X$ be the realisation (resp. concatenation) of a sequence in $X$ (resp. in $\mathcal{Y}_X$) indexed over a template $P \subseteq [0,\ell(\f)]$. Then $\f^{-1}$ is the realisation (resp. concatenation) of a sequence in $X$ (resp. in $\mathcal{Y}_X$) indexed over $P^{-1}$.
\end{lemma}
\begin{proof}
 Let $(x_p)_{p \in P} \subseteq X$ be a sequence with realisation $\f$.
 For each $p' = \ell - q \in P^{-1} - \{\ell\}$, let $p \in P$ be the predecessor of $q \in P$, and let $z_{p'} \coloneqq - \ell \cdot x_p^*$. If $\ell \in P^{-1}$ then let $z_\ell \in X$ be an arbitrary element.
 Then $\f^{-1}$ is the realisation of $(z_{p'})_{p' \in P^{-1}}$.
 
Let $(\f_p)_{p \in P} \subseteq \mathcal{Y}_X$ be a consistent sequence with realisation $\f$. 
For each $p' = \ell - q \in P^{-1} - \{\ell\}$, where $q \in P$ is the successor of $p \in P$, let $\g_{p'} \coloneq \f_p^{-1}$. If $\ell \in P^{-1}$ then let $\g_\ell \coloneqq \id$.
Then $\f^{-1}$ is the realisation of $(\g_{p'})_{p' \in P^{-1}}$.
\end{proof}

\begin{lemma}[Refining] \label{lem: refining}
    Let $P \subseteq [0,\ell]$ be a template, let $(x_p)_{p \in P} \subseteq X$ be a sequence and let $(\f_p)_{p \in P} \subseteq T_X$ be consistent. For any $t \in [0, \ell]$, there exists a template $Q \subseteq [0,\ell]$, a sequence $(z_q)_{q \in Q} \subseteq X$ and a consistent sequence $(\g_q)_{q \in Q} \subseteq T_X$ such that: 
    \begin{itemize}
        \item $t$ is an accumulation point of $Q$ (in particular, $t \notin Q$); 
        \item $\overline{Q} - Q = (\overline{P} - P) \cup \{t\}$;
        \item the realisation of $(z_q)_{q \in Q}$ is equivalent to the realisation of $(x_p)_{p \in P}$;
        \item the concatenation of $(\g_q)_{q \in Q}$ is equivalent to the concatenation of $(\f_p)_{p \in P}$.
    \end{itemize}
\end{lemma}

The process of replacing $P$ with $Q$ and either replacing $(x_p)_{p \in P}$ with $(z_q)_{q \in Q}$ or $(\f_p)_{p \in P}$ with $(\g_q)_{q \in Q}$ is called \textit{refining}. The fact that $\overline{P} - P \subseteq \overline{Q} - Q$ ensures that, by repeating this process a finite number of times, $P$ can be refined to admit any finite set of points in $[0,\ell]$ as accumulation points. 
\begin{proof} 
    Suppose that $t$ is not an accumulation point of $P$ and $t \notin \{0,\ell\}$. Then there exists $p_1, p_2 \in P$ such that $p_1 < t < p_2$ and $P \cap (p_1, p_2) \subseteq \{t\}$. Let $(t_n)_{n \in \mathbb{N}} \subseteq (p_1, t)$ be a strictly increasing sequence which converges towards $t$, let $(s_n)_{n \in \mathbb{N}} \subseteq (t, p_2)$ be a strictly decreasing sequence which converge towards $t$ and let $Q \coloneqq (P - \{t\}) \cup \{t_n, s_n: n \in \mathbb{N}\}$. Then $Q$ is a template with non-trivial accumulation point $t$ and $\partial Q = \partial P \cup \{t\}$.
    
    For each $p \in P - \{t\}$ let $z_p \coloneqq x_p$ and for each $n \in \mathbb{N}$ let $z_{t_n} \coloneqq x_{p_1}$ and $z_{s_n} \coloneqq x_{t}$ if $t \in P$ and $z_{s_n} \coloneqq x_{p_1}$ otherwise. Then the realisation of $(z_q)_{q \in Q}$ is equivalent to the realisation of $(x_p)_{p \in P}$. For each $p \in P - \{p_1,t\}$ let $g_p \coloneqq f_p$. Let $g_{p_1} \coloneqq f_{p_1}|_{[0,t_1-p_1]}$ and let $g_{s_1}: [0,p_2-s_1] \rightarrow X$ be defined by $g_{s_1}(s) = (s_1-t) \cdot f_t(s+s_1-t)$ if $t \in P$ and $g_{s_1}(s) = (s_1-p_1) \cdot f_{p_1}(t+s_1-p_1)$ otherwise. For each $n \in \mathbb{N}$ let $g_{t_n}:[0,t_{n+1}-t_n] \rightarrow X$ be defined by $g_{t_n}(s) = (t_n-p_1) \cdot f_{p_1}(s+t_n-p_1)$ and, if $n \geq 2$, let $g_{s_n}:[0,s_{n-1}-s_{n}] \rightarrow X$ be defined by $g_{s_n}(s) = (s_n-t) \cdot f_t(s+s_{n}-t)$ if $t \in P$ and $g_{s_n}(s) = (s_n - p_1) \cdot f_{p_1}(s+s_{n}-p_1)$ otherwise. Then $(\g_q)_{q \in Q}$ is consistent and its concatenation is equivalent to the concatenation of $(\f_p)_{p \in P}$.
    
    If $t = 0$ or $\ell$ then perform the above operation but only on one side of $t$.
\end{proof}

\begin{proposition} \label{prop: template generation}
    Let $Y \subseteq X$ be such that $\Stab_\mathbb{R}(y) \neq \{0\}$ and $y^* \notin \mathbb{R} \cdot y$ for all $y \in Y$. Then $T_X(Y)$ is the set of realisations of admissible sequences $(y_p)_{p \in P}$ in $\mathbb{R} \cdot (Y \cup Y^*)$, where $P$ is a template.
\end{proposition}
\begin{proof}
    We start by showing that the set of all such realisations is a closed subgroup. Lemma~\ref{lem: template inverse} implies that it is closed under taking inverses. Let $\ell, \ell' > 0$, let $P \subseteq [0,\ell]$, $P' \subseteq [0, \ell']$ be templates and let $(y_p)_{p \in P}, (y_{p'}')_{p' \in P'} \subseteq \mathbb{R} \cdot (Y \cup Y^*)$ be admissible sequences with realisations $\f,\f'$ respectively. Let $\tau \coloneqq \ell(\f \dotdiv (\f^{-1} \wedge \f')^{-1})$. 
    By refining $P,P'$ and $(y_p)_{p \in P},(y_{p'}')_{p' \in P'}$ if necessary, we can assume that $\tau$ is a non-trivial accumulation point of $P$ and $\ell - \tau$ is a non-trivial accumulation point of $P'$.
     Let $P'' \coloneqq (P \cap [0,\tau]) \cup ([\ell - \tau,\ell(\f \star \f')] \cap \{p' - \ell + 2\tau : p' \in P'\})$. Then $P''$ is a template of $[0, \ell(\f \star \f')]$. For each $p \in P \cap P''$, let $x_p \coloneqq y_p$ and for each $p' - \ell + 2\tau\in P'' - P$ let $x_{p'-\ell+2\tau} \coloneqq (\ell - 2\tau) \cdot y_{p'}'$. Then $(x_{p''})_{p'' \in P''}$ is admissible and the realisation of this sequence is $\f \star \f'$. 

     \begin{observation} \label{obs1}
         If $\f,\g \in T_X$ are such that $\g \preceq \f$ and $\f$ is the realisation of a sequence $(x_p)_{p \in P} \subseteq \mathbb{R} \cdot (Y \cup Y^*)$, where $P \subseteq [0,\ell(\f)]$ is a template, then $\g$ is the realisation of a sequence in $\mathbb{R} \cdot (Y \cup Y^*)$ indexed over a template. Indeed, up to refining $P$ and $(x_p)_{p \in P}$, we can assume that $\ell(\g) \in \overline{P}$. Let $P' \coloneqq \{p \in P : p \leq \ell(\g)\}$. Then $\g$ is the realisation of $(x_p)_{p \in P'}$. 
     \end{observation}
     
     Now let $(\f_n)_{n \in \mathbb{N}} \subseteq T_X$ be a sequence converging to a point $\f \in T_X$ such that, for each $n$, if $\ell_n \coloneqq \ell(\f_n)$ then there is a template $P_n \subseteq [0,\ell_n]$ and an admissible sequence $y^{(n)} = (y_p^{(n)})_{p \in P_n} \subseteq \mathbb{R} \cdot (Y \cup Y^*)$ such that $\f_n$ is the realisation of $y^{(n)}$. If $\ell \coloneqq \ell(\f)$ then $(\ell_n)_{n \in \mathbb{N}}$ converges to $\ell$. Take a subsequence of $(\f_n)_{n \in \mathbb{N}}$ so that $(\ell_n)_{n \in \mathbb{N}}$ is either strictly increasing or strictly decreasing. If $(\ell_n)_{n \in \mathbb{N}}$ is strictly decreasing then $\f \preceq \f_n$ for almost all $n$, so $\f$ is the realisation of some sequence in $\mathbb{R} \cdot (Y \cup Y^*)$ by Observation~\ref{obs1}.
     So suppose that $(\ell_n)_{n \in \mathbb{N}}$ is strictly increasing. By Observation~\ref{obs1} again, we can replace each $\f_n$ with a $\preceq$-smaller element so that $\f_n \preceq \f_{n+1}$ for each $n$. For each $n \in \mathbb{N}$, refine $P_n$ and $y^{(n)}$ finitely many times so that $\ell_m$ is a limit point of $P_n$ for all $m \leq n$. Define
     \[P \coloneqq \bigcup_{n \in \mathbb{N}} P_n \cap (\ell_{n-1}, \ell_n).\]
     Then $P$ is a template of $[0,\ell]$. 
     For each $n \in \mathbb{N}$ and each $p \in P \cap (\ell_{n-1},\ell_n)$ let $z_p \coloneqq y_p^{(n)}$. Then $(z_p)_{p \in P}$ is an admissible sequence whose realisation is $\f$.
     
    By Lemma~\ref{lem: orbit closure}, the fact that all realisations of admissible sequences in $\mathbb{R} \cdot (Y \cup Y^*)$ are in $T_X(Y)$ follows from Lemma~\ref{lem: sequence closure} below.
\end{proof}

\begin{lemma} \label{lem: sequence closure}
    Let $H \leq T_X$ be a closed subgroup. Let $\ell > 0$ and let $P \subseteq [0,\ell]$ be a template. Suppose that $(\h_p)_{p \in P} \subseteq H$ is a consistent and admissible sequence and let $\f$ be the concatenation of $(\h_p)_{p \in P}$. Then $\f \in H$.
\end{lemma}
\begin{proof}
    We prove the lemma by transfinite induction on the Cantor--Bendixson rank $\alpha$ of $\overline{P}$.
    
    If $\alpha=1$, then $\overline{P}$ is a compact set of isolated points and is therefore finite. Therefore $\f$ is a finite concatenation of elements in $H$ and is itself in $H$.

    Suppose that $\alpha > 1$ and recall that $\alpha = \beta + 1$ for some countable ordinal $\beta \geq 1$.
    Let $K \coloneqq \overline{P}^{(\beta)}$. Then $K$ is a compact set of isolated points so $K$ is finite.
    Let $k_0, \dots, k_m$ be the elements of $K \cup \{0, \ell\}$, ordered such that $0 = k_0 < k_1 < \dots < k_m = \ell$. For each $0 \leq i < m$ choose a point $k_i^+ = k_{i+1}^- \in (k_i, k_{i+1}) \cap P$ (this exists by (T2) because $K \cap P = \emptyset$). Let $f_{k_i}^+: [0,k_i^+-k_i] \rightarrow X$ be defined by $f_{k_i}^+(t) = k_i \cdot f(t+k_i)$ and let $f_{k_{i+1}}^-: [0,k_{i+1}-k_{i+1}^-] \rightarrow X$ be defined by $f_{k_{i+1}}^-(t) = k_{i+1}^- \cdot f(t+k_{i+1}^-)$.
    If $\f_{k_i}^-,\f_{k_{i+1}}^+ \in H$ for each $0 \leq i < m$ then $\f = \f_{k_0}^+ \star \f_{k_1}^- \star \f_{k_1}^+ \star \dots \star \f_{k_{m-1}}^+ \star \f_{k_m}^- \in H$. It therefore remains to show that $\f_{k_i}^-,\f_{k_{i+1}}^+$ are indeed elements of $H$ for each $i$.
    
    Fix $k \in \{k_i : 0 < i \leq m\}$ and let $Q \coloneqq \{p-k^- : p \in P \cap [k^-,k)\}$. Then $Q$ is a template of $[0,k - k^-]$ and $0 \in Q$. If $k = \ell$ and $\ell \notin K$ then the CB-rank of $\overline{Q}$ is at most $\beta$ and $\f_k^-$ is the concatenation of $(\h_p)_{p-k^- \in Q}$ so $\f_k^- \in H$ by the induction hypothesis. So assume that $k \in K$. Then $\overline{Q}^{(\beta)} = \{k - k^-\}$ and there exists a strictly increasing sequence $(q_n)_{n \in \mathbb{N}} \subseteq Q$ which converges towards $k-k^-$ (because $\beta \geq 1$). For each $n$, let $Q_n \coloneqq Q \cap [0,q_n]$, let $\beta_n$ be the CB-rank of $\overline{Q}_n$ and let $g_n \coloneqq f_k^-|_{[0,q_n]}$. Then $Q_n$ is a template of $[0,q_n]$, $\beta_n \leq \beta$ and $\g_n$ is the concatenation of $(\h_p)_{p - k^- \in Q_n}$. By the induction hypothesis, $\g_n \in H$ for each $n$. Moreover $(\g_n)_{n \in \mathbb{N}}$ converges towards $\f_k^-$ so $\f_k^- \in H$. By a similar argument $(\f_k^+)^{-1} \in H$ for all $k \in \{k_i : 0 \leq i < m\}$ and therefore $\f_k^+ \in H$ for all such $k$.
\end{proof}

\begin{proposition} \label{prop: universal real tree}
    Let $Y \subseteq X$ be such that each $y \in Y$ has a non-trivial $\mathbb{R}$-stabiliser and $y^* \notin \mathbb{R} \cdot y$ for all $y \in Y$.
    Suppose that $|Y| \leq 2^{\aleph_0}$ and there exist $y_1, y_2 \in Y$ such that $y_1 \neq y_2$ and $y_1 \notin \mathbb{R} \cdot y_2^*$. Then $(T_X(Y),d)$ is the universal $\mathbb{R}$-tree with valence $2^{\aleph_0}$.
\end{proposition}
\begin{proof}
    By Lemmas~\ref{lem: complete R tree} and \ref{lem: connected subgroups}, $T_X(Y)$ is a closed connected subspace of a complete $\mathbb{R}$-tree, so it is itself a complete $\mathbb{R}$-tree. Since the action of $T_X(Y)$ on itself is transitive, it suffices to check that the valence of $T_X(Y)$ at $\id$ is $2^{\aleph_0}$. Let $\kappa$ denote the valence of $T_X(Y)$ at $\id$. Two elements $\f,\g \in T_X(Y)$ lie in the same connected component of $T_X(Y) - \{\id\}$ if and only if there exists $\varepsilon > 0$ such that $f(t) = g(t)$ for all but countably many $t \in [0,\varepsilon]$. Let $\ell \in \mathbb{R}$, $\ell > 0$ and let $P = \{p_n \in (0,\ell] : n \in \mathbb{N}\}$ be a template of $[0,\ell]$ such that $p_n \rightarrow 0$. For each subset $\Omega \subseteq \mathbb{N}$ let $x^\Omega = (x^\Omega_{p_n})_{p_n \in P} \subseteq X$ where $x^\Omega_{p_n} = y_1$ if $n \in \Omega$ and $x^\Omega_{p_n} = y_2$ otherwise. It follows from Lemmas~\ref{lem: characterise admissiblity} and \ref{lem: admissible constants}, and the assumptions on $y_1,y_2$, that the sequence $x^\Omega$ and its realisation $\mathfrak{f}_\Omega$ are admissible. Moreover, if $\Omega, \Omega' \subseteq \mathbb{N}$ and the symmetric difference $\Omega \Delta \Omega'$ is infinite then $\f_\Omega$ and $\f_{\Omega'}$ lie in different connected components of $T_X(Y) - \{\id\}$. Therefore $\kappa \geq \{0,1\}^\mathbb{N} / \sim$, where $\Omega \sim \Omega'$ if and only if $|\Omega \Delta \Omega'| < \infty$. Each equivalence class of subsets has cardinality $\aleph_0$ so this implies that $\kappa \geq 2^{\aleph_0}$. Conversely, Proposition~\ref{prop: template generation} implies that cardinality of $T_X(Y)$ is bounded above by that of $\mathcal{C} \times (\mathbb{R} \cdot(Y \cup Y^*))$, where $\mathcal{C}$ is the set of countable subsets of $\mathbb{R}$. Thus $\kappa \leq |T_X(Y)| \leq 2^{\aleph_0}$.
\end{proof}

We can now prove that the only lines in $T_X(Y)$ which can have dense stabilisers are translates of standard axes.

\begin{lemma} \label{lem: controling axis stabilisers}
    Let $Y \subseteq X$ be such that each $y \in Y$ has a non-trivial $\mathbb{R}$-stabiliser and $y^* \notin \mathbb{R} \cdot y$ for all $y \in Y$.
    Let $L \subseteq T_X(Y)$ be a line such that $\Stab_G(L)$ acts on $L$ with dense orbits. Then $L = \f \star L'$ for some $\f \in T_X(Y)$ and some standard axis $L' \subseteq T_X(Y)$.
\end{lemma}
\begin{proof}
    Up to translating by an element of $T_X(Y)$, we can assume that $L$ intersects the identity. Let $\f \in L$ be a non-trivial element and let $P \subseteq [0,\ell(\f)]$ be a template such that $\f$ is the realisation of a sequence in $X$ indexed by $P$. Up to translating by an element of $T_X(Y)$, we can assume that $P$ can be chosen so that $0 \in P$. Let $p \in P$ be the successor of $0$ and let $x \in Y$ be such that $f(t) = x$ for all but countably many $t \in [0,p]$. The assumption on the stabiliser subgroup of $L$ implies that there exists $\id \precneq g \precneq f$ such that $g \in \Stab_{T_X(Y)}(L)$. Therefore $\g \con \f' \in L$, where $f' \coloneqq f|_{[0,p]}$. Now $\g \con \f' = \h$ where $h([0,\ell(g)]) = x$ and $h([\ell(\g), \ell(\g) + p]) = -\ell(\g) \cdot x$. But since $\h \in L$ and $\g \preceq \f$, we have $\f \preceq \h$ therefore for all but countably many $t \in [\ell(g), \ell(f)]$ we have $-\ell(\g) \cdot x = f(t) = x$. Thus $\h$ is constant with image $x$. It follows by induction that $\g^n \con \f \in L$ is constant with image $x$ for all $n \in \mathbb{N}$. In particular, each $\g^n$ is constant with image $x$ and $\ell(\g) \in \Stab_{\mathbb{R}}(x)$ so it follows that each $\g^{-n}$ is constant with image $x^*$. Thus $L = L_x \cup L_{x^*}$. 
\end{proof}

\subsection{Proof of Theorems~\ref{thm: centraliser spectrum} and \ref{thm: small valence action}} \label{sec: centraliser spectrum}

\begin{proof}[Proof of Theorem~\ref{thm: centraliser spectrum}]
Fix an arbitrary map $\iota: \Sub_{NC}(\mathbb{R}) \rightarrow \mathcal{K}$. If $\iota$ is the zero map then let $X_0 \coloneqq \mathbb{R}/ \mathbb{Z} \sqcup (\mathbb{R} / \mathbb{Z})'$, equipped with the natural action of $\mathbb{R}$ and the involution $(x + \mathbb{Z})^* = (-x + \mathbb{Z})'$, ${(x + \mathbb{Z})'}^* = -x + \mathbb{Z}$ for all $x \in \mathbb{R}$. This defines an action of $\mathbb{R} \rtimes_\alpha \la * \ra$ on $X$. Let $G \coloneqq T_{X_0}(X_0)$ and recall that, by Proposition~\ref{prop: universal real tree}, $(G,d)$ is the complete universal real tree with valence $2^{\aleph_0}$. By Lemma~\ref{lem: controling axis stabilisers}, all the stabilisers of lines in $T_{X_0}(X_0)$ are either trivial or cyclic so the action of $T_{X_0}(X_0)$ on itself by left multiplication satisfies the conclusion of the theorem. If $\iota$ is the characteristic map of $\mathbb{R}$ then let $X_1 \coloneqq X_0 \sqcup \{x, x'\}$, where $X_0$ is equipped with the same action of $\mathbb{R} \rtimes_\alpha \la * \ra$ as before, $\mathbb{R}$ acts on $\{x,x'\}$ trivially and $x^* \coloneqq x', (x')^* \coloneqq x$. Let $G \coloneqq T_{X_1}(X_1)$ and note that $(G,d)$ is again the complete universal real tree with valence $2^{\aleph_0}$. Then by Lemma~\ref{lem: controling axis stabilisers} the only orbit of lines with non-trivial and non-cyclic stabilisers is the orbit of the standard axis $\mathbf{L}_x$. It follows that $|A_H| = |\{G \cdot \mathbf{L}_x\}| = 1$ if $H = \mathbb{R}$ and $|A_H|= 0$ otherwise.
Thus we can (and do) assume that $\iota$ is not the zero map or the characteristic map of $\mathbb{R}$.

For each subgroup $H \in \Sub_{NC}(\mathbb{R})$ let $B_{H} \coloneqq \mathbb{R}/H \sqcup (\mathbb{R}/H)'$ and define an action of $\mathbb{R} \rtimes \la * \ra$ on $B_H$ as follows. If $r \in \mathbb{R}$ and $x + H \in \mathbb{R}/H$ then $r \cdot (x + H) \coloneqq r + x + H$, $r \cdot (x + H)' = (r + x + H)'$, and $(x + H)^* \coloneqq (-x + H)'$, ${(x + H)'}^* \coloneqq -x + H$. Let $X_{H}$ be the disjoint union of $\iota(H)$ copies of $B_H$ and let $X \coloneqq \sqcup \{X_H : H \in \Sub_{NC}(\mathbb{R})\}$. Since $\sum_{H \in \Sub_{NC}(\mathbb{R})} \iota(H) \leq 2^{\aleph_0}$, we have $|X| \leq 2^{\aleph_0}$.

Let $G \coloneqq T_X(X)$. By Proposition~\ref{prop: universal real tree} $(G,d)$ is the complete universal real tree with valence $2^{\aleph_0}$. Let $H \in \Sub_{NC}(\mathbb{R})$. If $L \subseteq G$ is a line whose orbit is an element of $A_H$ then Lemma~\ref{lem: controling axis stabilisers} implies that $L = \f \star L'$ for some $\f \in G$ and some standard axis $L'$. Lemma~\ref{lem: orbit closure} implies that, if $y,y' \in Y$ and $y' \in \mathbb{R} \cdot \{y,y^*\}$,  then $\mathbf{L}_y, \mathbf{L}_{y'}$ are in the same $G$-orbit and it is clear from the definition of the operation on $G$ that, if $y' \notin \mathbb{R} \cdot \{y,y^*\}$, then $\mathbf{L}_y, \mathbf{L}_{y'}$ are in different $G$-orbits. By Lemma~\ref{lem: stab of standard axis}, if $y \in Y$ then $G \cdot \mathbf{L}_y \in A_H$ if and only $\Stab_\mathbb{R}(y) = H$. Therefore $|A_H|$ is equal to the number of copies of $B_H$ in $X_H$, which is $\iota(H)$ by construction.
\end{proof}

To prove Theorem~\ref{thm: small valence action} we will need two more lemmas:

\begin{lemma} \label{lem: intersection of dense axes}
    Let $T$ be a real tree and let $G$ be a group acting freely on $T$. Let $L \subseteq T$ be a line such that $\Stab_G(L)$ acts on $L$ with dense orbits. Then, for all $g \in G$, the intersection $gL \cap L$ is either empty, a point or a line.
\end{lemma}
\begin{proof}
    Let $M \subseteq T$ be a line, let $\varphi:M \rightarrow \mathbb{R}$ be an isometry and let $y \coloneqq \varphi^{-1}(0)$. We will say that the \textit{explicit stabiliser} of $M$ is  $K \coloneqq \varphi(\Stab_G(M) \cdot y) \leq \mathbb{R}$. Note that $K$ does not depend on the choice of $\varphi$. It follows from the freeness of the action that, for any $g \in G$ and $x \in M$, we have $g \in \Stab_G(M)$ if and only if $gx \in M$ and $d(x,gx) \in K$. Moreover, if $g \in G$ then $\varphi \circ g^{-1}:gM \rightarrow \mathbb{R}$ is an isometry, $(\varphi \circ g^{-1})^{-1}(0) = gy$  and $\varphi \circ g^{-1}(\Stab_G(gM) \cdot gy) = \varphi(\Stab_G(M) \cdot y) = K$, so $K$ is an invariant of $G \cdot M$.

    Let $H \leq \mathbb{R}$ be the explicit stabiliser of $L$.
    Let $g \in G$ be such that $|gL \cap L| > 1$.
    The intersection $L \cap gL$ is closed and connected so this implies that there is a non-degenerate segment $[x,y] \subseteq L \cap gL$. By assumption, there exists $h \in \Stab_G(L)$ such that $hx \in [x,y]$.
    Then $hx \in gL$ and $d(x,gx) \in H$ so $h \in \Stab_G(L) \cap \Stab_G(gL)$. If $z \in [x,y]$ is such that $d(z, y) < d(x,hx)$ then $hz \in L \cap gL$ and $hz \notin [x,y]$. This argument applies to any non-degenerate segment $[x,y] \subseteq L \cap gL$, so $L \cap gL$ has infinite length. Since the action of $G$ is free, $L \cap gL$ cannot be a ray, so we must have that $L \cap gL = L$.
\end{proof}

Given a group $G$ acting freely and transitively on an $\mathbb{R}$-tree $T$, we establish the following convention. Identify $G$ with $T$ via an orbit map $g \mapsto g \cdot x_0$ for some $x_0 \in T$ and equip $G$ with the partial order defined by $g \preceq f$ if $[\id, g] \subseteq [\id, f]$. Note that $(G, \preceq)$ is a meet semilattice and $g \wedge f \neq \id$ if and only if $f$ and $g$ are in the same direction at $\id$. For all $g \in G$ define $\ell(g) \coloneqq d(\id, g)$, where $d$ is the metric on $T$.

\begin{lemma} \label{lem: small stabilisers large valence}
    Let $\kappa \leq 2^{\aleph_0}$ be a cardinal. Let $T$ be a real tree which is not a single point and such that there is a group $G$ acting freely and transitively on $T$. If, for any line $L \subseteq T$, we have $|\Stab_G(L) \setminus L| \geq \kappa$, then the valence of $T$ is $\geq \kappa$.
\end{lemma}
\begin{proof} 
    We first assume that the $G$-stabiliser of every line in $T$ is at most countable.

    \setcounter{claim}{0}
    \begin{claim} \label{claim: countbale classes}
        Let $f \in G$ be non-trivial and, for each $g \prec f$, let $h_g \coloneqq g^{-1} f \wedge f$. Then, for all but countably many $g \prec f$, we have $h_g = \id$.
    \end{claim}
    \begin{proof}
        \renewcommand{\qed}{$\blacksquare$}
        Let $A \subseteq G$ be the set of $g \prec f$ such that $h_g \neq \id$ and suppose that $A$ is uncountable. For each $n \in \mathbb{N}$, let $A_n \coloneqq \{g \in A : \ell(h_g) \geq 1/n\}$. Then $A = \cup_{n \in \mathbb{N}} A_n$ so there is some $n$ for which $A_n$ is uncountable. Since $[\id,f]$ is a finite union of segments of length $\leq 1/2n$, there exists a segment $I \subseteq [\id,f]$ with length $\leq 1/2n$ such that $P \coloneqq A_n \cap I$ is uncountable. We can moreover assume, up to taking a subsegment with the same property, that $I = [p_1,p_2]$ for some $p_1,p_2 \in [\id,f]$ such that $p_1,p_2 \in A_n$.
        Let $k \coloneqq p_1^{-1} p_2$. Then $k \preceq h_p$ for all $p \in P$ so $pk \preceq f$. In  particular, $p_1k^2 \preceq f$ and $p_1k \preceq f$ which implies that $k \wedge k^{-1} = \id$. Let $L \subseteq G$ be the convex hull of $\la k \ra$ and note that $L$ is a line. For every $p \in P$ we have $p_1 \preceq pk \preceq p_1k^2$ and it follows that $p_1^{-1}p$ stabilises $L$, but there are uncountably many such elements.
    \end{proof}

    Let $G_f \coloneqq \{g \in G : g \preceq f\}$ and define an equivalence relation $\sim$ on $G_f$ by: $g \sim g'$ if $g^{-1}f \wedge {g'}^{-1} f \neq \id$. If $g \sim g'$ and $g \prec g' \prec f$ then $g^{-1}g' \preceq g^{-1} f$ and $(g^{-1}g')^{-1} g^{-1}f \wedge g^{-1}f = {g'}^{-1}f \wedge g^{-1}f \neq \id$. By Claim~\ref{claim: countbale classes}, there are at most countably many such elements $g^{-1}g' \preceq g^{-1} f$. It follows that equivalence classes of $G_f$ are at most countable, which implies that $|G_f / \sim| = 2^{\aleph_0}$, and thus the valence of $T$ is $\geq 2^{\aleph_0}$. 

    Now assume $T$ has valence $< 2^{\aleph_0}$. By the above argument, there is a line $L \subseteq T$ whose stabiliser is uncountable. Up to translating $L$ by some element of $G$, we can assume that $\id \in L$. The orbit $\Stab_G(L) \cdot x$ is dense in $L$ for any $x \in L$ so, by Lemma~\ref{lem: intersection of dense axes}, for any $f,g \in G$, the intersection $fL \cap gL$ is either empty, a single point or a line. Therefore, if $f,g \in L$ and $f^{-1} \wedge g^{-1} \neq \id$, then $f^{-1} L = g^{-1} L$ so $fg^{-1} \in \Stab_G(L)$. By assumption $|\Stab_G(L) \setminus L| \geq \kappa$ so this implies that the valence of $T$ is bounded below by $\kappa$.   
\end{proof}

\begin{proof}[Proof of Theorem~\ref{thm: small valence action}]
    Let $\kappa \geq 4$ be a cardinal which is either infinite or even. Let $I_\kappa$ be a set with cardinality $\kappa / 2$ and let $X_\kappa \coloneqq \{x_i, x_i^* : i \in I_\kappa\}$, equipped with the trivial action of $\mathbb{R}$ and the involution $\ast: x_i \mapsto x_i^*, x_i^* \mapsto x_i$. Let $H_\kappa \leq T_{X_\kappa}(X_\kappa)$ be the subgroup of complexity 1 elements of $T_{X_\kappa}(X_\kappa)$. Then $H_\kappa$ is an incomplete real tree and the set of directions at $\id$ is in bijection with $X$.
    
    Next, fix a cardinal $3 \leq \kappa < 2^{\aleph_0}$ and suppose there exists a real tree $T$ with valence $\kappa$ and a group $G$ which admits a free transitive action on $T$.

    Suppose $T$ is complete. We will show that $\kappa \geq 2^{\aleph_0}$, leading to a contradiction. We will need to construct a set $\{a_{1,n}, a_{2,n} \in G : n \in \mathbb{R}\}$ such that $\ell(a_{i,n}) \leq 1/n^2$ for all $n$ and, for all $n,m \in \mathbb{N}$, we have $a_{1,n}^{-1} \wedge a_{1,m} = a_{1,n}^{-1} \wedge a_{2,m} = a_{2,n} \wedge a_{2,m}^{-1} = \id$. 
    To this end, we first show that there exist lines $L_1, L_2 \subseteq T$ such that $L_1 \cap L_2 = \{\id\}$ and the actions $\Stab_G(L_i) \acting L_i$ have dense orbits. 
    
    By Lemma~\ref{lem: small stabilisers large valence}, there exists a line $L_1 \subseteq T$ with uncountable, and therefore dense, stabiliser. Up to translating by an element of $G$, we can assume that $\id \in L_1$. 
    If the stabiliser of $L_1$ is not transitive, let $g \in L_1 - \Stab_G(L_1)$ and $L_2 \coloneqq g^{-1}L_1$. Then $L_2 \neq L_1$ and $\id \in L_2$ so by Lemma~\ref{lem: intersection of dense axes} $L_1 \cap L_2 = \{\id\}$. 
    If $\Stab_G(L_1)$ is transitive then $\Stab_G(L_1) = L_1$. Let $H \subseteq G$ be the set of elements $g \in G$ such that the segment $[\id, g]$ does not intersect any translate of $L_1$ in a non-degenerate segment. If $g,h \in H$, then $[\id,gh] = [\id, g'] \cup g' \cdot [h',h]$ for some $g' \preceq g$ and $h' \preceq h$. For all $k \in G$ we have $[\id,g'] \cap kL_1 \subseteq [\id,g] \cap kL_1$ and $g' \cdot [h',h] \cap kL_1 \subseteq g' \cdot ([\id,h] \cap {g'}^{-1} k L_1)$ so $|[\id, gh] \cap kL_1| \leq 1$ and $gh \in H$. Also $[\id,g^{-1}] \cap kL_1 = g^{-1} \cdot ([\id,g] \cap gkL_1)$ which has cardinality at most 1 so $g^{-1} \in H$. Thus $H$ is a subgroup of $G$. It is clear from the definition that $H$ is a subtree of $G$. It follows from the fact that $\Stab_G(L_1) = L_1$ that $H$ is non-trivial -- more precisely, the valence of $H$ is $\kappa - 2$. Thus by Lemma~\ref{lem: small stabilisers large valence} there is a line $L_2 \subseteq H$ containing $\id$ such that $\Stab_H(L_2) = \Stab_G(L_2)$ acts on $L_2$ with dense stabilisers. 

    Now fix rays $L_i^+ \subseteq L_i$ based at $\id$ for each $i$.
    For each $n \in \mathbb{N}$, let $a_{1,n} \in L_1^+ \cap \Stab_G(L_1), a_{2,n} \in L_2^+ \cap \Stab_G(L_2)$ be such that $0 < \ell(a_{1,n}) <  \ell(a_{2,n}) \leq 1/n^2$. Then for all $n,m \in \mathbb{N}$, $a_{1,n}^{-1}, a_{1,m}^{-1} \in L_1 - L_1^+$ and $a_{2,n}^{-1}, a_{2,m}^{-1} \in L_2 - L_2^+$ so 
    \[
        a_{1,n}^{-1} \wedge a_{1,m} =  a_{2,n}^{-1} \wedge a_{2,m}^{-1} = a_{1,n}^{-1} \wedge a_{2,m} = a_{2,n}^{-1} \wedge a_{1,m} = \id.
    \]
    Given a map $\theta: \mathbb{N} \rightarrow \{1,2\}$, define $g^\theta_n \coloneqq a_{\theta(1),1} \dots a_{\theta(n),n}$. The way we chose the $a_{1,i}$'s and $a_{2,i}$'s implies that $\ell(g^\theta_n) = \sum_{i=1}^n \ell(a_{\theta(i),i})$, $g^\theta_n \prec g^\theta_{n+1}$ and $(\ell(g^{\theta}_n))_{n \in \mathbb{N}}$ is convergent, so $(g^\theta_n)_{n \in \mathbb{N}}$ is Cauchy and has a limit $g^\theta \in G$. It also follows from the choice of $a_{1,i}$'s and $a_{2,i}$'s that, if $\theta \neq \theta'$, then $(g^\theta)^{-1} \wedge (g^{\theta'})^{-1} = \id$ which means that $T$ has at least $2^{\aleph_0}$ directions at $\id$.

    Finally, suppose that $\kappa$ is finite. By Lemma~\ref{lem: small stabilisers large valence}, there is a line $L \subseteq T$ such that $|\Stab_G(L) \setminus L| \leq \kappa$. The additive group $\mathbb{R}$ has no finite index subgroups, so this implies that $\Stab_G(L)$ acts transitively on $L$. We assume without loss of generality that $\id \in L$ and note that $L = \Stab_G(L)$ in this case.  Let $H \subseteq G$ be the set of elements $g \in G$ such that the segment $[\id, g]$ does not intersect any translate of $L_1$ in a non-degenerate segment. Recall that $H$ is a subgroup of $G$ and a subtree with valence $\kappa -2$. It follows by induction on $\kappa$ that $\kappa$ is even and there are lines $L_1, \dots, L_{\kappa/2} \subseteq T$ such that $L_i \cap L_j = \{\id\}$ for all $i \neq j$ and $\Stab_G(L_i) = L_i$ for all $i$. 
    
    To prove the uniqueness statement, recall the definition of the group $H_\kappa$ from the first paragraph of this proof. For each $i \in \{1, \dots, \kappa/2\}$ fix an isometry $\psi: \mathbf{L}_{x_i} \rightarrow L_i$ such that $\psi_i(\id) = \id$. These extend to an isometric and homomorphic embedding $\psi: H_\kappa \rightarrow T$. 
    
    It remains to show that $\psi$ is surjective. Let $p \in T$ be in the closure of $\psi(H_\kappa)$ and let $(h_n)_{n \in \mathbb{N}} \subseteq \psi(H_\kappa)$ be a sequence with limit $p$. Since $\psi(H_\kappa)$ is connected, we can assume without loss of generality that $h_n \preceq h_{n+1}$ for all $n \in \mathbb{N}$. Each $p^{-1} h_n$ lies in the same direction at $\id$ and $\ell(p^{-1} h_n) \rightarrow 0$. Therefore there exists $i \in \{1, \dots, \kappa / 2\}$ such that , for sufficiently large $n$, we have $p^{-1} h_n \in L_i \subseteq \psi(H_\kappa)$. Thus $p \in \psi(H_\kappa)$ and $\psi(H_\kappa)$ is closed. Now suppose there exists $p \in T - \psi(H_\kappa)$ and let $p' \in \psi(H_\kappa)$ be the closest point projection of $p$ onto $\psi(H_\kappa)$. Then $[p',p] \cap \psi(H_\kappa) = \{p'\}$, so $p$ lies in a different direction at $p'$ to any point of $\psi(H_\kappa)$. But since the valence of $H_\kappa$ is $\kappa$, this implies that the valence of $p'$ in $T$ is $\geq \kappa + 1$, which is a contradiction.
\end{proof}

\begin{remark} \label{rem: isomorphic actions}
    Let $\kappa \geq 4$ be a cardinal which is not both finite and odd, and let $H_\kappa$ be the group defined at the start of the above proof. Then $H_\kappa$ is isomorphic to the free product of $\kappa$ copies of $\mathbb{R}$ and the action of $H_\kappa$ on itself is precisely the action constructed in \cite{Chiswell-Muller} for free products of copies of $\mathbb{R}$.
    
    For each $i \in I_\kappa$, let $R_i = \mathbb{R}$. Let $\ast_{I_\kappa} R_i$ be the free product of $\{R_i : i \in I_\kappa\}$. Recall that, for each $i \in I_\kappa$, since $\Stab_{H_\kappa}(\mathbf{L}_{x_i}) = \mathbf{L}_{x_i}$, the map $\mathbf{L}_{x_i} \rightarrow R_i$ which maps $\f \in \mathbf{L}_{x_i}$ to $\ell(\f)$ if the image of $\f$ is $x_i$, and maps $\f$ to $-\ell(\f)$ otherwise, is an isomorphism. The group $H_\kappa$ is generated by $\cup_{i \in I_\kappa} \mathbf{L}_{x_i}$, so these isomorphisms extend to an isomorphism $\psi: H_\kappa \rightarrow \ast_{I_\kappa} R_i$. Moreover, length function $\ell \circ \psi^{-1} : \ast_{I_\kappa} R_i \rightarrow \mathbb{R}$ is precisely the Lyndon length function defined in \cite{Chiswell-Muller} to construct a free transitive action on an $\mathbb{R}$-tree. It follows that the resulting free transitive actions are isomorphic.

    If $\kappa = 2^{\aleph_0}$ then $H_\kappa \acting H_\kappa$ is also isomorphic to the free transitive action on Uryson's tree constructed in \cite{Berestovskii1989,Berestovskii2019}. 
\end{remark}

\subsection{Reduction retraction} \label{sec: reduction}

Suppose that $x^* \notin \mathbb{R} \cdot x$ and $\Stab_\mathbb{R}(x)$ is non-trivial for all $x \in X$.

\begin{proposition} \label{prop: reduction functor}
    There exists a retraction ${\Lightning} : \mathcal{Z}_X \rightarrow T_X(X)$ such that $(\f \con \g)^{\Lightning} = (\f^{\Lightning} \con \g^{\Lightning})^{\Lightning} = \f^{\Lightning} \star \g^{\Lightning}$ and $(\f^{-1})^{\Lightning} = {\f^{\Lightning}}^{-1}$ for all $\f,\g \in \mathcal{Z}_X$.
\end{proposition}

Let us start with some terminology.

\begin{definition}[Reducible at $t$]
    Let $\ell > 0$ and $Q \subseteq [0,\ell]$ be a finite template. Let $(\f_q)_{q \in Q} \subseteq T_X$ be a consistent sequence with concatenation $\f$. We say that $\f$ is \textit{reducible at} $t$ if there exists $t_1,t_2 \in Q$ with $t_1 < t < t_2$ such that $Q \cap (t_1,t_2) = \{t\}$ and there exists $0 < s \leq \min\{t-t_1, t_2-t\}$ such that for all but countably many $0 \leq \varepsilon \leq s$ we have $f_{t_1}^{-1}(\varepsilon) = f_t(\varepsilon)$.
\end{definition}
\begin{remark}
    If $\f$ is the concatenation of another finite sequence $(\f'_{q'})_{q' \in Q'} \subseteq T_X$ then $t \in Q'$ and $\f$ is also reducible at $t$ with respect to $(\f'_{q'})_{q' \in Q'}$. Thus reducibility at $t$ is well-defined.
\end{remark}

We now define reduction of a finite sequence of elements of $T_X$. 

\begin{definition}[Reduction of a finite sequence at $t$] \label{defn: reduction at t}
    Let $\ell > 0$ and $Q \subseteq [0,\ell]$ be a finite template. Let $(\f_q)_{q \in Q} \subseteq T_X$ be a consistent but inadmissible sequence with concatenation $\f$ and suppose that $\f$ is reducible at $t \in Q$. Let $t_1$ be the predecessor of $t$ and $t_2$ be its successor. Let $\sigma > 0$ be the maximal $0 < s \leq \min\{t - t_1,t_2 - t\}$ such that $f_{t_1}^*(t - t_1 - \varepsilon) = 2t \cdot f_t(\varepsilon)$ for all but countably many $0 \leq \varepsilon \leq s$. Let $Q^t \coloneqq Q \cup \{t - \sigma, t + \sigma\}$ and define a sequence $(\f^t_q)_{q \in Q^t} \subseteq T_X$ as follows:
    \begin{itemize}
        \item if $q < t_1$ or $q \geq t_2$, then $f_q^t \coloneqq f_q$;
        \item if $t_1 \neq t-\sigma$, then $f^t_{t_1} \coloneqq f_{t_1}|_{[0,t-t_1 - \sigma]}$;
        \item $f^t_{t-\sigma}: [0,\sigma] \rightarrow X$ is defined by $f^t_{t-\sigma}(u) = (t-\sigma) \cdot f_{t_1}(u + t - t_1 - \sigma)$;
        \item $f^t_t \coloneqq f_t|_{[0,\sigma]}$;
        \item if $t + \sigma \neq t_2$, then $f^t_{t+\sigma}: [0,t_2-t-\sigma] \rightarrow X$ is defined by $f^t_{t+\sigma}(u) = \sigma \cdot f_t(u + \sigma)$.
    \end{itemize}
    Observe that the concatenation of $(\f^t_q)_{q \in Q^t}$ is $\f$.
    Now define 
    \[ Q_t \coloneqq \{q \in Q^t : q \leq t-\sigma\} \cup \{q - 2\sigma : q \in Q^t, q \geq t+\sigma\}\]
    and note that $Q_t$ is a finite template of $[0,\ell - 2\sigma]$. Define $r_{Q,t}: Q^t - \{t\}\rightarrow Q_t$ by $r_{Q,t}(q) = q$ if $q \leq t - \sigma$ and $r_{Q,t}(q) = q-2\sigma$ otherwise. 
    
    Given $q \in Q_t$ let $\f^{(t)}_q \coloneqq \f^t_q$ if $q < t-\sigma$ and $\f^{(t)}_q = \f^t_{q+2\sigma}$ otherwise. The \textit{reduction} $\mathrm{red}(\f,Q,t)$ of $(\f,Q)$ at $t$ is the concatenation of $(\f^{(t)}_q)_{q \in Q_t}$.
\end{definition}

\begin{remark} \label{rem: decreasing reducibility}
By the maximality of $\sigma$, if $t-\sigma \neq t_1$ and $t+\sigma \neq t_2$ then $\f^{(t)}$ is not reducible at $t-\sigma$. This implies that the number of points at which $\mathrm{red}(\f,Q,t)$ is reducible is strictly smaller than the number of points at which $\f$ is reducible.
\end{remark}

\begin{lemma} \label{lem: reduction is independent of order}
    Suppose that $\ell,Q, (\f_q)_{q \in Q}$ and $\f$ are as in Definition~\ref{defn: reduction at t} and that $\f$ is reducible at the points $s,t \in Q$. Let $U$ be the template and $(\g_u)_{u \in U}$ be the sequence obtained by first reducing $\f$ at $s$ then at $r_{Q,s}(t)$ and let $V$ be the template and $(\h_v)_{v \in V}$ be the sequence obtained by reducing $\f$ first at $t$ then at $r_{Q,t}(s)$. Then $U=V$ and $(\g_u)_{u \in U} = (\h_v)_{v \in V}$.
\end{lemma}
\begin{proof}
    Relabel $s$ and $t$ if necessary so that $s < t$ and let $s_1,s_2 \in Q^s$, $t_1, t_2 \in Q^t$ be such that $s_1 < s_2$, $t_1 < t_2$, $Q^s \cap (s_1, s_2) = \{s\}$ and $Q^t \cap (t_1, t_2) = \{t\}$. Let $\sigma \coloneqq s-s_1 = s_2-s$ and $\tau \coloneqq t-t_1 = t_2-t$. First suppose that $s_2 \leq t_1$. It is clear from the definition that the reduction of $x$ at $s$ is reducible at $r_{Q,s}(t)$ and similarly the reduction of $x$ at $t$ is reducible at $r_{Q,t}(s)$.
    Since the collapsed intervals are disjoint, it is clear that $U=V$ and $(\g_u)_{u \in U} = (\h_v)_{v \in V}$ in this case.
    
    So suppose that $t_1 < s_2$. In this case $s_1 \leq s \leq t_1 < s_2 \leq t \leq t_2$. 
    Let $\sigma' \coloneqq t_1 - s$ and $\tau' \coloneqq t - s_2$. Then 
    \begin{align*} 
    U &= \{p \in Q : p \leq s_1\} \cup \{p - 2(\sigma + \tau') : p \in Q, p \geq t_2\} \\
    &= \{p \in Q : p \leq s_1\} \cup \{p - 2(t-s) : p \in Q, p \geq t_2\}
    \end{align*}
    and 
    \begin{align*} 
    V &= \{p \in Q : p \leq s - \sigma'\} \cup \{p - 2(\tau + \sigma') : p \in Q, p \geq t_2\} \\
    &= \{p \in Q : p \leq s - \sigma'\} \cup \{p - 2(t-s) : p \in Q, p \geq t_2\}.
    \end{align*}

    Since $t_2 - 2(\sigma + \tau') = s - \sigma'$, we indeed have $U=V$.

    Let $s_1' \in Q$ be such that $s_1'<s$ and $Q \cap (s_1',s) = \emptyset$ and let $t_2' \in Q$ be such that $t < t_2'$ and $Q \cap (t,t_2') = \emptyset$. For each $p < s_1'$, $g_p = f_p = h_p$. Also if $s_1 \neq s_1'$ then $g_{s_1'} = (f_{s_1'})|_{[0,s_1 - s_1']} = h_{s_1'}$.
    We have $g_{s_1}: [0,\tau-\tau'] \rightarrow X$ is given by $g_{s_1}(r) = (t + \tau') \cdot f(r+t + \tau')$ and $h_{s_1}: [0, \sigma - \sigma'] \rightarrow X$ is given by $h_{s_1}(r) = s_1 \cdot f(r+s_1)$.
    Note that $\tau - \tau' = \sigma - \sigma'$ and $t$ is the successor of $s$ in $Q$. By assumption, $(\f^s_{s_1})^{-1} = \f_s^s$ and $(\f^t_{t_1})^{-1} = \f_t^t$. Therefore, for all $r \in [0, \sigma - \sigma']$, 
    \[
        s_1 \cdot f(r + s_1) 
        = f^s_{s_1}(r) 
        = (f^s_s)^{-1}(r) 
        = -\sigma \cdot (s \cdot f(s + (\sigma - r)))^*
        = - s_2 \cdot f^*(s_2 - r)
    \]
    and 
    \begin{align*}
        (t + \tau') \cdot f(r + t + \tau') 
        &= \tau' \cdot f_t^t(r + \tau') \\
        &= \tau' \cdot (f_{t_1}^t)^{-1}(r + \tau') \\
        &= (\tau' - \tau) \cdot (t_1 \cdot f(t - (r+\tau')))^* \\
        &= -s_2 \cdot f^*(s_2 - r)
    \end{align*}
    Lastly, for all $p \in Q$ with $p \geq t_2'$, we have $g_{p-2(t-s)} = f_p = h   _{p-2(t-s)}$ and, if $t_2 \neq t_2'$, then $g_{t_2}, h_{t_2}:[0,t_2'-t_2] \rightarrow X$ are defined by $g_{t_2}(r) = h_{t_2}(r) = f_t(r+\tau)$.
\end{proof}

Combining Lemma~\ref{lem: reduction is independent of order} with Remark~\ref{rem: decreasing reducibility}, the following is well-defined.

\begin{definition}[Reduction of a finite sequence] \label{defn: reduction finite sequence}
    Let $\ell > 0$, let $Q \subseteq [0,\ell]$ be a finite template and let $(\f_q)_{q \in Q}$ be a consistent sequence with concatenation $\f$. If $\f$ is reducible at $t \in Q$ then reduce it at $t$. If $\red(\f,Q,t)$ is reducible at $t' \in Q_t$ then reduce it at $t'$. Iterating this procedure, we eventually obtain an admissible element $\f_Q^{\Lightning}$, which we call \textit{the reduction of $\f$ with respect to $Q$}.
\end{definition}

\begin{lemma} \label{reduction closed under Cauchy}
    Suppose that $m > 0$ and $Q = \{q_n : n \in \mathbb{N}\} \subseteq [0,m]$ is a template such that $q_n < q_{n+1}$ for each $n \in \mathbb{N}$. In particular, $m$ is the unique accumulation point of $Q$. For each $n$, let $Q_n \coloneqq \{q_i : 1 \leq i \leq n\}$. Let $(\f_q)_{q \in Q} \subseteq T_X$ be a sequence and for each $n$, let $\g_n$ be the concatenation of $(\f_q)_{q \in Q_n}$. Then $((\g_n)_Q^{\Lightning})_{n \in \mathbb{N}} \subseteq T_X(X)$ is Cauchy.
\end{lemma}
\begin{proof}
    Observe that $d((g_{n+1})_Q^{\Lightning},(g_n)_Q^{\Lightning}) \leq 2(q_{n+1} - q_n) \rightarrow 0$ as $n \rightarrow \infty$.
\end{proof}

\begin{definition}[Reduction of a Cauchy sequence]
    Let $m,Q, (\f_q)_{q \in Q}, (\g_n)_{n \in \mathbb{N}}$ be as in Lemma~\ref{reduction closed under Cauchy} and let $\f$ be the concatenation of $ (f_q)_{q \in Q}$. The \textit{reduction} of $\f$ with respect to $Q$, denoted $\f_Q^{\Lightning}$, is the limit of the Cauchy sequence $((\g_n)_Q^{\Lightning})_{n \in \mathbb{N}}$. 
\end{definition}

\begin{proof}[Proof of Proposition~\ref{prop: reduction functor}]
\setcounter{claim}{0}
 Let us define ${\Lightning}: \mathcal{Z}_X \rightarrow T_X(X)$. Let $\f \in \mathcal{Z}_X$ be an element with complexity $\alpha$. We proceed by transfinite induction on $\alpha$.
 
 First suppose that $\alpha = 1$ and fix a sequence $x = (x_p)_{p \in P} \subseteq X$ such that $\f$ is its realisation. Since $P$ is finite, we can (and do) assume that it is minimal among the set of templates for $\f$. For each $p \in P - \{\ell\}$, let $p' \in P$ be the successor of $p$ and define $f_p: [0,p'-p] \rightarrow X$ by $f_p(t) = p \cdot x_p$ for all $t$. Let $\f_\ell \coloneqq \id$. Then $\f$ is the concatenation of $(\f_p)_{p \in P}$. The reduction of $\f$ is $\f^{\Lightning} \coloneqq \f_P^{\Lightning}$ in the sense of Definition~\ref{defn: reduction finite sequence}.

So suppose that $\alpha = \beta + 1$ for some countable ordinal $\beta$ and suppose that, for any $\g \in \mathcal{Z}_X^{[\beta]}$ (recall Definition~\ref{def: complexity}), the reduction $\g^{\Lightning}$ of $\g$ has been defined in such a way that it depends only on $\g$. Suppose moreover that the following holds. 

\begin{itemize}
\item[(C)]Let $m > 0$, let $Q \subseteq [0,m]$ be a template such that $Q^{(\beta)} = m$ and fix a sequence $y = (y_q)_{q \in Q} \subseteq X$. Let $\{q_n : n \in \mathbb{N}\} \subseteq Q$ be such that $q_n < q_{n+1}$ for each $n \in \mathbb{N}$, so that $(q_n)_{n \in \mathbb{N}}$ is a strictly increasing sequence with limit $m$. For each $n \in \mathbb{N}$ let $Q_n \coloneqq \{q \in Q : q \leq q_n\}$ and let $\g_n$ be the realisation of $(y_q)_{q \in Q_n}$. Then $(\g_n^{\Lightning})_{n \in \mathbb{N}}$ is Cauchy.
\item[(F)] If $\f, \g \in \mathcal{Z}_X^{[\beta]}$ then $(\f \con \g)^{\Lightning} = \f^{\Lightning} \star \g^{\Lightning}$. 
\item[(I)]If $\f \in \mathcal{Z}_X^{[\beta]}$ then $(\f^{-1})^{\Lightning} = {\f^{\Lightning}}^{-1}$. 
\end{itemize}
In the case where $\beta = 1$, Property (C) holds by Lemma~\ref{reduction closed under Cauchy} and Property (F) holds by Lemma~\ref{lem: reduction is independent of order}. To check Property (I) in this case, recall that $\f^{-1}$ is the realisation of the sequence $(-\ell \cdot y_{p}^*)_{p \in P^{-1}}$, where $P^{-1} \coloneqq \{\ell - p : p \in P\}$ and $y_{p} = -\ell \cdot x_{p'}^*$ where $p' \in P$ is the predecessor of $\ell - p \in P$. Property (I) therefore follows from Definition~\ref{defn: reduction finite sequence}.

Now let $\f \in \mathcal{Z}_X$ be an element with complexity $\alpha$. Let $\ell \coloneqq \ell(\f)$, let $P \subseteq [0,\ell]$ be a template and $(x_p)_{p \in P} \subseteq X$ be a sequence with realisation $\f$. Suppose moreover that $\overline{P}^{(\beta)} \subseteq \overline{Q}^{(\beta)}$ for any template $Q$ which indexes a sequence in $X$ with realisation $\f$. Note that this property renders $\overline{P}^{(\beta)}$ unique (although $P$ itself may not be). Since the complexity of $\f$ is $\alpha$, $\overline{P}^{(\beta)}$ is finite.

If $\overline{P}^{(\beta)} = \{\ell\}$, then let $\{q_n : n \in \mathbb{N}\} \subseteq P$ be such that $q_n < q_{n+1}$ for each $n \in \mathbb{N}$, so that $(q_n)_{n \in \mathbb{N}}$ is a strictly increasing sequence with limit $\ell$. For each $n \in \mathbb{N}$ let $P_n \coloneqq \{p \in P : p \leq p_n\}$ and let $\f_n \in \mathcal{Z}^{[\beta]}$ be the realisation of $(x_p)_{p \in P_n}$.
By Item (C), the sequence $(\f_n^\reduced)_{n \in \mathbb{N}}$ is Cauchy. Let $\f^\reduced \in T_X(X)$ be the limit of $(\f_n^\reduced)_{n \in \mathbb{N}}$.

If $\overline{P}^{(\beta)} = \{0\}$ then let $\f^\reduced \coloneqq ({\f^{-1}}^\reduced)^{-1}$.

Now suppose that $\overline{P}^{(\beta)} = \{0,\ell\}$. Choose a point $t \in P$ and let $P_t^- \coloneqq \{p \in P : p \leq t\}$ and $P_t^+ \coloneqq \{p-t : p \in P$ and $p \geq t\}$. Then $P_t^-$ and $P_t^+$ are templates of $[0,t]$ and $[0,\ell-t]$ respectively. For each $p \in P_t^-$ let $x_p^- \coloneqq x_p$ and for each $p \in P_t^+$ let $x_p^+ \coloneqq t \cdot x_{p+t}$. Let $\f_t^-, \f_t^+$ be the realisations of $(x_p^-)_{p \in P_t^-}$ and $(x_p^+)_{p \in P_t^+}$ respectively. Then let $\f_t^{\Lightning}$ be the reduction of ${\f_t^-}^\reduced \con {\f_t^+}^{\reduced}$ with respect to the template $\{0,t,\ell\}$.

\begin{claim} \label{two-sided Cauchy}
If $t,t' \in P$ then $\f_t^{\Lightning} = \f_{t'}^{\Lightning}$.
\end{claim}
\begin{proof}
\renewcommand{\qedsymbol}{$\blacksquare$}
Relabel $t$ and $t'$ if necessary so that $t < t'$. Let $P_{[t,t']} \coloneqq \{p - t : p \in P, t \leq p \leq t'\}$ and note that $P_{[t,t']}$ is a template of $[0,t'-t]$ whose closure has CB-rank $\leq \beta$. Let $\f_{[t,t']}^{\Lightning}$ denote the reduction of the realisation of $(x_{p+t})_{p \in P_{[t,t']}}$ and let $Q \coloneqq \{0,t',t,\ell\}$. It follows from the definition that $\f_t^+ = (\f_{[t,t']} \con \f_{t'}^+)_Q^{\Lightning}$ and $\f_{t'}^- = (\f_{[t,t']}^{-1}) \con \f_t)_Q^{\Lightning}$. Therefore $\f_t^{\Lightning}$ is obtained from $(\f_t^-)^{-1} \con \f_{[t,t']} \con \f_{t'}^+$ by first reducing at $t'$ then at $r_{P,t'}(t)$ and $\f_{t'}^{\Lightning}$ is obtained from $(\f_t^-)^{-1} \con \f_{[t,t']} \con \f_{t'}^+$ by first reducing at $t$ then at $r_{P,t}(t')$. Thus $\f_t^{\Lightning} = \f_{t'}^{\Lightning}$ by Lemma~\ref{lem: reduction is independent of order}.
\end{proof} 

The reduction of $\f$ is $\f^{\Lightning} \coloneqq \f_t^{\Lightning}$ for any $t \in P^{[\beta]}$ such that if $s \in P$ with $s < t$ and $P \cap (s,t) = \emptyset$ then $x_{s} \neq x_t$. In particular any template for $\f$ contains $t$ so the definition of $\f^{\Lightning}$ does not depend on $P$.

In the general case, let $b_0, \dots, b_m \in [0,\ell]$ be such that $b_0 < \dots < b_m$ and $\overline{P}^{(\beta)} \cup \{0,\ell\} = \{b_1, \dots, b_m\}$. For each $i \in \{0, \dots, m-1\}$, let $P_i \coloneqq \{p-b_i : p \in P, b_i \leq p \leq b_{i+1}\}$ and $f_{b_i}:[0,b_{i+1} - b_i] \rightarrow X$ be defined by $f_{b_i}(t) = f(t+b_i)$. For each such $i$, the reduction $\f_{b_i}^{\Lightning}$ has been defined. Moreover $\f = \f_0 \con \dots \con \f_m$, so let $\f^{\Lightning}$ be the reduction of the consistent sequence $(\f_{b_i}^\reduced)_{i \in \{1,\dots,m-1\}}$ with respect to the finite template $\{b_0, \dots, b_m\} \subseteq [0,\ell]$. The reduction of $\f$ is well-defined by Lemma~\ref{lem: reduction is independent of order} and the uniqueness of $\{b_0, \dots, b_m\}$ while Property~(C) holds by Lemma~\ref{reduction closed under Cauchy}. Property (F) is checked in the claim below. Property (I) follows immediately from the construction.

\begin{claim}
    Let $\f,\g \in \mathcal{Z}_X^{[\alpha]}$. Then $(\f \con \g)^{\Lightning} = (\f^{\Lightning} \con \g {\Lightning})^{\Lightning}= \f^{\Lightning} \star \g^{\Lightning}$.
\end{claim}
\begin{proof}
    \renewcommand{\qedsymbol}{$\blacksquare$}
    If the complexities of both $\f$ and $\g$ are strictly less than $\alpha$, then the claim follows by the induction hypothesis. So suppose that at least one of $\f, \g$ has complexity $\alpha$.

    Let $\ell \coloneqq \ell(\f), m \coloneqq \ell(g)$ and let $P \subseteq [0,\ell], Q \in [0,m]$ be templates which can index sequences whose realisations are $\f$ and $\g$ respectively. Suppose moreover that $\overline{P}^{(\beta)} \subseteq \overline{P'}^{(\beta)}$ for any template $P'$ which indexes a sequence in $X$ with realisation $\f$, and similarly $\overline{Q}^{(\beta)} \subseteq \overline{Q'}^{(\beta)}$ for any template $Q'$ which indexes a sequence in $X$ with realisation $\g$.
    
    If $\ell \in \overline{P}^{(\beta)}$ or $0 \in \overline{Q}^{(\beta)}$ then, by definition, $(\f \con \g)^{\Lightning} = (\f^{\Lightning} \con g^{\Lightning})^{\Lightning} = \f^{\Lightning} \star \g^{\Lightning}$. So suppose that $\ell \notin \overline{P}^{(\beta)}$ and $0 \notin \overline{Q}^{(\beta)}$. Let $p \coloneqq \max \overline{P}^{(\beta)} \cup \{0\}$ and $q \coloneqq \min \overline{Q}^{(\beta)} \cup \{\ell(\g)\}$ and let $f' \coloneqq f|_{[0,p]}$ and $g' \coloneqq q \cdot g|_{[q,\ell(g)]}$. Let $\h_1,\h_2 \in T_X(X)$ be such that $\f = \f' \con \h_1$ and $\g = \h_2 \con \g'$, so $\f \con \g = \f' \con \h_1 \con \h_2 \con \g'$. Then, using the induction hypothesis and Lemma~\ref{reduction closed under Cauchy}, we have $(\h_1 \con \h_2)^{\Lightning} = (\h_1^{\Lightning} \con \h_2^{\Lightning})$. Therefore
    \begin{align*}
    (\f^{\Lightning} \con \g^{\Lightning})^{\Lightning} &= (({\f'}^{\Lightning} \con \h_1^{\Lightning})^{\Lightning} \con (\h_2^{\Lightning} \con {\g'}^{\Lightning})^{\Lightning})^{\Lightning} \\
    &= ({\f'}^{\Lightning} \con (\h_1^{\Lightning} \con \h_2^{\Lightning})^{\Lightning} \con {\g'}^{\Lightning})^{\Lightning} \\
    &= ({\f'}^{\Lightning} \con (\h_1 \con \h_2)^{\Lightning} \con {\g'}^{\Lightning})^{\Lightning} \\
    &= (\f \con \g)^{\Lightning}.
    \end{align*}
    The first and last inequalities follow from the definition of reduction and the second follows from Lemma~\ref{lem: reduction is independent of order}. 
\end{proof}

This completes the proof of the proposition. 
\end{proof}

\section{Actions on $\Lambda$-trees} \label{sec: Lambda tree groups}

Some of the ideas from the previous section can also be used to construct actions on $\Lambda$-trees for arbitrary totally ordered abelian groups $\Lambda$. Since we do not assume that $\Lambda$ is uncountable or that every bounded set in $\Lambda$ has a supremum, the exact analogue of the construction does not work. On the other hand, the notion of completeness is no longer very relevant in this context so there is no need for the more involved notion of templates introduced in Section~{\ref{sec: templates} to ensure completeness. The group we obtain will be an analogue of the subgroup of $T_X(X)$ consisting of elements with complexity 1. 

As before, the construction allows for flexibility with respect to the stabilisers of `lines'. To illustrate this, we will prove the following proposition.

\begin{proposition} \label{prop: Lambda tree result}
    Let $\Lambda$ be a totally ordered abelian group, let $H \leq \Lambda$ be any subgroup and let $\kappa = 2|\Lambda / H|$. Then there exists a group $G$ acting freely and transitively on a $\Lambda$-tree $T$ with valence $\kappa$ such that there exists a subspace $L \subseteq T$ and isometry $\varphi: L \rightarrow \Lambda$ (where $\Lambda$ is equipped with the $\Lambda$-metric $d(\lambda_1, \lambda_2) = |\lambda_1 - \lambda_2|$) such that $\varphi(\Stab_G(L) \cdot x) = H$, where $x = \varphi^{-1}(0)$. In particular $\Stab_G(L) \cong H$. 
\end{proposition}

Let $\alpha:\Lambda \rightarrow \Lambda$ be the order reversing automorphism of $\Lambda$ defined by $\alpha(\lambda) = - \lambda$ for all $\lambda \in \Lambda$.
Let $X$ be a set equipped with an action of $\Lambda \rtimes_\alpha \la * \ra$, where $*$ is an element of order two. We will abuse notation and identify $\Lambda$ with the normal subgroup $(\Lambda, \id) \unlhd \Lambda \rtimes_\alpha \la \ast \ra$. 

\begin{definition}
    Let $\mathcal{Y}_X$ denote the set of maps $f:(0,\ell] \rightarrow X$, where $\ell \in \Lambda$ with $\ell \geq 0$, such that, for some $k \in \mathbb{N} \cup \{0\}$ and some finite sequence $0 = p_0 < \dots < p_k = \ell$, the map $f$ if constant on each interval $(p_i, p_{i+1}]$.
    When $\ell = 0$, $f$ is necessarily the empty map, which we denote by $\id: \emptyset \rightarrow X$.
    The \textit{length} of $f:(0,\ell] \rightarrow X$ is $\ell(f) \coloneqq \ell$.
\end{definition}

\begin{remark}
    We can define templates for intervals in $\Lambda$ in analogy with the notion for intervals in $\mathbb{R}$ (see Definition~\ref{templates}). Then the set $P \coloneqq \{p_0, \dots, p_k\}$ in the above definition is just a finite template for the interval $[0,\ell]$ and $f$ is the realisation of a sequence $(x_p)_{p \in P} \subseteq X$.
\end{remark}

\begin{definition} 
Fix $f,g \in \mathcal{Y}_X$.
    \begin{itemize}
        \item Define $f^{-1}:(0, \ell(f)] \rightarrow X$ as follows. Let $\{p_0, \dots, p_k\} \subseteq [0,\ell]$ be a template for $f$ and for each $i \in \{0, \dots,k-1\}$ and $t \in (p_i,p_{i+1}]$ let $f^{-1}(t) = -\ell(f) \cdot f^*(p_{i+1})$.
        \item Define $f \con g:(0, \ell(f) + \ell(g)] \rightarrow X$ by 
        \[ 
            f \con g(t) = 
            \begin{cases}
                f(t) \quad &\text{if } t \in (0, \ell(f)]; \\
                -\ell(f) \cdot g(t - \ell(f)) &\text{otherwise.}
            \end{cases}
        \]
        \item We say that $f \preceq g$ if $f$ is a restriction of $g$. 
    \end{itemize}
\end{definition}

\begin{lemma}
    $(\mathcal{Y}_X, \preceq, \id, \con)$ is an ore.
\end{lemma}
\begin{proof}
    It is straightforward to see that $(Y, \preceq, \id)$ is a median semilattice. The rest of the proof is completely analogous to that of Lemma~\ref{lem: big tree ore}. 
\end{proof}

Let $(S_X, \star)$ be the group extracted from $\mathcal{Y}_X$.

\begin{lemma}
    The map $\ell: \mathcal{Y}_X \rightarrow \Lambda$ is a $\Lambda$-length function (Definition~\ref{def: length function}). If $d$ is the resulting $\Lambda$-metric on $S_X$ then $(S_X,d)$ is a $\Lambda$-tree.
    
    If $x$ and $x^*$ are in different $\Lambda$-orbits for all $x \in X$ then the valence of $S_X$ is the cardinality of $X$. 
\end{lemma}
\begin{proof}
    It is clear from the definition that $\ell$ is a $\Lambda$-length function. By Proposition~\ref{prop: median metric} $(S_X,d)$ is a median $\Lambda$-metric space. For any $f \in S_X - \{\id\}$ we have $f^\perp = \{\id\}$ so $S_X$ has rank 1 by Lemma~\ref{lem: rank}. For any $f \in S_X$, the map $t \in [0,\ell(f)] \mapsto f|_{(0,t]} \in S_X$ is a geodesic from $\id$ to $f$. Therefore $S_X$ is a $\Lambda$-tree by Lemma~\ref{lem: tree characterisation}.

    Fix an element $\lambda \in \Lambda$ with $\lambda > 0$. 
    If $x$ and $x^*$ are in different $\Lambda$-orbits then the constant map $f_x: (0,\lambda] \rightarrow X$ with image $x$ is admissible (by a similar argument to the proof of Lemma~\ref{lem: admissible constants}). Assume that $x$ and $x^*$ are in different $\Lambda$-orbits for all $x \in X$. For all $x \neq y$ we have $f_x \wedge f_y = \id$ and for any $g \in S_X$ there is some $x$ such that $f_x \wedge g > \id$, so the set of directions at $\id$ is in bijection with $X$.
\end{proof}

\begin{proof}[Proof of Proposition~\ref{prop: Lambda tree result}]
    Let $X \coloneqq (\Lambda/H) \sqcup (\Lambda/H)'$ equipped with the obvious $\Lambda$-action and define $*: X \rightarrow X$ by $(\lambda + H)^* = (-\lambda + H)'$ and ${(\lambda + H)'}^* = -\lambda + H$ for all $x + H \in \Lambda / H$. This defines an action of $\Lambda \rtimes_\alpha \la * \ra$ on $X$. Let $(S_X, \star)$ be the resulting group, equipped with its canonical $\Lambda$-metric $d$. Then $S_X$ is a $\Lambda$-tree with valence $|X| = 2|\Lambda / H|$. Moreover, for each $x \in X$, there is subspace $\mathbf{L}_x \coloneqq \{f \in S_X : f$ is constant with image $x$ or $x^*\}$ and the map $\varphi: \mathbf{L}_x \rightarrow \Lambda$, defined by $\varphi(f) = \ell(f)$ if the image of $f$ is $x$ and $\varphi(f) = -\ell(f)$ otherwise, is an isometry. Then, $\varphi(\Stab_{S_X}(\mathbf{L}_x) \cdot \id) = H$ so this completes the proof.
\end{proof}

\begin{remark}[2-torsion]
    If $\Lambda$ is not 2-divisible then $S_X$ can contain some order 2 elements which act by inversions (i.e. fixed point free order two isometries). Suppose that $\lambda \in \Lambda$ is such that $2\lambda' \neq \lambda$ for all $\lambda' \in \Lambda$ and let $X \coloneqq \Lambda / \la \lambda \ra$, equipped with the natural action of $\Lambda$ and the involution $(t + \la \lambda \ra)^* = -t + \la \lambda \ra$. Let $S_X$ be the resulting group and $\Lambda$-tree. Let $f:(0,\lambda] \rightarrow X$ be the constant map with image $x \coloneq 0 + \la \lambda \ra$. If $f = a \con b \con b^{-1} \con c$ then $2 \ell(b) \in \la \lambda \ra$ but since $0 \leq \ell(b) < \lambda$ and $\lambda$ is not 2-divisible this can only happen if $b = \id$. So $f \in S_X$ and $f^{-1}$ is constant with image $- \lambda \cdot x = x$, so $f^{-1} = f$.

    The only fixed point free isometries of $\Lambda$-trees are inversions and hyperbolic isometries (which have infinite order) so the only finite order elements in any $S_X$ have order 2. Moreover $S_X$ can only contain inversions if $\Lambda$ is not 2-divisible \cite[Lemma~3.1.2, Theorem~5.1.4]{Chiswell-intro_to_lambda_trees}.
\end{remark}

\section{Actions on products of $\mathbb{R}$-trees}

In this section, we present two distinct constructions of free transitive actions on $\ell^1$ products of $\mathbb{R}$-trees. In Subsection~\ref{sec: higher rank}, we extend the ideas of Subsection~\ref{sec: TX construction} to produce groups which act on products of $\mathbb{R}$-trees with arbitrary flat stabilisers. To illustrate the flexibility of this construction, we will prove a variant of Theorem~\ref{thm: centraliser spectrum} (Theorem~\ref{thm: arbitrary flats}). After establishing some facts about reducible actions on products of $\mathbb{R}$-trees in Subsection~\ref{sec: reducible actions}, we will use the construction from Subsection~\ref{sec: higher rank} to prove the existence of a free transitive and irreducible action on a product of two $\mathbb{R}$-trees in Subsection~\ref{sec: irreducible action}. In Subsection~\ref{sec: embedding BMW groups}, we present an entirely different construction which allows one to isometrically embed any BMW group with a positive BMW presentation into a free dense action on a product of two $\mathbb{R}$-trees.

\subsection{Actions with prescribed flat stabilisers} \label{sec: higher rank}

Let $N \in \{\mathbb{N}\} \cup \{\{1, \dots, n\}: n \in \mathbb{N}\}$, let $\mathbf{R} \coloneqq \ell^1(N)$ be equipped with its natural additive group structure, and, for each $n \in N$, let $\chi_n \in \mathbf{R}$ be the characteristic map of $n$.

Recall that $\mathcal{K}$ is the set of cardinals $\kappa$ such that $\kappa \leq 2^{\aleph_0}$ and let $\Sub_{D}(\mathbf{R})$ be the set of dense subgroups $H \leq \mathbf{R}$. Let $\overline{\Sub}_{D}(\mathbf{R})$ be the quotient of $\Sub_{D}(\mathbf{R})$ under linear isometries and denote the equivalence class of each $H \in \Sub_D(\mathbb{R})$ by $[H]$.

\begin{theorem} \label{thm: arbitrary flats}
    Let $\iota: \overline{\Sub}_{D}(\mathbf{R}) \rightarrow \mathcal{K}$ and $\eta: N \rightarrow \{0,1\}$ be arbitrary maps such that $\eta$ is non-zero and $\iota$ is supported on $\leq 2^{\aleph_0}$ elements of $\overline{\Sub}_{D}(\mathbf{R})$. 
    For each $n \in N$ such that $\eta(n) = 1$, let $T_n$ be the universal real tree with valence $2^{\aleph_0}$ and, for all $n \in N$ such that $\eta(n) = 0$, let $T_n \coloneqq \mathbb{R}$. Let $\mathcal{T} \coloneqq (T_n)_{n \in N}$.
    Then there exists a group $G$ which acts freely and transitively on the $\ell^1$ product $\ell^1(\mathcal{T}, z)$, for some $z \in \prod_{n \in N} T_n$, such that the following holds. For each $[H] \in \overline{\Sub}_{D}(\mathbf{R})$, let $A_H$ be the set of $G$-orbits of maximal flats $F \subseteq S$ such that $\Stab_G(F) \acting F$ is isomorphic to $H \acting \mathbf{R}$.
    Then $|A_H| = \iota([H])$.
\end{theorem}

For each $n \in N$, let $X_n$ be a set equipped with an action of $\mathbf{R}$ and let $Y_n \coloneqq X_n \sqcup X_n^*$, where $X_n^*$ is a copy of $X_n$ and $*:X_n \rightarrow X_n^*$ is a bijection. For each $x^* \in X_n^*$, let $(x_n^*)^* \coloneqq x_n$, so $\ast:Y_n \rightarrow Y_n$ is an involution. For each $x_n^* \in X_n^*$ and $r \in \mathbf{R}$, let $r \cdot x_n^* \coloneqq (r \cdot x_n)^*$.

\begin{remark}
    This setup is not quite analogous to that of Section~\ref{sec: TX construction sum}: we have an action of $\mathbf{R} \times \la \ast \ra$ on $Y_n$ rather than an action of $\mathbf{R} \rtimes_\alpha \la \ast \ra$, where $\alpha: \mathbf{R} \rightarrow \mathbf{R}$ is the order reversing automorphism given by $\alpha(r) = -r$ for all $r \in \mathbf{R}$. 
\end{remark}

Recall that, for each $n$, $\mathcal{Z}_{Y_n}$ denotes the set of equivalence classes of realisations of sequences $(x_p)_{p \in P} \subseteq Y_n$, where $P \subseteq [0,\ell]$ is a template and $\ell > 0$ (see Section~\ref{sec: templates}).

\begin{definition} \label{def: higher rank ore}
    Let $\mathcal{F} \coloneqq \{\f = (\f_n)_{n \in N} \in \prod_{n \in N} \mathcal{Z}_{Y_n} : \sum_{n \in N} \ell(\f_n) < \infty\}$. For each $n \in N$, let $\ell_n(\f) \coloneqq \ell(\f_n)$. Let $\mathfrak{o}(\f) \coloneqq (\ell_n(\f))_{n \in N} \in \mathbf{R}$; we call $\mathfrak{o}(\f)$ the \textit{outline} of $\f$.
    Let $\id \in \mathcal{F}$ be the unique element with outline 0.

    Set $\f \preceq \g$ if and only if $\f_n \preceq \g_n$ for all $n \in N$.
\end{definition}

\begin{definition}[Signed length, signed outline]
    Given $n \in N$ and $\f \in \mathcal{Z}_{Y_n}$, let $P \subseteq [0,\ell(\f)]$ be a template and $(x_p)_{p \in P}$ be a sequence with realisation $\f$. For each $p \in P - \{\ell\}$, let $p' \in P$ be the successor of $p$. Define:
    \begin{align*}
        &\ell^+(\f) \coloneqq \sum \{p' - p : p \in P - \{\ell\} \text{ and } x_p \in X_n\} \\
        &\ell^-(\f) \coloneqq \sum \{p' - p : p \in P - \{\ell\} \text{ and } x_p \in X_n^*\}
    \end{align*}
    The \textit{signed length} of $\f$ is $\sigma(\f) \coloneqq \ell^+(\f) - \ell^-(\f) \in \mathbb{R}$.

    If $\f = (\f_n)_{n \in N} \in \mathcal{F}$ then the \textit{signed outline} of $\f$ is $\tau(\f) \coloneqq (\sigma(\f_n))_{n \in N} \in \mathbf{R}$.
\end{definition}

\begin{definition} \label{def: higher rank ore operations}
\begin{itemize}
    \item Given $\f \in \mathcal{F}$, let $\f^{-1} \coloneqq (\f_n^{-1})_{n \in N}$, where $f_n^{-1}:[0,\ell_n(\f)] \rightarrow X_n$ is defined by $f_n^{-1}(t) = -\tau(\f) \cdot f_n^*(\ell_n(\f) - t)$ for each $n$.
    \item Define an operation $\con: \mathcal{F} \times \mathcal{F} \rightarrow \mathcal{F}$ as follows. Let $\f, \g \in \mathcal{F}$ and, for each $n$, let $a_n : [0,\ell_n(\f) + \ell_n(\g)] \rightarrow X_n$ be the map given by 
    \[a_n (t) = 
    \begin{cases}
        f_n(t) \quad &\text{if } t \in [0,\ell_n(\f)] ;\\
        \tau(\f) \cdot g_n(t - \ell_n(\f)) \quad &\text{otherwise.}
    \end{cases}
    \]
    and let $\mathfrak{a}_n \in \mathcal{Y}_{X_n}$ be the equivalence class of $a_n$. Define $\mathfrak{f} \con \mathfrak{g} \coloneqq (\mathfrak{a_n})_{n \in N}$.
\end{itemize}
\end{definition}

\begin{remark} \label{rem: concatenating tau}
If $\f = (\f_n)_{n \in N}, \g = (\g_n)_{n \in N} \in \mathcal{F}$, then $\tau(\f \con \g) = \tau(\f) + \tau(\tau(\f) \cdot \g) = \tau(\f) + \tau(\g)$, since $X_n$ and $X_n^*$ are $\mathbf{R}$-invariant for all $n \in N$.
\end{remark}
\begin{remark} \label{rem: inverse tau}
    If $\f = (\f_n)_{n \in N} \in \mathcal{F}$ then $\tau(\f^{-1}) = \tau(\f^*) = - \tau(\f)$.
\end{remark}

\begin{lemma} \label{lem: higher rank monoid}
    $(\mathcal{F}, \con)$ is a cancellative monoid with identity $\id$ and with involution $-1$. 
\end{lemma}
\begin{proof}
    Given $\f,\g,\h \in \mathcal{F}$, we have $\mathfrak{o((f \con g) \con h)} = \mathfrak{o}(\f) + \mathfrak{o}(\g) + \mathfrak{o}(\h) = \mathfrak{o(f \con (g \con h))}$. Using Remark~\ref{rem: concatenating tau}, we have that, for each $n \in N$ and $t \in [0, \ell_n(\f) + \ell_n(\g) + \ell_n(\h)]$:
    \begin{align*}
        ((f \con g) \con h)_n(t) &=
        \begin{cases}
            f_n(t) \quad &\text{if } t \in [0,\ell_n(\f)]; \\
            \tau(\f) \cdot g_n(t - \ell_n(\f)) &\text{if } t \in (\ell_n(\f), \ell_n(\g)];\\
            (\tau(\f) + \tau(\g)) \cdot h(t - \ell_n(f) - \ell_n(g)) &\text{otherwise}
        \end{cases} \\
        &= (f \con (g \con h))_n(t).
    \end{align*}
    Therefore $((\f \con \g) \con \h)_n = (\f \con (\g \con \h))_n$ for each $n \in N$, which implies that $\mathfrak{(f \con g) \con h = f \con (g \con h)}$. Thus $\con$ is associative. It is clear from the definitions that $\id$ is a two-sided identity for $\con$, so $(\mathcal{F}, \con)$ is a monoid. 
    
    Let $\f \in \mathcal{F}$. Then $(\f^{-1})^{-1} = (\aaa_n)_{n \in N}$ has outline $\mathfrak{o}(\f)$. Using Remark~\ref{rem: inverse tau}, we have that, for each $n \in N$ and $t \in [0,\ell_n(\f)]$, 
    \[
        a_n(t) = -\tau(\f^{-1}) \cdot (- \tau(\f) \cdot f_n^*(\ell_n(\f) - (\ell_n(\f) - t)))^* = \tau(\f) \cdot (- \tau(\f) \cdot f_n(t)) = f_n(t).
    \]
    Therefore $(\f^{-1})^{-1}= \f$ and $-1:\mathcal{F} \rightarrow \mathcal{F}$ is an involution.    
    Fix $\f, \g \in \mathcal{F}$. Then $(\f \con \g)^{-1} = (\aaa_n)_{n \in N}$, where for each $i$, $a_n:[0,\ell_n(\f) + \ell_n(\g)] \rightarrow X_n$ is defined by 
    \begin{align*}
    a_n(t) &= - \tau(\f \con \g) \cdot (f \con g)_n^*(\ell_n(\f) + \ell_n(\g) - t) \\
    &= 
    \begin{cases}
        - (\tau(\f) + \tau(\g)) \cdot (\tau(\f) \cdot g_n(\ell_n(\g) - t))^* \quad &\text{if } t \in [0, \ell_n(\g)) ; \\
        - (\tau(\f) + \tau(\g)) \cdot f_n^*(\ell_n(\f) - (t - \ell_n(\g))) \quad &\text{otherwise}
    \end{cases} \\
    &= \begin{cases}
        g_n^{-1}(t) \quad &\text{if } t \in [0, \ell_n(\g)) ; \\
        \tau(\g^{-1}) \cdot f_n^{-1}(t - \ell_n(\g)) \quad &\text{otherwise}
    \end{cases} \\
    &= (g^{-1} \con f^{-1})_n, \quad \text{unless } t = \ell_n(\g).
    \end{align*}
    Therefore $(\mathfrak{f} \con \mathfrak{g})^{-1} = \mathfrak{g}^{-1} \con \mathfrak{f}^{-1}$.

    Finally, let $\f,\g,\h \in \mathcal{F}$ and suppose that $\f \con \g = \f \con \h$. Then $\mathfrak{o(g) = o(h)}$ and, for each $n$ and for all but countably many $t \in [0, \ell_n(\g)] = [0,\ell_n(\h)]$, we have that $\tau(f) \cdot g_n(t) = \tau(f) \cdot h_n(t)$ so $g_n(t) = h_n(t)$. Therefore $\g = \h$. It is clear that if $\f \con \g = \h \con \g$ then $\f = \h$, so $(\mathcal{F}, \preceq)$ is cancellative.
\end{proof}

\begin{lemma} \label{rem: admissibility invariance}
    If $\f \in \mathcal{F}$ is admissible and $r \in \mathbf{R}$, then the equivalence class $r \cdot \f$ of $(r \cdot f_n)_{n \in N}$ is admissible.
\end{lemma}
\begin{proof}
    Suppose $\aaa, \bbb, \ccc \in \mathcal{F}$ are such that $r \cdot \f = \aaa \con \bbb \con \bbb^{-1} \con \ccc$. Since $X_n$ and $X_n^*$ are $\mathbf{R}$-invariant for each $n \in N$, we have $\tau(\g) = \tau(-r \cdot \g)$ for all $\g \in \mathcal{F}$. It follows that $\f =  -r \cdot \aaa \con - r \cdot \bbb \con -r \cdot \bbb^{-1} \con - r \cdot \ccc$. But $(-r \cdot \bbb)^{-1} = (\aaa_n)_{n \in N}$ where, for each $n \in N$ and $t \in [0, \ell_n(\bbb)]$,
    \begin{align*}
        a_n(t) 
        &= - \tau(-r \cdot \bbb) \cdot (-r \cdot b_n(\ell_n(\bbb) - t))^* \\
        &= (-\tau(\bbb) - r) \cdot b_n^*(\ell_n(\bbb) - t) \\
        &= - r \cdot b_n^{-1}(t).
    \end{align*}
    Thus $(-r \cdot \bbb)^{-1} = - r \cdot \bbb$ and, since $\f$ is admissible, this implies that $\bbb = \id$.
\end{proof}

\begin{lemma} \label{lem: product ore v1}
    ($\mathcal{F},\preceq, \id, \con, -1)$ is an ore.
\end{lemma}
\begin{proof}
(O2) holds by Lemma~\ref{lem: higher rank monoid}. Let us go through the remaining axioms.
\begin{itemize} 
    \item[(O1)] This follows from the fact that $(\mathcal{Y}_{X_i}, \preceq)$ is a median semilattice for each $i$ (by Lemma~\ref{lambda trees: median semilattice}).
    \item[(O3)] Let $\mathfrak{f},\mathfrak{g} \in \mathcal{F}$. It is clear from the definition of $\con$ that, if there exists $\mathfrak{h} \in \mathcal{F}$ such that $\f = \g \con \h$, then $\g \preceq \f$. Conversely, suppose that $\g \preceq \f$. For each $n \in N$ define $h_n:[0,\ell_n(\f) - \ell_n(\g)] \rightarrow X$ by $h(t) \coloneqq -\tau(\g) \cdot f(t + \ell_n(\g))$. Then $\f = \g \con \h$.
    \item[(O4)] By Lemma~\ref{lambda trees: median semilattice}, each $(\mathcal{Y}_{X_n}, \preceq)$ is a median semilattice. Let $m_n: \mathcal{Y}_{X_n}^3 \rightarrow \mathcal{Y}_{X_n}$ be the median map from Theorem~\ref{thm: median semilattice}. For each $\f,\g,\h \in \mathcal{F}$ let $m(\f,\g,\h) \coloneqq (m_n(\f_n, \g_n, \h_n))_{n \in N}$. One can check that $m(\f,\g,\h) = (\f \wedge \g) \vee (\g \wedge \h) \vee (\f \wedge \h)$. Let $\mathfrak{x = (x_n)_{n \in N}} \in \mathcal{F}$. Then $m(\mathfrak{x \wedge f, x \wedge g, x \wedge h}) = (m_n(\mathfrak{x}_n \wedge \f_n, \mathfrak{x}_n \wedge \g_n, \mathfrak{x}_n \wedge \h_n))_{n \in N} = (\mathfrak{x}_n \wedge m_n(\f_n, \g_n, \h_n))_{n \in N} = \mathfrak{x} \wedge m(\f,\g,\h)$.
    \item[(O5)] Let $\f,\g,\h \in \mathcal{F}$ and suppose that $\f,\g$ are admissible, $\f \perp \h$ and $\f \vee \h = \f \con \g$. Let $I_\f \coloneqq \{n \in N : \ell_n(\f) > 0\}$ and $I_\h \coloneqq \{n \in N : \ell_n(\h) > 0\}$. Then $I_\f \cap I_\h = \emptyset$ and $\f \vee \h = \mathfrak{a}$ where $\aaa_n = \f_n$ if $n \in I_\f$, $\aaa_n = \h_n$ if $n \in I_\h$ and $a_n: \{0\} \rightarrow X_n$ is any map of length 0 otherwise. Therefore, up to choosing a different representative for $\g$, we have that $\h = \tau(\f) \cdot \g$. By Remark~\ref{rem: admissibility invariance}, $\h$ is admissible.
    \item[(O6)] Let $\f, \g \in \mathcal{F}$ be admissible and orthogonal. Let $I_\f \coloneqq \{n \in N : \ell_n(\f) > 0\}$ and $I_\g \coloneq \{n \in N : \ell_n(\g) > 0\}$. Then $I_\f \cap I_\g = \emptyset$ and $\f \vee \g = \h$ where $\h_n = \f_n$ if $n \in I_f$, $\h_n = \g_n$ if $n \in I_\g$ and $h_n: \{0\} \rightarrow X_n$ is any map of length 0 otherwise. It follows that $\f \vee \g = \f \con -\tau(\f) \cdot \g = \g \con -\tau(\g) \cdot \f$. By Remark~\ref{rem: admissibility invariance}, $-\tau(\f) \cdot \g$ and $-\tau(\g) \cdot \f$ are admissible. Also, $\f^{-1} \vee -\tau(f) \cdot \g = \f^{-1} \con \g = -\tau(f) \cdot \g \con -\tau(\g) \cdot \f^{-1}$ the fact that $I_\f \cap I_\g = \emptyset$ implies that $\f^{-1} \perp -\tau(\f) \cdot \g$ and $\g^{-1} \perp -\tau(\g) \cdot \f$. 
    \qedhere
\end{itemize}
\end{proof}

\begin{definition} \label{def: product group}
Let $(G,\twist)$ be the group extracted from $(\mathcal{F},\con)$. 
\end{definition}

\begin{remark} \label{rem: constants are admissible}
    Any element $\f \in \mathcal{F}$ such that each $\f_n$ is constant is admissible and is therefore an element of $G$. Indeed if $\f$ is inadmissible then the image of some $\f_n$ must contain both some $x \in X_n$ and an element in the orbit of $x^*$, which is contained in $X_n^*$, so $\f_n$ is not constant.
\end{remark}
Let $\ell: \mathcal{F} \rightarrow \mathbb{R}$ be the map defined by $\ell(\mathfrak{f}) = \sum_{n \in N} \ell_n(\mathfrak{f})$. Then $\ell$ is a length function. Let $d$ be the resulting metric on $G$ given by $d(\f,\g) = \ell(\f^{-1} \star \g)$ for all $\f, \g \in G$ (see Proposition~\ref{prop: median metric}).

\medskip

We will need a few lemmas to prove the following:

\begin{proposition} \label{prop: product of trees}
The metric space $(G,d)$ is the $\ell^1$ product of $|N|$ complete real trees, with respect to some basepoint.
\end{proposition}

\begin{definition}
    For each $n \in N$, let $\mathcal{Z}_n \coloneqq \{\f \in \mathcal{F} : \f_i = \id \; \forall \; i \neq n\}$ and let $\psi_n:\mathcal{Z}_n \rightarrow \mathcal{Z}_{Y_n}$ be the bijection given by $\psi_n(\f) = \f_n$ for all $\f \in \mathcal{Z}_n$.

    Let $T_n \coloneqq \{\f_n \in \mathcal{Z}_{Y_n} : \psi_n^{-1}(\f_n) \in G\} \subseteq \mathcal{Z}_{Y_n}$ for each $n \in N$.
\end{definition}

\begin{lemma} \label{lem: yet another real tree}
\begin{enumerate}
    \item There is an operation $\con$ and an involution $-1$ on $\mathcal{Z}_{Y_n}$ such that $(\mathcal{Z}_{Y_n}, \preceq,$ $\con, -1, \id)$ is an ore and $T_n$ is its set of admissible elements. 
    \item Let $(T_n, \star)$ be the group extracted from $\mathcal{Z}_{Y_n}$. Then $\ell$ is a length function on $T_n$ and, if $d_n$ is the metric on $T_n$ given by $d_n(\f_n,\g_n) = \ell_n(\f_n^{-1} \star \g_n)$ for all $\f_n, \g_n \in T_n$ then $(T_n,d_n)$ is a complete real tree.
    \item If $|X_n| = 1$ then $T_n$ is a line and if $2 \leq |X_n| \leq 2^{\aleph_0}$ then $T_n$ is the universal real tree with valence $2^{\aleph_0}$.
\end{enumerate}
\end{lemma}
\begin{proof}
\begin{enumerate} 
    \item Observe that $\mathcal{Z}_{n}$ is invariant under $\con$ and $-1$ and that $\id \in \mathcal{Z}_{n}$. Also if $\g \in \mathcal{F}$ and $\g \preceq \f$ then $\g \in \mathcal{Z}_n$. Thus it follows from Lemma~\ref{lem: product ore v1} that $(\mathcal{Z}_n, \preceq, \con, -1, \id)$ is an ore. The bijection $\psi_n$ then endows $\mathcal{Z}_{Y_n}$ with the structure of an ore. By definition, the set of admissible elements of $\mathcal{Z}_{Y_n}$ is $T_n$. 
    
    \item It is immediate from the definition that $\ell$ is a length function on $\mathcal{Z}_{Y_n}$. Thus, by Proposition~\ref{prop: median metric}, $(T_n, d_n)$ is a median space. For all $\f_n \in T_n$ we have $\f_n^\perp = \{\id\}$ so, by Lemma~\ref{lem: rank}, $T_n$ has rank 1. Given $\f_n \in T_n$ and $t \in [0,\ell(\f_n)]$, let $\gamma(t) \coloneqq \g_n$, where $g_n = f_n|_{[0,t]}$. Then $\g_n \in T_n$ and $\gamma:[0,\ell(\f_n)] \rightarrow T_n$ is a geodesic from $\id$ to $\f_n$. Therefore $T_n$ is an $\mathbb{R}$-tree by Lemma~\ref{lem: tree characterisation}.

    The argument to show that $T_n$ is complete is very similar to the argument showing that $T_X(Y)$ is closed in the proof of Proposition~\ref{prop: template generation}: Suppose that $(\aaa_i)_{i \in \mathbb{N}} \subseteq T_n$ is a Cauchy sequence and, for each $i \in \mathbb{N}$, let $P_i \subseteq [0, \ell(\aaa_i)]$ be a template and $(a_p)_{p \in P_i} \subseteq Y_n$ be a sequence with realisation $\aaa_i$. Since $T_n$ is downward closed, we can assume that $(\aaa_i)_{i \in \mathbb{N}}$ is strictly increasing. Refine each $P_i$ and $(a_p)_{p \in P_i}$ using Lemma~\ref{lem: refining} so that $\ell(\aaa_{i-1}), \ell(\aaa_i) \in \overline{P_i}$ for each $i \in \mathbb{N}$, where $\aaa_0 \coloneq \id_n$. Then let $P \coloneqq \cup_{i \in \mathbb{N}} (P_i \cap [\ell(\aaa_{i-1}), \ell(\aaa_i)])$, observe that $P$ is a template for $[0, \lim_{i \rightarrow \infty} \ell(\aaa_i)]$, and let $\aaa \in \mathcal{Z}_n$ be the equivalence class of the sequence $(a_p)_{p \in P}$. It follows from the admissibility of the $\aaa_i$'s that $\aaa \in T_n$. By construction, $(\aaa_i)_{i \in \mathbb{N}}$ converges to $\aaa$.

    \item Suppose that $|X_n|= 1$ and let $L \coloneqq \{\f_n \in \mathcal{Z}_{Y_n} : \f_n$ is constant$\}$. By Remark~\ref{rem: constants are admissible}, $L \subseteq T_n$ and the fact that $|X_n| = 1$ implies that $L$ is isometric to $\mathbb{R}$. Let $\ell > 0$, let $P \subseteq [0,\ell]$ be a template and let $\f_n \in \mathcal{Z}_{Y_n}$ be the realisation of a sequence $(y_p)_{p \in P} \subseteq Y_n$. Suppose the CB-rank of $P$ is as small as possible. If $\f_n$ is not constant then there exists $p_1, p_2 \in P - \{\ell\}$ such that $p_2$ is the successor of $p_1$ and either $y_{p_1} \in X_n$ and $y_{p_2} \in X_n^*$ or $y_{p_1} \in X_n^*$ and $y_{p_2} \in X_n$. In either case, this implies that $\f$ is inadmissible. Therefore $T_n = L$. 

    If $2 \leq |X_n| \leq 2^{\aleph_0}$ then it follows by the same argument as the proof of Proposition~\ref{prop: universal real tree} that $T_n$ is the universal real tree with valence $2^{\aleph_0}$.
    \qedhere
\end{enumerate}
\end{proof}

\begin{lemma} \label{lem: admissible factors in F}
    Consider $\f = (\f_n)_{n \in N} \in \mathcal{F}$. Then $\f$ is admissible if and only if $\f_n \in T_{n}$ for each $n \in N$.
\end{lemma}
\begin{proof}
    Suppose $\f$ is admissible and there exists $n \in N$ such that $\f_n = \aaa_n \con \bbb_n \con \bbb_n^{-1} \con \ccc_n$ for some $\aaa_n, \bbb_n, \ccc_n \in \mathcal{Z}_{Y_n}$. For each $i \in N - \{n\}$ define $\aaa_i \coloneqq \bbb_i \coloneqq \id$ and $\ccc_i \coloneqq (-\sigma(\aaa_n) \chi_n) \cdot \f_i$. Let $\aaa \coloneqq (\aaa_i)_{i \in N}, \bbb \coloneqq (\bbb_i)_{i \in N}$ and $\ccc \coloneqq (\ccc_i)_{i \in N}$. Then $\f = \aaa \con \bbb \con \bbb^{-1} \con \ccc$, so $\bbb_n = \id$.

    Conversely, suppose that $\f_n \in T_n$ for all $n \in N$. Let $\aaa, \bbb, \ccc \in \mathcal{F}$ be such that $\f = \aaa \con \bbb \con \bbb^{-1} \con \ccc$. For each $n \in N$, $\f_n$ is the equivalence class of the map $f_n: [0,\ell_n(\f)] \rightarrow Y_n$ given by:
    \begin{align*}
        f_n(t) 
        &= 
        \begin{cases}
            a_n(t) \quad &\text{if } t \in [0,\ell(\aaa_n)]; \\
            \tau(\aaa) \cdot b_n(t - \ell(\aaa_n)) &\text{if } t \in (\ell(\aaa_n), \ell(\aaa_n) + \ell(\bbb_n)]; \\
            (\tau(\aaa) + \tau(\bbb)) \cdot b_n^{-1}(t - \ell(\aaa_n) - \ell(\bbb_n)) &\text{if } t \in (\ell(\aaa_n) + \ell(\bbb_n), \ell(\aaa_n) + 2 \ell(\bbb_n)]; \\
            \tau(\aaa) \cdot c_n(t - \ell(\aaa_n) - 2 \ell(\bbb_n)) &\text{if } t \in (\ell(\aaa_n) + 2 \ell(\bbb_n), \ell(\f_n)]
        \end{cases}
        \\
        &=
        \begin{cases}
            a_n(t) \quad &\text{if } t \in [0,\ell(\aaa_n)]; \\
            \tau(\aaa) \cdot b_n(t - \ell(\aaa_n)) &\text{if } t \in (\ell(\aaa_n), \ell(\aaa_n) + \ell(\bbb_n)]; \\
            \tau(\aaa) \cdot b_n^*(2\ell(\bbb_n) - t + \ell(\aaa_n)) &\text{if } t \in (\ell(\aaa_n) + \ell(\bbb_n), \ell(\aaa_n) + 2 \ell(\bbb_n)]; \\
            \tau(\aaa) \cdot c_n(t - \ell(\aaa_n) - 2 \ell(\bbb_n)) &\text{if } t \in (\ell(\aaa_n) + 2 \ell(\bbb_n), \ell(\f_n)]
        \end{cases}
    \end{align*}
    for all $t \in [0, \ell(\f_n)]$. Let $\widehat{\tau} \in \mathbf{R}$ be such that $\widehat{\tau}(n) = \tau(\aaa)(n)$ and $\widehat{\tau}(i) = 0$ if $i \neq n$. Let $\widehat{\bbb}_n \coloneqq \widehat{\tau} \cdot \bbb_n \in \mathcal{Z}_{Y_n}$ and $\widehat{\ccc}_n \coloneqq \widehat{\tau} \cdot \ccc_n \in \mathcal{Z}_{Y_n}$. Then ${\widehat{\bbb}_n}^{-1}$ is the equivalence class of the map $[0,\ell(\bbb_n)] \rightarrow Y_n$ given by $t \mapsto (\widehat{\tau} -\tau(\psi_n^{-1}(\widehat{\bbb}_n))) \cdot b_n^*(\ell(\bbb_n) - t)$. It follows that
    $\aaa_n \con \widehat{\bbb}_n \con (\widehat{\bbb}_n)^{-1} \con \widehat{\ccc}_n = \f_n$. Therefore $\widehat{\bbb}_n = \id$ which, by Remark~\ref{rem: admissibility invariance}, implies that $\bbb_n = \id$. As this holds for all $n \in N$, we have that $\bbb = \id$ and $\f$ is admissible.  
\end{proof}

\begin{proof}[Proof of Proposition~\ref{prop: product of trees}]
Let $\mathcal{T} \coloneqq (T_n)_{n \in N}$.
Lemma~\ref{lem: admissible factors in F} implies that, as a set, $G$ is the $\ell^1$ product $\ell^1(\mathcal{T}, \id)$.
For all $\f, \g \in G$ we have  $d(\f,\g) =\ell(\f^{-1} \star \g) = \sum_{n \in N} d_n(\f_n, \g_n)$. Therefore $d$ is precisely the $\ell^1$ metric on $G = \ell^1(\mathcal{T}, \id)$ and, by Lemma~\ref{lem: yet another real tree}, each $T_n$ is a complete real tree.
\end{proof}

\begin{definition}
    A \textit{standard flat} in $G$ is a maximal flat $F \subseteq G$ of the form 
    \[F = \{\f \in G : \f_n \text{ is constant with image } x_n \text{ or } x_n^* \; \forall \; n \in N\}, \]
    for some $(x_n)_{n \in N} \in \prod_{n \in N} X_n$. We denote $F((x_n)_{n \in N}) \coloneqq F$.
\end{definition}

The stabilisers of maximal flats are straightforward to describe:

\begin{lemma} \label{lem: standard flat stabilisers}
    Let $(x_n)_{n \in N} \in \prod_{n \in N} X_n$ and let $F \coloneqq F((x_n)_{n \in N})$. Then 
    \[\Stab_G(F) = \{ \f \in F : \tau(\f) \in \cap_{n \in N} \Stab_{\mathbf{R}}(x_n)\} \cong \cap_{n \in N} \Stab_{\mathbf{R}}(x_n). \]
    The restriction of $\tau$ to $F$ is an isometry $F \rightarrow \mathbf{R}$ and the restriction $\tau|_{\Stab_G(F)}: \Stab_G(F)) \rightarrow \cap_{n \in N} \Stab_{\mathbf{R}}(x_n)$ is an isomorphism.
\end{lemma}

\begin{lemma} \label{lem: only standard flats have fun stabilisers}
    Let $F \subseteq G$ be a maximal flat and let $H \leq \mathbf{R}$ be such that $\Stab_G(F) \acting F$ is isomorphic to $H \acting \mathbf{R}$. If $H$ is dense in $\mathbf{R}$ then $F$ is a translate of a standard flat.
\end{lemma}
\begin{proof}
    We can replace $F$ by a $G$-translate so that $\id \in F$. Note that $\Stab_{G}(F) \subseteq F$ in this case. Let $\f \in \Stab_{G}(F)$ be non-trivial. Up to translating $F$ by another element of $G$, we can assume that, for all $n \in N$, there exists a non-trivial element $\f_n' \preceq \f_n$ which is constant, say with image $y_n \in Y_n$. 
    
    Let $\varphi: F \rightarrow \mathbf{R}$ be an isometry such that $\varphi(\id) = 0$ and, for each $n \in N$, we have $\varphi^{-1}(\mathbb{R} \cdot \chi_n) = \cap_{i \neq n} p_i^{-1}(\id)$.

    Let $E \coloneqq \varphi(\Stab_{G}(F))$ and note that $E \in [H]$, so in particular $E \in \Sub_{D}(\mathbf{R})$.
    For each $n \in N$ let $p_n: \mathbf{R} \rightarrow \mathbb{R}$ be the projection defined by $p_n((t_i)_{i \in N}) = t_n$. The image $p_n(E)$ is dense in $\mathbb{R}$ for each $n \in N$. 
    Fix $n \in N$ and let $\g \in \Stab_{G}(F)$ be such that $\id \precneq \g_n \precneq \f_n'$. Let $\h \coloneqq \g \twist \f$, so for all but countably many $t \in [0,\ell(\g_n)]$, $h_n(t) = g_n(t) = y_n$ and for all but countably many $t \in (\ell(\g_n), \ell(\g_n) + \ell(\f_n)]$, $h_n(t) = \tau(g) \cdot y_n$. But $\h \in F$, $\g_n \preceq \h$ and $\ell(\h_n) > \ell_n(\f_n)$ so $\f_n \preceq \h_n$, which implies that, for all but countably many $t \in [\ell(\g_n), \ell(\f_n)]$, we have $h_n(t) = f_n(t) = y_n$. Therefore $\tau(\g) \cdot y_n = y_n$ and $\h_n$ is constant with image $y_n$. It follows by induction on $k \in \mathbb{N}$ that the $n$th coordinate of $\g^k \star \f$ is constant with image $y_n$. It follows that the $n$th coordinate of $\g^k$ is constant with image $y_n$ for all $k$ and therefore the $n$th coordinate of $\g^{-k}$ is constant with image $y_n^*$. Thus, for all $\h \in F$, the element $\h_n$ is constant with image $y_n$ or $y_n^*$. The same argument holds for all $n \in N$ so we have shown that $F = F((x_n)_{n \in N})$, where, for each $n \in N$, $x_n \in X_n$ is such that $y_n = x_n$ or $y_n = x_n^*$.
\end{proof}

\begin{proof}[Proof of Theorem~\ref{thm: arbitrary flats}]
Let $\iota: \overline{\Sub}_{D}(\mathbf{R}) \rightarrow \mathcal{K}, \eta: N \rightarrow \{0,1\}$ be arbitrary maps such that $\eta$ is non-zero. Fix $m \in N$ such that $\eta(m) = 1$.
For each $[H] \in \overline{\Sub}_{D}(\mathbf{R})$, fix a representative $H \leq \mathbf{R}$ and let $B_H \coloneqq \mathbf{R} / H$, equipped with the natural action of $\mathbf{R}$ by addition. Let $D_H$ be the disjoint union of $\iota([H])$ copies of $B_H$.
If $\iota$ is the zero map then let $X_m \coloneqq B_{\{0\}}$, and if $\iota = \chi_\mathbf{R}$ then let $X_m \coloneqq B_\mathbf{R} \sqcup B_{\{0\}}$. Otherwise, let $X_m \coloneqq \sqcup_{H \in \Sub_{D}(\mathbf{R})} D_H$. 
If $ n \in N - \{m\}$ and $\eta(n) = 1$, then let $X_n \coloneqq B_\mathbf{R} \sqcup B_{\{0\}}$. If $n \in N$ and $\eta(n) = 0$, then let $X_n \coloneqq B_\mathbf{R}$. 

Let $\mathcal{F}$ be the resulting ore and let $G$ be its extracted group, equipped with the metric $d$ defined above. Then $(G,d)$ is the $\ell^1$ product $\ell^1(\mathcal{T}, \id)$. By Lemma~\ref{lem: yet another real tree}, for each $n \in N$, the metric space $(T_n,d_n)$ is the complete universal real tree with valence $2^{\aleph_0}$ if $\eta(n) = 1$ and $T_n$ is isometric to $\mathbb{R}$ if $\eta(n) = 0$.

Fix a subgroup $H \in \Sub_{D}(\mathbf{R})$. Given $(x_n)_{n \in \mathbb{N}} \in \prod X_n$, it follows from Lemma~\ref{lem: standard flat stabilisers} that, if $F = F((x_n)_{n \in \mathbb{N}}$ is the corresponding standard flat, then $G \cdot F \in A_H$ if and only if $\Stab_\mathbf{R}(x_m) \in [H]$ and $x_n \in B_\mathbf{R} \subseteq X_n$ for all $n \neq m$. Two standard flats $F((x_n)_{n \in \mathbb{N}}), F((y_n)_{n \in \mathbb{N}})$ whose stabilisers act with dense orbits
are in the same $G$-orbit if and only if $x_m, y_m$ are in the same $\mathbf{R}$-orbit. Therefore there are precisely $\iota([H])$ orbits of standard flats in $A_H$. It follows from Lemma~\ref{lem: only standard flats have fun stabilisers} that $|A_H| = \iota([H])$.
\end{proof}

\subsection{Reducible actions} \label{sec: reducible actions}

A natural question to ask when studying a group acting on a product space is whether it is ``reducible", either in the sense that it splits non-trivially as a direct product or that a subgroup large enough to encompass some of the geometry of the group splits non-trivially as a direct product. Given a group $G$ acting properly cocompactly on the product of two locally finite simplicial trees, one says that $G$ is \textit{reducible} if a finite index subgroup of $G$ splits non-trivially as a direct product. Inspired by this, we consider the following notions.
Fix a group $G$ and a space $X$ such that $X$ is isometric to the $\ell^1$-product of two unbounded real trees $T_1, T_2$.

An action of a group $G$ on a finite rank median space $X$ is called \textit{essential} (or sometimes \textit{minimal}, see \cite{Fioravanti-convex-cores}) if there is no proper $G$-invariant convex subspace of $X$.

\begin{definition}
    A free cobounded / essential action of $G$ on $X$ is \textit{coboundedly / essentially reducible} if there exists a subgroup $H \leq G$ which splits non-trivially as a direct product and such that the induced action of $H$ on $X$ is cobounded / essential.
\end{definition}

\begin{lemma} \label{lem: well-behaved factors}
    Let $T_1, T_2$ be $\mathbb{R}$-trees which are not isometric to $\mathbb{R}$ and let $X \coloneqq T_1 \times T_2$ be their $\ell^1$ product.
    Let $H = H_1 \times H_2 \leq \Isom(T_1) \times \Isom(T_2)$ be a subgroup such that $H_1,H_2 \neq \{\id\}$ and the action of $H$ on $X$ is free and essential. For each $i \in \{1,2\}$, let $p_i: H \rightarrow \Isom(T_i)$ be the canonical projection. Then, up to relabelling $H_1, H_2$, the following holds. Given any point $(z_1, z_2) \in X$, we have $H_1 = \Stab_H(T_1 \times \{z_2\}), \; H_2 = \Stab_H(\{z_1\} \times T_2)$ and the induced actions $H_1 \acting T_1 \times \{z_2\}, \; H_2 \acting \{z_1\} \times T_2$ are free and essential.

    It follows that, if the action of $H$ on $X$ is in addition cobounded / cocompact / transitive, then the $H_1 \acting T_1 \times \{z_2\}$ and the $H_2 \acting \{z_1\} \times T_2$ are cobounded / cocompact / transitive for each $i \in \{1,2\}$.
\end{lemma}
\begin{proof}
Since the action of $H$ on $T_1 \times T_2$ is essential, so are the actions $p_1, p_2$. Let us show that there is a hyperbolic element in $p_i(H_1 \cup H_2)$ for $i = 1,2$. Suppose to the contrary that every element in $p_i(H_1 \cup H_2)$ fixes a point in $T_i$. Let $h_1 \in H_1, h_2 \in H_2$ and $x_1 \in T_i$ be a point fixed by $p_i(h_1)$. Then $p_i(h_1)$ fixes $p_i(h_2) x_1$ and therefore fixes the segment $[x_1, p_i(h_2) x_1]$ pointwise. Since $p_i(h_2)$ is elliptic, the midpoint $x_2$ of the (possibly degenerate) segment $[x_1, p_i(h_2) x_1]$ is fixed by $p_i(h_2)$ as well as $p_i(h_1)$. Thus $p_i(h_1h_2)$ fixes a point. But by \cite[Theorem~C(1)]{Fioravanti-convex-cores} there is a hyperbolic element in $p_i(H)$ so this is a contradiction.

Now, up to relabelling $H_1, H_2$, we can assume that there exists an element $h \in H_1$ such that $p_1(h)$ is hyperbolic. Let $\ell \subseteq T_1$ be the axis of $p_1(h)$. Then, since $H_2$ commutes with $h$, $p_1(H_2)$ stabilises $\ell$. Moreover, if $p_1(H_2)$ contains a hyperbolic element $p_1(h_2)$, then its axis is also $\ell$ and, since $H_1$ commutes with $H_2$, this implies that $p_1(H_1)$ also stabilises $\ell$. Since the action of $p_1(H)$ is essential, this implies that $T_1 = \ell$. So we can assume that $p_1(H_2)$ contains no hyperbolic elements. Since $p_1(H_2)$ stabilises $\ell$, it follows that there is a point $x_0 \in \ell$ which is fixed by every element of $p_1(H_2)$. Let $h_1 \in H_1, h_2,h_2' \in H_2$. Then $p_1(h_2') p_1(h_1h_2) x_0 = p_1(h_1) x_0 = p_1(h_1h_2) x_0$. Therefore $p_1(H_2)$ fixes the $p_1(H)$-orbit of $x_0$ pointwise and therefore its convex hull, which is the entirety of $T$. Note that this also implies that the action of $p_1(H_1)$ on $T_1$ is essential.

Since the action of $H$ on $T$ is free, the action of $p_2(H_2)$ on $T_2$ must be non-trivial. If $p_2(H_1)$ contains a hyperbolic element then the above argument implies that $T_2$ is a line. Therefore $p_2(H_1)$ does not contain a hyperbolic element, so $p_2(H_2)$ does and the above argument implies that the action of $p_2(H_2)$ on $T_2$ is essential while the action of $p_2(H_1)$ is trivial. The proposition follows.
\end{proof}

\subsection{A free transitive and irreducible action}
\label{sec: irreducible action}

We can now prove the following:

\begin{corollary} \label{cor: irreducible product}
    Let $T_1 = T_2$ be the complete universal real tree with valence $2^{\aleph_0}$. There exists a group $G \leq \Isom(T_1) \times \Isom(T_2)$ which is essentially reducible but coboundedly irreducible.
\end{corollary}
\begin{proof}
Let $X_1 = X_2 = \mathbb{R}$. Define an action of $\mathbb{R}^2$ on $X_1$ and $X_2$ as follows. If $(r_1,r_2) \in \mathbb{R}^2$, $x_1 \in X_1$ and $x_2 \in X_2$, then $(r_1, r_2) \cdot x_1 = r_2 + x_1$ and $(r_1,r_2) \cdot x_2 = r_1 + x_2$. Let $(\mathcal{F}, \con, -1, \preceq, \id)$ be the ore constructed in Section~\ref{sec: higher rank} using these actions and let $(G, \twist)$ be the group extracted from $\mathcal{F}$. Let $d$ be the metric on $G$, so that $(G,d)$ is the $\ell^1$ product $T_1 \times T_2$. 

For each $i \in \{1,2\}$, let $p_i:\Isom(T_1) \times \Isom(T_2) \rightarrow \Isom(T_i)$ be the canonical projection.
Suppose $H \leq G$ splits non-trivially as a direct product and the induced action $H \acting T_1 \times T_2$ is essential. By Lemma~\ref{lem: well-behaved factors}, $H = H_1 \times H_2$, where $H_1 = \ker(p_2) \cap H$ and $H_2 = \ker(p_1) \cap H$. Fix $i \in \{1,2\}$ and let $\f \in T_i$ be a constant element with image $x_i \in X_i$ and length $\ell > 0$. If $\h \in H_i$, then $\sigma(\h) = 0$, so $\ell^-(\f^{-1} \star \h) \geq \ell^+(\f) = \ell$. Therefore $d_i(\f, \h) = \ell(\f^{-1} \star \h) \geq \ell$. Thus the action of $H_i$ on $T_i$ is not cobounded, which implies that the action of $H$ on $T_1 \times T_2$ is not cobounded.

To see that $G$ is essentially reducible in spite of being coboundedly irreducible, let $H_1 \coloneqq p_1(\ker(p_2))$ and $H_2 \coloneqq p_2 (\ker(p_1))$. Fix $i \in \{1,2\}$. Let $\f \in T_1$ and let $(y_p)_{p \in P} \subseteq Y_i$ be the sequence whose realisation is $\f$. For each $p \in P$, fix $y_p' \in Y_i - \{y_p\}$ such that, if $y_p \in X_i$ then $y_p' \in X_i$, and if $y_p \in X_i^*$ then $y_p' \in X_i^*$. Let $\f'$ be the realisation of $(y_p')_{p \in P}$. Then $\f'$ is admissible, $\sigma({\f'}^{-1}) = - \sigma(\f') = -\sigma(\f)$ and $\f^{-1} \wedge {\f'}^{-1} = \id$. Therefore $\f \preceq \f \con {\f'}^{-1} = \f \star {\f'}^{-1}$ and $\sigma(\f \star {\f'}^{-1}) = 0$. It follows that $\f \star {\f'}^{-1} \in H_i$. Therefore the action $H_i \acting T_i$ is essential, which implies that $H_1$ and $H_2$ are non-trivial and the action of $H_1 \times H_2 \cong \la H_1 \times \{\id\}, \{\id\} \times H_2 \ra \leq G$ on $T_1 \times T_2$ is essential.
\end{proof}

\subsection{Embedding BMW groups into products of $\mathbb{R}$-trees}
\label{sec: embedding BMW groups}

The groups constructed in Section~\ref{sec: higher rank} contain no isometrically embedded irreducible BMW groups:

\begin{proposition} \label{prop: no BMW subgroups}
    Let $(G,\star)$ be the group from Definition~\ref{def: product group}. Let $H$ be a BMW group with BMW presentation $\la A \cup X | R \ra$ and suppose there exists a map $\psi: H \hookrightarrow G$ which is both a homomorphism and an isometric embedding, where $H$ is equipped with the word metric corresponding to $A \cup X$. Then $H$ is reducible.
\end{proposition}
\begin{proof}
    Let $d_H$ denote the word metric on $H$ with respect to $A \cup X$.
    
    We first show that we can assume without loss of generality that $N = \{1,2\}$, $G$ is isometric to the $\ell^1$ product $T_1 \times T_2$ and $\psi(\la A \ra) \subseteq T_1 \times \{\id\}$ and $\psi(\la X \ra) \subseteq \{\id\} \times T_2$.
    The fact that $\psi$ is an isometric embedding implies that, for all $a \in A \cup A^{-1}$ and $x \in X \cup X^{-1}$ we have $\ell(\psi(a)) = \ell(\psi(x)) = 1$ and $\ell(\psi(ax)) = 2$. Then 
    \[
        2 = \ell(\psi(a^{-1}x)) = \ell(\psi(a)) + \ell(\psi(x)) - 2\ell(\psi(a) \wedge \psi(x)) = 2 - 2\ell(\psi(a) \wedge \psi(x)),
    \]
    so $\psi(a) \wedge \psi(x) = \id$. Therefore $\psi(a), \psi(x)$ concatenate geodesically for all $a \in A \cup A^{-1}, x \in X \cup X^{-1}$. Moreover, if $a \in A \cup A^{-1}, x \in X \cup X^{-1}$ and $a' \in A \cup A^{-1}, x' \in X \cup X^{-1}$ are the unique elements such that $ax' = xa'$, then $\psi(a) \con \psi(x') = \psi(a) \star \psi(x') = \psi(x) \star \psi(a') = \psi(x) \con \psi(a')$. Therefore $\psi(a) \perp \psi(x)$. Both $\la A \ra$ and $\la X \ra$ are $\mathbb{Z}$-trees when equipped with $d_H$, so $\psi(\la A \ra)$ and $\psi(\la X \ra)$ are $\mathbb{Z}$-trees. Therefore there exists $n,m \in N$ with $n \neq m$ such that $\psi(\la A \ra) \subseteq \mathcal{Z}_n$ and $\psi(\la X \ra) \subseteq \mathcal{Z}_m$. The set $\mathcal{E}$ of elements $\f \in \mathcal{F}$ such that $\f_i = \id$ for all $i \in N$ such that $i \notin \{n,m\}$ is invariant under $\con$ and $-1$, contains $\id$ and, if $\f,\g \in \mathcal{F}$ such that $\f \in \mathcal{E}$ and $\g \preceq \f$, then $\g \in \mathcal{E}$. Thus $(E, \con, -1, \id, \preceq)$ is an ore and its extracted group contains $\psi(H)$.

    Thus, we assume from now on that $N = \{1,2\}$, $G$ is isometric to the $\ell^1$ product $T_1 \times T_2$ and $\psi(\la A \ra) \subseteq T_1 \times \{\id\}$ and $\psi(\la X \ra) \subseteq \{\id\} \times T_2$.

    For all $g \in \la A \ra$, $f \in \la X \ra$, let $\g_1 \in T_1, \f_2 \in T_2$ be such that $\psi(g) = (\g_1, \id), \psi(f) = (\id, \f_2)$. Define
    \begin{align*} 
        &V_A \coloneqq \cup \{y \in Y_1: a_1(t) = y \text{ for some } a \in A \cup A^{-1} \text{ and uncountably many } t \in [0,\ell(\aaa_1)]\} \\
        &V_X \coloneqq \cup \{y \in Y_2: x_2(t) = y \text{ for some } x \in X \cup X^{-1} \text{ and uncountably many } t \in [0,\ell(\mathfrak{x}_2)]\}.
    \end{align*}
       
    Let $n \coloneqq \max\{(2|A|)!, (2|X|)!\}$. 
    If $g \in \la A \ra$ and $x \in X \cup X^{-1}$, then there exists $f \in \la A \ra$ such that $g^n x = x f$. Therefore
    \[
        (\g_1^n, \tau(\psi(g^n)) \cdot \mathfrak{x}_2) = (\tau(\psi(x)) \cdot \f_1, \mathfrak{x}_2)
    \]
    Thus $\tau(\psi(g^n)) \cdot \mathfrak{x}_2 = \mathfrak{x}_2$ for all $g \in \la A \ra$ and $x \in X \cup X^{-1}$.
    By a symmetric argument, $\tau(\psi(g^n)) \cdot \aaa_1 = \aaa_1$ for all $g \in \la X \ra$ and $a \in A \cup A^{-1}$. 

    Consider the actions $\alpha_A: \la A \ra \rightarrow \Sym(Y_2)$ and $\alpha_X: \la X \ra \rightarrow \Sym(Y_1)$ defined by $\alpha_A(g)(y_2) = \tau(\psi(g)) \cdot y_2$ and $\alpha_X(f)(y_1) = \tau(\psi(f)) \cdot y_1$ for all $g \in \la A \ra, f \in \la X \ra, y_1 \in Y_1$ and $y_2 \in Y_2$. Let $W_A \subseteq Y_1$ be the $\la X \ra$-orbit of $V_A$ and let $W_X$ be the $\la A \ra$-orbit of $V_X$. Let $\beta_A$ be the restriction of $\alpha_A$ to $W_X$ and let $\beta_X$ be the restriction of $\alpha_A$ to $W_X$.
    Then, by the above argument, every $\la A \ra$-orbit in $W_X$ has cardinality dividing $n$ and every $\la X \ra$-orbit in $W_A$ has cardinality dividing $n$. Since $A$ and $X$ are finite, there are finitely many subgroups of $\la A \ra$ with index at most $n$ and finitely many subgroups of $\la X \ra$ with index at most $n$. Let $H_1 \leq \la A \ra$ be the intersection of all subgroups of $\la A \ra$ with index at most $n$ and let $H_2 \leq \la X \ra$ be the intersection of all subgroups of $\la X \ra$ with index at most $n$. Then $H_1 \leq \ker(\beta_A)$ and $H_2 \leq \ker(\beta_X)$ and $H_1$ has finite index in $\la A \ra$ and $H_2$ has finite index in $\la X \ra$. 
    
    Let $H' \coloneqq \la H_1, H_2 \ra \leq H$. Then $H'$ is a finite index subgroup of $H$. Let $h \in H_1$ and $f \in \la X \ra$. Let us show by induction on the word length of $f$ that $\beta_A(h) \circ f_2 \simeq f_2$. If $f \in X \cup X^{-1}$ then this follows from the fact that $h \in \ker(\beta_2)$, since $f_2(t) \in W_X$ for all but countably many $t \in [0,\ell(\f_2)]$. Let $m > 1$ and suppose the word length of $f$ is $m$. Then $f = f' x$ where $f' \in \la X \ra$ has word length $m-1$ and $x \in X \cup X^{-1}$. By the induction hypothesis, $\psi(hf') = (\h_1, \tau(\psi(h)) \cdot \f_2') = (\h_1, \f_2') = \psi(f') \star (-\tau(\psi(f')) \cdot \h_1, \id)$. Moreover, for all $w \in W_X$, we have 
    \[
        \tau((-\tau(\psi(f')) \cdot \h_1, \id)) \cdot w
        = \tau(\psi({f'}^{-1}hf')) \cdot w 
        = w.
    \]
    Therefore:
    \begin{align*}
        (\h_1, \tau(\h_1) \cdot \f_2) 
        &= \psi(hf) \\
        &= \psi(hf'x) \\
        &= \psi(f') \star (-\tau(\psi(f')) \cdot \h_1, \tau((-\tau(\psi(f')) \cdot \h_1, \id)) \cdot \mathfrak{x}_2) \\ 
        &= \psi(f') \star (-\tau(\psi(f')) \cdot \h_1, \mathfrak{x}_2) \\
        &= (\h_1, \f_2),
    \end{align*}
    so $\beta_A(h) \circ f_2 \simeq f_2$. A symmetric argument shows that $\beta_X(h) \circ g_1 \simeq g_1$ for all $h \in H_2$ and $g \in \la A \ra$. It follows that $H_1$ and $H_2$ commute, so $H \cong H_1 \times H_2$.
\end{proof}

I do not know whether it is possible to isometrically embed an irreducible BMW group into a group acting freely and transitively on a product of two (complete) real trees (see Question~\ref{Q: BMW embedding}). If one requests only that $G$ acts on the product with dense orbits, the following theorem shows that this can be done for BMW groups equipped with a positive BMW presentation. Note that the group defined by Wise in \cite[Example~4.1]{Wise-CSC} has a positive BMW presentation and is irreducible by Corollary~6.8 in loc. sit.

\begin{theorem} \label{thm: BMW metric groups}
    Let $H$ be a BMW group with a positive BMW presentation $\la A \cup X \; | \; R \ra$ and let $\Cay(H, A \cup X)$ be the corresponding Cayley graph.
    There exists a group $G$ such that the following hold:
    \begin{itemize} 
        \item[i.] there is an injective homomorphism $H \hookrightarrow G$;
        \item[ii.] $G$ acts freely with dense orbits on the $\ell^1$ product of two $\mathbb{R}$-trees $T_1 \times T_2$;
        \item[iii.] there is an isometric embedding $\psi: \Cay(H,A \cup X) \hookrightarrow T_1 \times T_2$, which is equivariant relative to $H \hookrightarrow G$.
    \end{itemize}
    If $H$ is irreducible, then for any subgroup $L \leq G$ which splits non-trivially as a direct product, the induced action of $L$ on $T_1 \times T_2$ does not have dense orbits.
\end{theorem}

\begin{proof}
Let $\mathcal{R}$ be closure of $R$ under cyclic permutations.
Let $\Gamma$ be the Cayley complex of $H$ with respect to the presentation $\la A \sqcup X \; | \; R \ra$ and recall that $\Gamma = \Gamma_A \times \Gamma_X$ is the product of the Cayley graph $\Gamma_A$ of $\la A \ra$ with the Cayley graph $\Gamma_X$ of $\la X \ra$. Let $\pi_A: \Gamma \rightarrow \Gamma_A$ and $\pi_X: \Gamma \rightarrow \Gamma_X$ be the projection maps and let $\id_A, \id_X$ be the identity in $\la A \ra, \la X \ra$ respectively. The action of $H$ on $\Gamma$ does not permute the factors so there are projection actions $H \acting \Gamma_A$ and $H \acting \Gamma_X$ defined (on the vertex sets) by 
\[ h \cdot g_A = \pi_A(h(g_A,\id_X)), \quad h \cdot g_X = \pi_X(h(\id_A,g_X)) \quad \forall \; h \in H, g_A \in \la A \ra, g_X \in \la X \ra.\]

\begin{definition}
    The induced action of $\la A \ra$ on $\Gamma_X$ permutes the set $X$ and can be read off the presentation of $H$: for all $a \in A$ and $x \in X$, there exist a unique $a' \in A$ and a unique $x' \in X$ such that $ax{a'}^{-1}{x'}^{-1} \in \mathcal{R}$, and we then have $a \cdot x = x'$. Similarly the action of $\la X \ra$ on $\Gamma_A$ permutes the set $A$ and can be read off the presentation: given $a \in A, x \in X$ there exist unique $a' \in A, x' \in X$ such that $a'x'a^{-1}x^{-1} \in \mathcal{R}$ and $x \cdot a = a'$.

    For each $a \in A$ and $x \in X$ let $\sigma_a \in \Sym(X)$ and $\sigma_x \in \Sym(A)$ denote the resulting permutations. 
\end{definition}

\begin{definition} 
\begin{itemize}
    \item Let $Y_A$ be a graph with vertex set $A$ such that, for each $a_1, a_2 \in A$ with $a_1 \neq a_2$, there is a directed edge $e(a_1, a_2)$ from $a_1$ to $a_2$ of length $1$. Similarly, let $Y_X$ be a graph with vertex set $X$ such that, for each $x_1, x_2 \in X$ with $x_1 \neq x_2$, there is a directed edge $e(x_1, x_2)$ from $x_1$ to $x_2$ of length $1$. 
    Let $d$ denote the resulting metrics on $Y_A$ and $Y_X$.
    \item For each $n \in \mathbb{N}$, let $Y_A(n) \subseteq Y_A, Y_X(n) \subseteq Y_X$ be the subsets consisting of points whose distance to a vertex is a multiple of $1/n$. 
    For all distinct $a_1, a_2 \in A$, $x_1, x_2 \in X$ let $e_n(a_1, a_2) \coloneqq Y_A(n) \cap e(a_1, a_2)$ and $e_n(x_1, x_2) \coloneqq Y_X(n) \cap e(x_1, x_2)$.
\end{itemize}
\end{definition}

Let $n \in \mathbb{N}$. We will construct a BMW group $G_n$ with generating set $Y_A(n) \sqcup Y_X(n)$.
We first define actions of the free groups $\la Y_A(n) \ra$ and $\la Y_X(n) \ra$ on $Y_X(n)$ and $Y_A(n)$ respectively. 

Recall that a \textit{directed cycle} in a directed graph is a cycle with edges $(e_0, \dots, e_{k-1})$ such that, for each $i \in \mathbb{Z}/k\mathbb{Z}$, the edge $e_i$ is directed from $e_{i-1} \cap e_i$ to $e_i \cap e_{i+1}$. Observe that, for any sequence of vertices $C = (v_0, \dots, v_{k-1})$ in $Y_A$ (resp. in $Y_X$) with $k \geq 2$, the cycle $(e(v_0,v_1), e(v_1,v_2), \dots, e(v_{k-1},v_{k}))$ is the unique directed cycle in $Y_A$ (resp. $Y_X$) with vertices $C$. In both $Y_A(n)$ and $Y_X(n)$, define a sequence $C(n)$ as follows. For each $i \in \{1, \dots, k-1\}$ and $j \in \{0, \dots, n-1\}$, let $y_{i,j} \in e(v_i, v_{i+1})$ be the point such that $d(v_i, y_{i,j}) = j/n$. Let $C(n) \coloneqq (y_{i,j} : i \in \{0, \dots, k-1\}, j \in \{0, \dots, n-1\})$, ordered lexicographically. Less formally, $C(n)$ is the sequence of points in $Y_A(n)$ (resp. $Y_X(n)$) that one encounters if one starts at $v_0$ and follows the directed cycle in $Y_A$ with vertices $C$.

Given $y_A \in Y_A(n)$, there is a unique $a \in A$ such that $y_A \in e(b,a) - \{b\}$ for some $b \in A$. Define $\varphi_n(y_A) \in \Sym(Y_X(n))$ as follows. Let $C_1 \dots C_k$ be the cyclic decomposition of $\sigma_a$; so, for each $i$, $C_i = (x_0, \dots, x_m)$ for some $m \geq 0$ and $x_0, \dots, x_{m-1} \in X$ and $\sigma_a(x_j) = x_{j+1}$ for each $j \in \mathbb{Z} / (m+1) \mathbb{Z}$. Let $\varphi_n(y_A)$ be the symmetry of $Y_A(n)$ with cyclic decomposition $C_1(n) \dots C_k(n)$. This defines a homomorphism $\varphi_n: \la Y_A(n) \ra \rightarrow \Sym(Y_X(n))$. Note that, if $Y_X(n)$ is equipped with the metric induced by $Y_X$, the action $\varphi_n$ is \textit{not} by isometries.

Similarly, if $y_X \in Y_X(n)$, let $x \in X$ be the unique vertex of $Y_X$ such that $y_X \in e(z,x) - \{z\}$ for some $z \in X$ and let $C_1 \dots C_k$ be the cyclic decomposition of $\sigma_x$. Define $\varphi_n(y_X) \in \Sym(Y_A(n))$ to be the symmetry of $Y_A(n)$ with cyclic decomposition $C_1(n) \dots C_k(n)$. This determines a homomorphism $\varphi_n: \la Y_X(n) \ra \rightarrow \Sym(Y_A(n))$. 

\begin{definition}
For each $y_A \in Y_A(n), y_X \in Y_X(n)$, let 
\[
    r_n(y_A, y_X) \coloneqq y_A \; \varphi_n(y_A)^{-1}(y_X) \; (\varphi_n(y_X)^{-1}(y_A))^{-1} \; y_X^{-1} \in \la Y_A(n) \cup Y_X(n) \ra
\]
and let $R(n) \coloneqq \{r_n(y_A,y_X) : y_A \in Y_A(n), \; y_X \in Y_X(n)\}$. Let $G_n$ be the group with presentation $\la Y_A(n) \cup Y_X(n) \; | \; R(n) \ra$ and let $\Gamma_n$ be the Cayley complex of $G_n$ with respect to this presentation. The image of a relation $r(y_A,y_X)$ in $\Gamma_n$ is illustrated in Figure~\ref{fig: BMW relation}.
\end{definition}

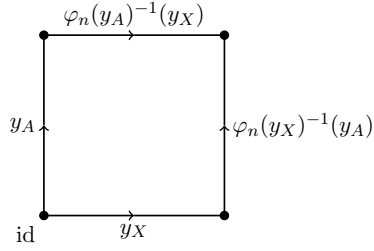
\begin{figure}
    \centering
    \scalebox{0.8}[0.8]{
    \begin{tikzpicture}
        \draw[thick,->] (0,0) -- (0,1.5);
        \draw[thick] (0,1.5) -- (0,3);
        \draw[thick,->] (0,0) -- (1.5,0);
        \draw[thick] (1.5,0) -- (3,0);
        \draw[thick,->] (3,0) -- (3,1.5);
        \draw[thick] (3,1.5) -- (3,3);
        \draw[thick,->] (0,3) -- (1.5,3);
        \draw[thick] (1.5,3) -- (3,3);
        \node[anchor=north] at (1.5,0) {$y_X$};
        \node[anchor=west] at (3,1.5) {$\varphi_n(y_X)^{-1}(y_A)$};
        \node[anchor=south] at (1.5,3) {$\varphi_n(y_A)^{-1}(y_X)$};
        \node[anchor=east] at (0,1.5) {$y_A$};
        \filldraw[black] (0,0) circle (2pt);
        \filldraw[black] (0,3) circle (2pt);
        \filldraw[black] (3,0) circle (2pt);
        \filldraw[black] (3,3) circle (2pt);
        \node at (-0.3,-0.3) {$\id$};
    \end{tikzpicture}
    }
    \caption{The relation $r_n(y_A,y_X)$ in $\Gamma_n$, where $y_A \in Y_A(n)$ and $y_X \in Y_X(n)$.}
    \label{fig: BMW relation}
\end{figure}

\setcounter{claim}{0}
\begin{claim}
    For each $n \in \mathbb{N}$, $\la Y_A(n) \cup Y_X(n) \; | \; R(n) \ra$ is a BMW presentation.
\end{claim}
\begin{proof}
\renewcommand{\qedsymbol}{$\blacksquare$}
    Let $z_A \in Y_A(n) \cup Y_A(n)^{-1}, z_X \in Y_X(n) \cup Y_X(n)^{-1}$ and let $y_A \in Y_A(n)$ and $y_X \in Y_X(n)$ be such that either $z_A = y_A$ or $z_A = y_A^{-1}$, and $z_X = y_X$ or $z_X = y_X^{-1}$. Then $r_n(y_A,y_X)$ is the unique element of $R(n)$ of the form $z_A z_X z_A' z_X'$ or $z_A'z_X'z_Az_X$ or $z_A^{-1}z_X'z_A'z_X^{-1}$ or $z_A'z_X^{-1}z_Az_X'$ for some $z_A' \in Y_A \cup Y_A^{-1}$ and $z_X' \in Y_X \cup Y_X^{-1}$.
\end{proof}

Let $d_n$ be the path metric on the 1-skeleton $\Gamma_n^{(1)}$, where each edge is assigned a length of $1/n$.

Equip $\mathbb{N}$ with the partial order $n \preceq m$ if and only if $n$ divides $m$, and note that $(\mathbb{N}, \preceq)$ is a net. If $m = kn$, where $k,m,n \in \mathbb{N}$, define $\psi_{n,m}: G_n \rightarrow G_m$ as follows. If $y \in Y_A(n) \cup Y_X(n)$, set $\psi_{n,m}(y) \coloneqq y^k$. If $y_A \in Y_A(n), y_X \in Y_X(n)$, then 
\begin{align*} 
    \psi_{n,m}(r_n(y_A,y_X)) &= y_A^k \; \varphi_n(y_A)^{-1}(y_X)^k \; \varphi_n(y_X)^{-1}(y_A)^{-k} \; y_X^{-k} \\
    &= y_A^k \; \varphi_m(y_A)^{-k}(y_X)^k \; \varphi_m(y_X)^{-k}(y_A)^{-k} \; y_X^{-k}.
\end{align*}
Let $a \in A$ and $x \in X$ be such that $y_A \in e(a',a) - \{a'\}$ and $y_X \in e(x',x) - \{x'\}$ for some $a' \in A, x' \in X$.
Note that, since $y_A \in Y_A(n)$, for any $i \in \{0, \dots, k-1\}$ and for all $z_X \in Y_X(m)$ we have $\varphi_m(z_X)^i(y_A) \in e(a',a) - \{a'\}$ for some $a' \in A$. Therefore $\varphi_j(\varphi_m(z_X)^i(y_A)) = \varphi_j(y_A)$ for any $j \in \mathbb{N}$.
Similarly, $\varphi_j(\varphi_m(z_A)^i(y_X)) = \varphi_j(y_x)$ for any $z_A \in Y_A(m), i \in \{0, \dots, k-1\}$ and $j \in \mathbb{N}$. It follows that $\psi_{n,m}(r(y_A,y_X)) = \id$ (see Figure~\ref{fig: homomorphism}). Therefore $\psi_{n,m}$ extends a homomorphism $\psi_{n,m}: G_n \rightarrow G_m$. 

By construction, $\psi_{n,m}$ is an isometric embedding when restricted to balls of radius 1 in $G_n$. It extends naturally to a local isometry $\overline{\psi}_{n,m} : \Gamma_n \rightarrow \Gamma_m$, where the image $\overline{\psi}_{n,m}(\Gamma_n)$ is the convex hull of $\psi_{n,m}(G_n)$. Since both $\Gamma_n$ and $\overline{\psi}_{n,m}(\Gamma_n)$ are simply connected, this implies that $\psi_{n,m}$ is an isometric embedding. 

\begin{figure}
    \centering
    \scalebox{1.5}[1.5]{
    \begin{tikzpicture}
        \draw[thick,color=MidnightBlue!90,->] (0,0) -- (0,0.5) node[anchor=east] {\tiny $y_A$};
        \draw[thick,color=MidnightBlue!90,->] (0,0.4) -- (0,1.5) node[anchor=east] {\tiny $y_A$};
        \draw[thick,color=MidnightBlue!90] (0,1.4) -- (0,2.32);
        \draw[thick,color=MidnightBlue!90,dotted] (0,2.39) -- (0,2.65);
        \draw[thick,color=MidnightBlue!90,->] (0,2.68) -- (0,3.5) node[anchor=east] {\tiny $y_A$};
        \draw[thick,color=MidnightBlue!90] (0,3.4) -- (0,4);
        
        \draw[color=MidnightBlue!60,->] (1,0) -- (1,0.5);
        \node[color=MidnightBlue!60,anchor=east] at (1,0.4) {\tiny $y_{A,1}$};
        \draw[color=MidnightBlue!60,->] (1,0.4) -- (1,1.5);
        \node[color=MidnightBlue!60,anchor=east] at (1,1.4) {\tiny $y_{A,1}$};
        \draw[color=MidnightBlue!60] (1,1.4) -- (1,2.33);
        \draw[color=MidnightBlue!60,dotted] (1,2.4) -- (1,2.6);
        \draw[color=MidnightBlue!60] (1,2.67) -- (1,3);
        \draw[color=MidnightBlue!60,->] (1,3) -- (1,3.5);
        \node[color=MidnightBlue!60,anchor=east] at (1,3.4) {\tiny $y_{A,1}$};
        \draw[color=MidnightBlue!60] (1,3.4) -- (1,4);

        \draw[color=MidnightBlue!60,->] (2,0) -- (2,0.5);
        \node[color=MidnightBlue!60,anchor=east] at (2,0.4) {\tiny $y_{A,2}$};
        \draw[color=MidnightBlue!60,->] (2,0.4) -- (2,1.5);
        \node[color=MidnightBlue!60,anchor=east] at (2,1.4) {\tiny $y_{A,2}$};
        \draw[color=MidnightBlue!60] (2,1.4) -- (2,2.33);
        \draw[color=MidnightBlue!60,dotted] (2,2.4) -- (2,2.6);
        \draw[color=MidnightBlue!60] (2,2.67) -- (2,3);
        \draw[color=MidnightBlue!60,->] (2,3) -- (2,3.5);
        \node[color=MidnightBlue!60,anchor=east] at (2,3.4) {\tiny $y_{A,2}$};
        \draw[color=MidnightBlue!60] (2,3.4) -- (2,4);

        \draw[color=MidnightBlue!60,->] (3,0) -- (3,0.5);
        \node[color=MidnightBlue!60,anchor=east] at (3.1,0.4) {\tiny $y_{A,k-1}$};
        \draw[color=MidnightBlue!60,->] (3,0.4) -- (3,1.5);
        \node[color=MidnightBlue!60,anchor=east] at (3.1,1.4) {\tiny $y_{A,k-1}$};
        \draw[color=MidnightBlue!60] (3,1.4) -- (3,2.33);
        \draw[color=MidnightBlue!60,dotted] (3,2.4) -- (3,2.6);
        \draw[color=MidnightBlue!60] (3,2.67) -- (3,3);
        \draw[color=MidnightBlue!60,->] (3,3) -- (3,3.5);
        \node[color=MidnightBlue!60,anchor=east] at (3.1,3.4) {\tiny $y_{A,k-1}$};
        \draw[color=MidnightBlue!60] (3,3.4) -- (3,4);

        \draw[thick,color=MidnightBlue!90,->] (4,0) -- (4,0.5) node[anchor=west] {\tiny $y_{A,k}$};
        \draw[thick,color=MidnightBlue!90,->] (4,0.4) -- (4,1.5) node[anchor=west] {\tiny $y_{A,k}$};
        \draw[thick,color=MidnightBlue!90] (4,1.4) -- (4,2.32);
        \draw[thick,color=MidnightBlue!90,dotted] (4,2.39) -- (4,2.65);
        \draw[thick,color=MidnightBlue!90] (4,2.68) -- (4,3);
        \draw[thick,color=MidnightBlue!90,->] (4,3) -- (4,3.5) node[anchor=west] {\tiny $y_{A,k}$};
        \draw[thick,color=MidnightBlue!90] (4,3.4) -- (4,4);

        \draw[thick,Thistle,->] (0,0) -- (0.5,0) node[anchor=north] {\tiny $y_X$};
        \draw[thick,Thistle,->] (0.4,0) -- (1.5,0) node[anchor=north] {\tiny $y_X$};
        \draw[thick,Thistle] (1.4,0) -- (2.32,0);
        \draw[thick,Thistle,dotted] (2.39,0) -- (2.65,0);
        \draw[thick,Thistle] (2.68,0) -- (3,0);
        \draw[thick,Thistle,->] (3,0) -- (3.5,0) node[anchor=north] {\tiny $y_X$};
        \draw[thick,Thistle] (3.4,0) -- (4,0);

        \draw[color=Thistle!70,->] (0,1) -- (0.5,1) node[anchor=north] {\tiny $y_{X,1}$};
        \draw[color=Thistle!70,->] (0.4,1) -- (1.5,1) node[anchor=north] {\tiny $y_{X,1}$};
        \draw[color=Thistle!70] (1.4,1) -- (2.33,1);
        \draw[color=Thistle!70,dotted] (2.4,1) -- (2.6,1);
        \draw[color=Thistle!70] (2.67,1) -- (3,1);
        \draw[color=Thistle!70,->] (3,1) -- (3.5,1) node[anchor=north] {\tiny $y_{X,1}$};
        \draw[color=Thistle!70] (3.4,1) -- (4,1);

        \draw[color=Thistle!70,->] (0,2) -- (0.5,2) node[anchor=north] {\tiny $y_{X,2}$};
        \draw[color=Thistle!70,->] (0.4,2) -- (1.5,2) node[anchor=north] {\tiny $y_{X,2}$};
        \draw[color=Thistle!70] (1.4,2) -- (2.33,2);
        \draw[color=Thistle!70,dotted] (2.4,2) -- (2.6,2);
        \draw[color=Thistle!70] (2.67,2) -- (3,2);
        \draw[color=Thistle!70,->] (3,2) -- (3.5,2) node[anchor=north] {\tiny $y_{X,2}$};
        \draw[color=Thistle!70] (3.4,2) -- (4,2);

        \draw[color=Thistle!70,->] (0,3) -- (0.5,3) node[anchor=north] {\tiny $y_{X,k-1}$};
        \draw[color=Thistle!70,->] (0.4,3) -- (1.5,3) node[anchor=north] {\tiny $y_{X,k-1}$};
        \draw[color=Thistle!70] (1.4,3) -- (2.33,3);
        \draw[color=Thistle!70,dotted] (2.4,3) -- (2.6,3);
        \draw[color=Thistle!70] (2.67,3) -- (3,3);
        \draw[color=Thistle!70,->] (3,3) -- (3.5,3) node[anchor=north] {\tiny $y_{X,k-1}$};
        \draw[color=Thistle!70] (3.4,3) -- (4,3);

        \draw[thick,Thistle,->] (0,4) -- (0.5,4) node[anchor=south] {\tiny $y_{X,k}$};
        \draw[thick,Thistle,->] (0.4,4) -- (1.5,4) node[anchor=south] {\tiny $y_{X,k}$};
        \draw[thick,Thistle] (1.4,4) -- (2.32,4);
        \draw[thick,Thistle,dotted] (2.39,4) -- (2.65,4);
        \draw[thick,Thistle] (2.68,4) -- (3,4);
        \draw[thick,Thistle,->] (3,4) -- (3.5,4) node[anchor=south] {\tiny $y_{X,k}$};
        \draw[thick,Thistle] (3.4,4) -- (4,4);

        \filldraw[black] (0,0) circle (1pt);
        \filldraw[color=gray!120] (0,1) circle (1pt);
        \filldraw[color=gray!120] (0,2) circle (1pt);
        \filldraw[color=gray!120] (0,3) circle (1pt);
        \filldraw[black] (0,4) circle (1pt);
        \filldraw[color=gray!120] (1,0) circle (1pt);
        \filldraw[color=gray!120] (1,1) circle (1pt);
        \filldraw[color=gray!120] (1,2) circle (1pt);
        \filldraw[color=gray!120] (1,3) circle (1pt);
        \filldraw[color=gray!120] (1,4) circle (1pt);
        \filldraw[color=gray!120] (2,0) circle (1pt);
        \filldraw[color=gray!120] (3,0) circle (1pt);
        \filldraw[black] (4,0) circle (1pt);
        \filldraw[color=gray!120] (2,1) circle (1pt);
        \filldraw[color=gray!120] (3,1) circle (1pt);
        \filldraw[color=gray!120] (4,1) circle (1pt);
        \filldraw[color=gray!120] (2,2) circle (1pt);
        \filldraw[color=gray!120] (3,2) circle (1pt);
        \filldraw[color=gray!120] (4,2) circle (1pt);
        \filldraw[color=gray!120] (2,3) circle (1pt);
        \filldraw[color=gray!120] (3,3) circle (1pt);
        \filldraw[color=gray!120] (4,3) circle (1pt);
        \filldraw[color=gray!120] (3,4) circle (1pt);
        \filldraw[black] (4,4) circle (1pt);
        \filldraw[color=gray!120] (2,4) circle (1pt);
    \end{tikzpicture}
    }
    \caption{The $\psi_{n,m}$-image of $r_n(y_A,y_X)$, where $y_A \in Y_A(n)$ and $y_X \in Y_X(n)$. For each $i \in \{1, \dots, k\}$, we denote $y_{A,i} \coloneqq \varphi_m(y_X)^{-i}(y_A)$ and $y_{X,i} \coloneqq \varphi_m(y_A)^{-i}(y_X)$.}
    \label{fig: homomorphism}
\end{figure}
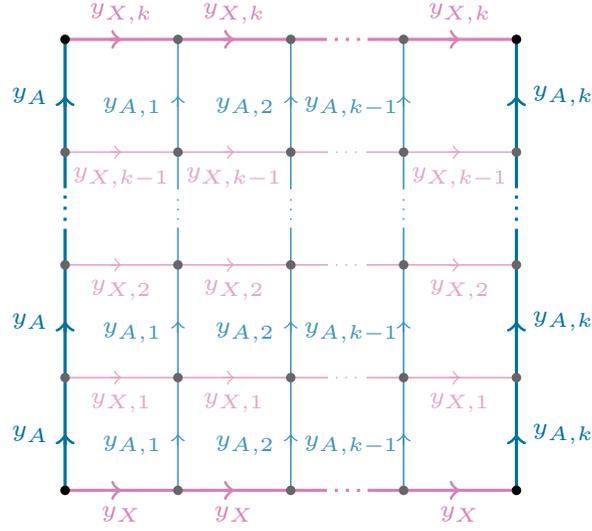

\begin{definition}
Let $G$ be the direct union of the system $\la G_n, \psi_{n,m} \ra$.

Since each $\psi_{n,m}$ is an isometric embedding, there is a metric $d:G^2 \rightarrow G$ such that the restriction of $d$ to $G_n^2$ is $d_n$ for any $n \in \mathbb{N}$. Let $X$ be the completion of $(G,d)$. 

For each $n \in \mathbb{N}$, let $\psi_n: \la Y_A(n) \cup Y_X(n) \ra \rightarrow G$ be the canonical embedding. It is a homomorphism and an isometric embedding.
\end{definition}

\begin{claim}
    $X$ is the $\ell^1$ product of two $\mathbb{R}$ trees $T_A \times T_X$. 
\end{claim}
\begin{proof}
\renewcommand{\qedsymbol}{$\blacksquare$}
    Let $W_A \coloneqq \cup_{n \in \mathbb{N}} Y_A(n)$ and $W_X \coloneqq \cup_{n \in \mathbb{N}} Y_X(n)$. Let $S_A \coloneqq \la W_A \ra \leq G$ and $S_X \coloneqq \la W_X \ra \leq G$. We will show that $S_A$ and $S_X$ are $\mathbb{Q}$-trees.

    Let $f,g,h \in S_A$. Then there exists $n \in \mathbb{N}$ such that $f,g,h \in \la Y_A(n)\ra$. Note that $d(f,g) = d_n(f,g) \in \mathbb{Z}/n \leq \mathbb{Q}$ so $S_A$ is indeed a $\mathbb{Q}$-metric space.
    Since $(\la Y_A(n) \ra, d_n)$ is a $\mathbb{Z}/n$-tree, there exists a median $p \in \la Y_A(n) \ra$ for $d_n$ and therefore for $d$. For all $m \succeq n$, the restriction of $\psi_{n,m}$ to $\la Y_A(n) \ra$ is an isometric embedding into $(\la Y_A(m), d_m \ra$. It follows that $p$ is the unique median of $\{f,g,h\}$ in $S_A$, so $S_A$ is median, and that $S_A$ has rank 1. Let $n \in \mathbb{N}$ and let $y_A \in Y_A(n)$. For each $p/q \in \mathbb{Q} \cap (0,1/n]$ such that $p$ and $q$ are coprime,  $\gamma(t) \coloneqq \psi_q(y_A^p)$. Let $\gamma(0) \coloneqq \id$. Then $\gamma$ is a $\mathbb{Q}$-geodesic in $G$ from $\id$ to $\psi_n(y_A)$. It follows that $G$ is a geodesic $\mathbb{Q}$-metric space. By Lemma~\ref{lem: tree characterisation}, this implies that $S_A$ is $\mathbb{Q}$-tree. A symmetric argument shows that $S_X$ is a $\mathbb{Q}$-tree. 

    For each $n \in \mathbb{N}$, Proposition~\ref{prop: BMW basics} implies that $(\la Y_A(n) \cup Y_X(n) \ra, d_n)$ is, as a $\mathbb{Z}/n$-metric space, the $\ell^1$ product of $\la Y_A(n) \ra$ and $\la Y_X(n) \ra$. It follows that $G$ is the $\ell^1$-product of $S_A$ and $S_X$ as a $\mathbb{Q}$-metric space. Let $T_A$ be the completion of $S_A$ and let $T_X$ be the completion of $S_X$. Then $X$ is the $\ell^1$ product $T_A \times T_X$.
\end{proof}

Since each element of $G$ has positive translation length in $G$, the action of $G$ on its completion $X$ is free.
This proves Item ii of the theorem. The injection $\psi_1$ verifies Item i, and naturally extends to an $H$-equivariant isometric embedding $\psi: \Cay(H,A \cup X) \hookrightarrow T_A \times T_X$, verifying Item iii.
The last part of the theorem is proved in the following claim:

\begin{claim} \label{claim: reducible embedding} 
    If $G$ contains a subgroup $K \leq G$ which splits non-trivially as a direct product such that the induced action $K \acting T_A \times T_X$ has dense orbits, then $H$ is reducible.
\end{claim}
\begin{proof}
\renewcommand{\qedsymbol}{$\blacksquare$}
    Note that $G \leq \Isom(T_A) \times \Isom(T_X)$. Let $p_A: G \rightarrow \Isom(T_A), p_X: G \rightarrow \Isom(T_X)$ be the canonical projections. Also recall that $(G,d)$ is isomorphic to the $\ell^1$ product $S_A \times S_X$. Let $\rho_A : G \rightarrow S_A, \rho_X: G \rightarrow S_X$ be the canonical projections.
    
    By Lemma~\ref{lem: well-behaved factors}, $K \cong K_A \times K_X$, where $K_A \leq \ker(p_X) \leq S_A, K_X \leq \ker(p_A) \leq S_X$, and $K_A$ is dense in $T_A$, and $K_X$ is dense in $T_X$. 

    Let $\Gamma_A, \Gamma_X$ be the Cayley graphs of the free groups $\la A \ra, \la X \ra$ with respect to the generating sets $A,X$. 
    Let $q_A: H \rightarrow \Isom(\Gamma_A)$ and $q_X: H \rightarrow \Isom(\Gamma_X)$ be the canonical projections. We will show that $q_A(H)$ and $q_X(H)$ are discrete with respect to the compact open topology on $\Isom(\Gamma_A)$ and $\Isom(\Gamma_X)$. By \cite[Proposition~1.2]{Burger-Moses}, this implies that $H$ is reducible.

    Let $h \in H$ be such that $q_X(h)$ fixes the ball of radius $1$ in $\Gamma_X$ around $\id_X$ pointwise. In particular, $q_X(h)$ fixes $\id_X$, so $h \in \la A \ra$. Moreover $h \cdot x = x$ for all $x \in X$. Then, for all $n \in \mathbb{N}$, the element $\varphi_n(\psi_{1,n}(h)) \in \Sym(Y_X(n))$ is the identity map. Let $f \coloneqq \psi_1(h)$. Then $p_X(f)(g) = g$ for all $g \in S_X$ with $d(g,\id_X) \leq 1$.
    
    Let $n \in \mathbb{N}$ and suppose that $p_X(f)(g) = g$ for all $g \in S_X$ such that $d(g,\id_X) < n$. Let $g \in S_X$ be such that $d(g,\id_X) < n+1$. There exists $k_X \in K_X$ such that $d(k_X, \id_X) < n$ and $d(k_X, g_X) < 1$. Then $p_X(f)(k_X) = k_X$, which, since $p_A(k_X)(f) = f$, implies that $f$ and $k_X$ commute. Let $g' \coloneqq k_X^{-1} g$, so $d(\id_x,g) < 1$. Then 
    \[
        p_X(f)(g)
        = \rho_X(f k_X g') 
        = \rho_X(k_X f g') 
        = k_X p_X(f)(g') 
        = k_X g'
        = g.
    \]
    It follows by induction on $n$ that $p_X(f)$ is the trivial map on $S_X$. Since $\psi$ is $H$-equivariant, this implies in particular that $q_X(h)$ is the identity map on $\Gamma_X$. Therefore the identity is isolated in $q_X(H)$, so $q_X(H)$ is discrete. A symmetric argument shows that $q_A(H)$ is discrete.
\end{proof}

\end{proof}

\bibliographystyle{alpha}
\bibliography{Biblio}

\begin{thebibliography}{GKMS15}

\bibitem[Ber89]{Berestovskii1989}
V.~N. Berestovski\u{\i}.
\newblock Quasicones of {L}obachevskii spaces at infinity.
\newblock In {\em Proceedings of the International Conference on Algebra Dedicated to the Memory of A. I. Malcev}, page 112. Akad. Nauk SSSR Sibirsk. Otdel., Inst. Mat., Novosibirsk, 1989.

\bibitem[Ber19]{Berestovskii2019}
V.~N. Berestovski\u{\i}.
\newblock On the {U}ryson {$\mathbb{R}$}-tree.
\newblock {\em Sibirsk. Mat. Zh.}, 60(1):14--27, 2019.

\bibitem[BF95]{Bestvina-Feighn}
Mladen Bestvina and Mark Feighn.
\newblock Stable actions of groups on real trees.
\newblock {\em Invent. Math.}, 121(2):287--321, 1995.

\bibitem[BM00]{Burger-Moses}
Marc Burger and Shahar Mozes.
\newblock Lattices in product of trees.
\newblock {\em Publications Math\'ematiques de l'IH\'ES}, 92:151--194, 2000.

\bibitem[Bow13]{Bowditch-coarse-median}
Brian~H. Bowditch.
\newblock Coarse median spaces and groups.
\newblock {\em Pacific J. Math.}, 261(1):53--93, 2013.

\bibitem[Bow18]{Bowditch-2018}
Brian~H. Bowditch.
\newblock Large-scale rigidity properties of the mapping class groups.
\newblock {\em Pacific J. Math.}, 293(1):1--73, 2018.

\bibitem[Bow24]{Bowditch}
Brian Bowditch.
\newblock Median algebras, 2024.
\newblock Preprint, \url{https://bhbowditch.com/papers/median-algebras.pdf}.

\bibitem[BP10]{Berestovskii-Plaut}
V.~N. Berestovski\u{\i} and C.~P. Plaut.
\newblock Covering {$\mathbb{R}$}-trees, {$\mathbb{R}$}-free groups, and dendrites.
\newblock {\em Adv. Math.}, 224(5):1765--1783, 2010.

\bibitem[Cap19]{Caprace-finite&infinite}
Pierre-Emmanuel Caprace.
\newblock Finite and infinite quotients of discrete and indiscrete groups.
\newblock In {\em Groups {S}t {A}ndrews 2017 in {B}irmingham}, volume 455 of {\em London Math. Soc. Lecture Note Ser.}, pages 16--69. Cambridge Univ. Press, Cambridge, 2019.

\bibitem[Chi01]{Chiswell-intro_to_lambda_trees}
Ian Chiswell.
\newblock {\em Introduction to {$\Lambda$}-trees}.
\newblock World Scientific Publishing Co., Inc., River Edge, NJ, 2001.

\bibitem[CM10]{Chiswell-Muller2010}
Ian Chiswell and Thomas M\"uller.
\newblock Embedding theorems for tree-free groups.
\newblock {\em Math. Proc. Cambridge Philos. Soc.}, 149(1):127--146, 2010.

\bibitem[CM12]{Chiswell-Muller}
Ian Chiswell and Thomas Müller.
\newblock {\em A Universal Construction for Groups Acting Freely on Real Trees}.
\newblock Cambridge Tracts in Mathematics. Cambridge University Press, 2012.

\bibitem[CRHK24]{CRHK}
Montserrat Casals-Ruiz, Mark Hagen, and Ilya Kazachkov.
\newblock Real cubings and asymptotic cones of hierarchically hyperbolic groups, 2024.
\newblock Preprint, \url{https://www.wescac.net/cones.html}.

\bibitem[DP01]{Dyubina-Polterovich}
Anna Dyubina and Iosif Polterovich.
\newblock Explicit constructions of universal {$\mathbb{R}$}-trees and asymptotic geometry of hyperbolic spaces.
\newblock {\em Bull. Lond. Math. Soc.}, 33(6):727--734, 2001.

\bibitem[Dun97]{Dunwoody97}
M.~J. Dunwoody.
\newblock Groups acting on protrees.
\newblock {\em J. Lond. Math. Soc.}, 56(1):125--136, 08 1997.

\bibitem[Fio24]{Fioravanti-convex-cores}
Elia Fioravanti.
\newblock Convex cores for actions on finite-rank median algebras.
\newblock {\em Ann. Inst. Fourier.}, 74(5):1895--1942, 2024.

\bibitem[GKMS15]{GKMS}
Andrei-Paul Grecianu, Alexei~V. Kvaschuk, Alexei~G. Myasnikov, and Denis~E. Serbin.
\newblock Groups acting on hyperbolic {$\Lambda$}-metric spaces.
\newblock {\em Internat. J. Algebra Comput.}, 25(6):977--1042, 2015.

\bibitem[GLP94]{GLP}
D.~Gaboriau, G.~Levitt, and F.~Paulin.
\newblock Pseudogroups of isometries of {${\bf R}$} and {R}ips' theorem on free actions on {${\bf R}$}-trees.
\newblock {\em Israel J. Math.}, 87(1-3):403--428, 1994.

\bibitem[Kec95]{Kechris}
Alexander~S. Kechris.
\newblock {\em Classical Descriptive Set Theory}, volume 156 of {\em Graduate Texts in Mathematics}.
\newblock Springer-Verlag New York, NY, 1995.

\bibitem[KMS14]{KMS2}
Olga Kharlampovich, Alexei Myasnikov, and Denis Serbin.
\newblock Infinite words and universal free actions.
\newblock {\em Groups Complex. Cryptol.}, 6(1):55--69, 2014.

\bibitem[Mes24]{Messaci}
Mohamed~Lamine Messaci.
\newblock Isometric actions on locally compact finite rank median spaces, 2024.
\newblock Preprint, arXiv:2309.10760.

\bibitem[MNO92]{Mayer-Nikiel-Oversteegen}
John~C. Mayer, Jacek Nikiel, and Lex~G. Oversteegen.
\newblock Universal spaces for {${\bf R}$}-trees.
\newblock {\em Trans. Amer. Math. Soc.}, 334(1):411--432, 1992.

\bibitem[MRS05]{MRS}
Alexei~G. Myasnikov, Vladimir~N. Remeslennikov, and Denis~E. Serbin.
\newblock Regular free length functions on {Lyndon}'s free {{\(\mathbb{Z}[t]\)}}-group {{\(F^{\mathbb{Z}[t]}\)}}.
\newblock In {\em Groups, languages, algorithms. Proceedings of the AMS-ASL joint special session on interactions between logic, group theory, and computer science, Baltimore, MD, USA, January 16--19, 2003.}, pages 37--77. Providence, RI: American Mathematical Society (AMS), 2005.

\bibitem[MS84]{Morgan-Shalen1984}
John~W. Morgan and Peter~B. Shalen.
\newblock Valuations, trees, and degenerations of hyperbolic structures. {I}.
\newblock {\em Ann. of Math. (2)}, 120(3):401--476, 1984.

\bibitem[MS91]{Morgan-Shalen1991}
John~W. Morgan and Peter~B. Shalen.
\newblock Free actions of surface groups on {${\bf R}$}-trees.
\newblock {\em Topology}, 30(2):143--154, 1991.

\bibitem[Nik89]{Nikiel}
Jacek Nikiel.
\newblock Topologies on pseudo-trees and applications.
\newblock {\em Mem. Amer. Math. Soc.}, 82(416):vi+116, 1989.

\bibitem[Rad20]{Radu}
Nicolas Radu.
\newblock New simple lattices in products of trees and their projections.
\newblock {\em Canad. J. Math.}, 72(6):1624--1690, 2020.
\newblock With an appendix by Pierre-Emmanuel Caprace.

\bibitem[Rat04]{Rattaggi}
Diego~Attilio Rattaggi.
\newblock {\em Computations in groups acting on a product of trees: {N}ormal subgroup structures and quaternion lattices}.
\newblock ProQuest LLC, Ann Arbor, MI, 2004.
\newblock Thesis (Dr.sc.math.)--Eidgenoessische Technische Hochschule Zuerich (Switzerland).

\bibitem[RS94]{Rips-Sela}
E.~Rips and Z.~Sela.
\newblock Structure and rigidity in hyperbolic groups. {I}.
\newblock {\em Geom. Funct. Anal.}, 4(3):337--371, 1994.

\bibitem[Run18]{Rungtanapirom}
Nithi Rungtanapirom.
\newblock Quaternionic arithmetic lattices of rank 2 and a fake quadric in characteristic 2.
\newblock {\em Internat. J. Algebra Comput.}, 28(6):1049--1090, 2018.

\bibitem[Sel95]{Sela1995}
Z.~Sela.
\newblock The isomorphism problem for hyperbolic groups. {I}.
\newblock {\em Ann. of Math. (2)}, 141(2):217--283, 1995.

\bibitem[Sel09]{Sela2009}
Z.~Sela.
\newblock Diophantine geometry over groups. {VII}. {T}he elementary theory of a hyperbolic group.
\newblock {\em Proc. Lond. Math. Soc. (3)}, 99(1):217--273, 2009.

\bibitem[Sho54]{Sholander}
Marlow Sholander.
\newblock Medians, lattices, and trees.
\newblock {\em Proc. Amer. Math. Soc.}, 5:808--812, 1954.

\bibitem[SV17]{Stix-Vdovina}
Jakob Stix and Alina Vdovina.
\newblock Simply transitive quaternionic lattices of rank 2 over {$\mathbb{F}_q(t)$} and a non-classical fake quadric.
\newblock {\em Math. Proc. Cambridge Philos. Soc.}, 163(3):453--498, 2017.

\bibitem[Tit77]{Tits}
J.~Tits.
\newblock A ``theorem of {L}ie-{K}olchin'' for trees.
\newblock In {\em Contributions to algebra (collection of papers dedicated to {E}llis {K}olchin)}, pages 377--388. Academic Press, New York-London, 1977.

\bibitem[Wis96]{Wise-thesis}
Daniel~T. Wise.
\newblock {\em Non-positively curved squared complexes: {A}periodic tilings and non-residually finite groups}.
\newblock ProQuest LLC, Ann Arbor, MI, 1996.
\newblock Thesis (Ph.D.)--Princeton University.

\bibitem[Wis07]{Wise-CSC}
Daniel~T. Wise.
\newblock Complete square complexes.
\newblock {\em Comment. Math. Helv.}, 82(4):683--724, 2007.

\bibitem[Zas98]{Zastrow98}
Andeas Zastrow.
\newblock Construction of an infinitely generated group that is not a free product of surface groups and abelian groups, but which acts freely on an $\mathbb{R}$-tree.
\newblock {\em Proc. R. Soc. Edinb. A}, 128(2):433–445, 1998.

\bibitem[Zei16]{Zeidler}
Rudolf Zeidler.
\newblock Coarse median structures and homomorphisms from {K}azhdan groups.
\newblock {\em Geom. Dedicata}, 180:49--68, 2016.

\end{thebibliography}

\bigskip
{\footnotesize
  \noindent
  {\textsc{University of Bristol, School of Mathematics, Bristol, UK}} \par\nopagebreak
  \texttt{penelope.azuelos@bristol.ac.uk}

\end{document}